\documentclass[a4paper,reqno, 11pt]{amsart}  
\usepackage[DIV=12, oneside]{typearea}
\usepackage[utf8]{inputenc}
\usepackage[T1]{fontenc}
\usepackage[english]{babel}
\usepackage[centertags]{amsmath}
\usepackage{amstext,amssymb,amsopn,amsthm}
\usepackage{mathrsfs}
\usepackage{dsfont}
\usepackage{bbm}
\usepackage{thmtools}
\usepackage{graphicx}
\usepackage[titletoc,title]{appendix}
\usepackage{etexcmds}

\usepackage[backgroundcolor=white, bordercolor=blue,
linecolor=blue]{todonotes}
\usepackage[colorlinks=true, linkcolor=black, citecolor=black]{hyperref}
\usepackage{enumitem}
\setlist[enumerate]{itemsep=0mm}

\addto\extrasenglish{}
\addto\extrasenglish{}
\addto\extrasenglish{}

\theoremstyle{plain}
\declaretheorem[title=Theorem, parent=section]{theorem}
\declaretheorem[title=Lemma,sibling=theorem]{lemma}

\declaretheorem[title=Corollary,sibling=theorem]{corollary}

\theoremstyle{definition}
\declaretheorem[title=Definition,sibling=theorem]{definition}
\declaretheorem[title=Remark,sibling=theorem]{remark}
\declaretheorem[title=Remark, numbered=no]{remark*}
\declaretheorem[title=Example, sibling=theorem]{example}
\declaretheorem[title=Assumption, numbered=no]{assumption*}

\numberwithin{equation}{section}

\newcommand{\N}{\mathds{N}}
\newcommand{\R}{\mathds{R}}

\def\hmath$#1${\texorpdfstring{{\rmfamily\textit{#1}}}{#1}}

\newcommand{\cE}{\mathcal{E}}
\newcommand{\cEs}{\cE^{K_s}}
\newcommand{\cEa}{\cE^{K_a}}

\newcommand{\eps}{\varepsilon}

\newcommand{\U}{\widetilde{u}}

\newcommand{\BIGOP}[1]
{
\mathop{\mathchoice%
{\raise-0.22em\hbox{\huge $#1$}}%
{\raise-0.05em\hbox{\Large $#1$}}{\hbox{\large $#1$}}{#1}}}

\def\Xint#1{\mathchoice
   {\XXint\displaystyle\textstyle{#1}}%
   {\XXint\textstyle\scriptstyle{#1}}%
   {\XXint\scriptstyle\scriptscriptstyle{#1}}%
   {\XXint\scriptscriptstyle\scriptscriptstyle{#1}}%
   \!\int}
\def\XXint#1#2#3{{\setbox0=\hbox{$#1{#2#3}{\int}$}
     \vcenter{\hbox{$#2#3$}}\kern-.5\wd0}}

\def\dashint{\Xint-}
\newcommand{\BIGboxplus}{\mathop{\mathchoice%
{\raise-0.35em\hbox{\huge $\boxplus$}}%
{\raise-0.15em\hbox{\Large $\boxplus$}}{\hbox{\large $\boxplus$}}{\boxplus}}}

\DeclareMathOperator{\supp}{supp}

\DeclareMathOperator{\tail}{Tail}
\DeclareMathOperator{\pv}{p.v.}

\renewcommand{\d}{\textnormal{d}}

\begin{document}
\allowdisplaybreaks
 \title{Nonlocal operators related to nonsymmetric forms II: Harnack inequalities}

\author{Moritz Kassmann}
\author{Marvin Weidner}
\address{Fakult\"{a}t f\"{u}r Mathematik\\Universit\"{a}t Bielefeld\\Postfach 100131\\D-33501 Bielefeld}
\email{moritz.kassmann@uni-bielefeld.de}
\urladdr{https://www.uni-bielefeld.de/math/kassmann}

\address{Fakult\"{a}t f\"{u}r Mathematik\\Universit\"{a}t Bielefeld\\Postfach 100131\\D-33501 Bielefeld}
\email{mweidner@math.uni-bielefeld.de}

\keywords{nonlocal operator, energy form, nonsymmetry, 
regularity, Harnack inequality, De Giorgi}

\thanks{Moritz Kassmann and Marvin Weidner gratefully acknowledge financial support by the German Research Foundation (SFB 1283 - 317210226 resp. GRK 2235 - 282638148).}

\subjclass[2010]{47G20, 35B65, 31B05, 60J75, 35K90}

\allowdisplaybreaks

\begin{abstract}
Local boundedness and Harnack inequalities are studied for solutions to parabolic and elliptic integro-differential equations whose governing nonlocal operators are associated with nonsymmetric forms. We present two independent proofs, one being based on the De Giorgi iteration and the other one on the Moser iteration technique. This article is a continuation of \cite{KaWe22}, where H\"older regularity and a weak Harnack inequality are proved in a similar setup. 
\end{abstract}

\maketitle
\section{Introduction}  
The aim of this work is to prove local boundedness estimates and a Harnack inequality for weak solutions to parabolic equations of type
\begin{equation}
\label{PDE}\tag{$\text{PDE}$}
\partial_t u - L u = f~~ \text{in } I_R(t_0) \times B_{2R} \subset \R^{d+1},
\end{equation}
where $B_{2R} \subset \Omega$ is some ball, $I_{R}(t_0) := (t_0 - R^{\alpha}, t_0 + R^{\alpha}) \subset \R$, and $f \in L^{\infty}(I_R(t_0) \times B_{2R})$. \eqref{PDE} is governed by a linear nonlocal operator of the following form 
\begin{equation}
\label{eq:op}
-L u (x) = 2\pv\int_{\R^d} (u(x)-u(y))K(x,y)\d y.
\end{equation}
Such operators are determined by jumping kernels $K : \R^d \times \R^d \to [0,\infty]$ which are allowed to be nonsymmetric. We also investigate solutions to the equation
\begin{equation}
\label{PDEdual}\tag{$\widehat{\text{PDE}}$}
\partial_t u - \widehat{L} u = f~~ \text{in } I_R(t_0) \times B_{2R} \subset \R^{d+1},
\end{equation}
which is driven by the dual operator $\widehat{L}$ associated with $L$.

In this work we prove local boundedness of weak solutions to \eqref{PDE} and \eqref{PDEdual} via an adaptation of the De Giorgi method to nonlocal operators with nonsymmetric jumping kernels. We also provide an alternative proof of local boundedness via the Moser iteration. Finally, combined with the weak Harnack inequality from \cite{KaWe22}, we obtain a full Harnack inequality.

The novelty of our result consists in the lack of symmetry of the underlying operator.
Let us decompose $K = K_s + K_a$ into its symmetric part $K_s$ and its antisymmetric part $K_a$, where
\begin{equation*}
K_s(x,y) = \frac{K(x,y)+K(y,x)}{2}, ~~ K_a(x,y) = \frac{K(x,y)-K(y,x)}{2}, ~~x,y \in \R^d.
\end{equation*}
Note that nonnegativity of $K$ implies
\begin{equation}
\label{eq:KaKs}
\vert K_a(x,y) \vert \le K_s(x,y).
\end{equation}
We can write for the nonsymmetric bilinear form associated with $L$
\begin{align*}
\cE(u,v) := 2\int_{\R^d} \int_{\R^d} (u(x)-u(y))v(x) K(x,y) \d y \d x =: \cEs(u,v) + \cEa(u,v),
\end{align*}
where
\begin{align*}
&\cEs(u,v) = \int_{\R^d} \int_{\R^d} (u(x)-u(y))(v(x)-v(y)) K_s(x,y) \d y \d x,\\
&\cEa(u,v) = \int_{\R^d} \int_{\R^d} (u(x)-u(y))(v(x)+v(y)) K_a(x,y) \d y \d x.
\end{align*}
In order to treat the antisymmetric part of the bilinear form, a refinement of the existing techniques for symmetric operators is required.

We have in mind the following three prototypes of kernels $K$ for $\alpha \in (0,2)$:
\begin{equation}
\label{eq:kernelclassf}
K_1(x,y) = g(x,y) |x-y|^{-d-\alpha},
\end{equation}
where $g : \R^d \times \R^d \to [\lambda,\Lambda] \subset (0,\infty)$ is a suitable nonsymmetric function,
\begin{equation}
\label{eq:kernelclassV}
K_2(x,y) = \vert x-y \vert^{-d-\alpha} + (V(x)-V(y))\mathbbm{1}_{\{\vert x-y \vert \le L\}}(x,y)\vert x-y \vert^{-d-\alpha},
\end{equation}
where $L \in (0,\infty]$ and $V: \R^d \to \R^d$ is a suitable function, and
\begin{align}
\label{eq:kernelclasscone}
K_3(x,y) = |x-y|^{-d-\alpha} \mathbbm{1}_D(x-y) + |x-y|^{-d-\beta} \mathbbm{1}_C(x-y),
\end{align}
where $C \subset \R^d$ is a cone, $D \subset \R^d$ is a double-cone such that $C \cap D = \emptyset$, and $0 < \beta < \frac{\alpha}{2}$.

\subsection{Main results}

Our first main result is the following Harnack inequality for weak solutions to 
\eqref{PDE}. We state and discuss our assumptions in \autoref{sec:prelim}.

\begin{theorem}
\label{thm:apriorifHI}
Assume \eqref{K2}, \eqref{cutoff}, \eqref{eq:kuppershort}, \eqref{Sob}, and \eqref{Poinc} for some $\alpha \in (0,2)$.
\begin{itemize}
\item[(i)] Assume that \eqref{K1} holds true for some $\theta \in [\frac{d}{\alpha},\infty]$. Then there exist $c > 0$ and $0 < c_1 < c_2 < c_3 < c_4 \le 1$ such that for every $0 < R \le 1$ and every nonnegative, weak solution $u$ to \eqref{PDE} in $I_R(t_0) \times B_{2R}$
\begin{align}
\label{eq:apriorifHI}
\begin{split}
&\sup_{(t_0 - c_2R^{\alpha} , t_0 - c_1 R^{\alpha}) \times B_{R/4}} u  \le c \inf_{(t_0 + c_1 R^{\alpha} , t_0 + c_4 R^{\alpha}) \times B_{R/2}} u \\
&\qquad\qquad\qquad\qquad+ c\sup_{(t_0 - c_3 R^{\alpha} , t_0 - c_1 R^{\alpha})}\tail_{K,\alpha}(u,R)+ c R^{\alpha} \Vert f \Vert_{L^{\infty}},
\end{split}
\end{align}
where $B_{2R} \subset \Omega \subset \R^d$.

\item[(ii)] Assume that \eqref{K1glob} holds true for some $\theta \in (\frac{d}{\alpha},\infty]$. Then there exist $c > 0$ and $0 < c_1 < c_2 < c_3 < c_4 \le 1$ such that for every $0 < R \le 1$ and every nonnegative, weak solution $u$ to \eqref{PDEdual} in $I_R(t_0) \times B_{2R}$
\begin{align}
\label{eq:apriorifHIdual}
\begin{split}
&\sup_{(t_0 - c_2R^{\alpha} , t_0 - c_1 R^{\alpha}) \times B_{R/4}} u  \le c \inf_{(t_0 + c_1 R^{\alpha} , t_0 + c_4 R^{\alpha}) \times B_{R/2}} u \\
&\qquad\qquad\qquad\qquad+ c\sup_{(t_0 - c_3 R^{\alpha} , t_0 - c_1 R^{\alpha})}\widehat{\tail}_{K,\alpha}(u,R)+ c R^{\alpha} \Vert f \Vert_{L^{\infty}},
\end{split}
\end{align}
\end{itemize}
\end{theorem}

The aforementioned Harnack inequality for nonnegative weak solutions $u$ to \eqref{PDE} is a direct consequence of a weak Harnack inequality as it was proved in \cite{KaWe22} (see \autoref{thm:wHI}), and an $L^{\infty} - L^{1}$-estimate of the form (see \autoref{thm:LB1} or \autoref{thm:LB1dual})
\begin{equation}
\label{eq:introLB3}
\sup_{(t_0 - (\frac{R}{8})^{\alpha} , t_0) \times B_{\frac{R}{2}}} \hspace*{-2ex} u \le c \left(\dashint_{(t_0 - (\frac{R}{4})^{\alpha} , t_0) \times B_{R}} u  + \sup_{(t_0 - (\frac{R}{4})^{\alpha} , t_0)}\tail_{K,\alpha}(u,R) + R^{\alpha}\Vert f \Vert_{L^{\infty}}\right).
\end{equation}
Therefore large parts of this paper are dedicated to proving \eqref{eq:introLB3}. 
Given $0 < R \le 1$, the nonlocal tail term is defined as follows:
\begin{equation*}
\tail_{K,\alpha}(v,R,x_0) := R^{\alpha}\int_{B_{2R}(x_0) \setminus B_{\frac{R}{2}}(x_0)} \frac{|v(y)|}{|x_0 - y|^{d+\alpha}} \d y + \sup_{x \in B_{\frac{3R}{2}}(x_0)} \int_{B_{2R}(x_0)^c} |v(y)| K(x,y) \d y.
\end{equation*}
For a detailed discussion of nonlocal tail terms, we refer the reader to \autoref{sec:tails}.

\begin{remark}[time-inhomogeneous kernels]
\label{remark:time-inhom}
It is possible to extend \autoref{thm:apriorifHI} to time-inhomogeneous jumping kernels $k : I \times \R^d \times \R^d \to [0,\infty]$ by following a similar approach as in \cite{KaWe22}. For $k_s$ we may assume pointwise comparability with a time-homogeneous jumping kernel satisfying \eqref{cutoff}, \eqref{elower} and \eqref{eq:kuppershort}. In place of the first estimate in \eqref{K1}, we need
\begin{align*}
\left\Vert \int_{B_{2r}} \frac{|k_a(\cdot;\cdot,y)|^2}{J(\cdot,y)} \d y \right\Vert_{L^{\mu,\theta}_{t,x}(I_r \times B_{2r})} \le C
\end{align*}
for a suitable symmetric jumping kernel $J : \R^d \times \R^d \to [0,\infty]$. The parameters $(\mu,\theta)$ have to satisfy the compatibility condition 
\begin{align}
\label{eq:CP}\tag{CP}
\frac{d}{\alpha\theta} + \frac{1}{\mu} < 1.
\end{align}
Then, if suitable time-inhomogeneous analogs to \eqref{K2} and \eqref{UJS}, or \eqref{UJSdual}, hold true, we can prove a Harnack inequality of the form \eqref{eq:apriorifHI} and \eqref{eq:apriorifHIdual}, for nonnegative, weak solutions to the corresponding parabolic equations \eqref{PDE} and \eqref{PDEdual}, respectively. For solutions to \eqref{PDE} we can also allow for equality in \eqref{eq:CP} if $\theta > \frac{d}{\alpha}$. The range of exponents prescribed by \eqref{eq:CP} is in align with the important results from the local theory (see \cite{ArSe67}, \cite{LSU68}).
\end{remark}

\begin{remark}
We observe that there is a positive distance of size $2(1-2^{-\alpha})R^{\alpha}$ between the two time intervals in the estimates \eqref{eq:apriorifHI} and \eqref{eq:apriorifHIdual}. The existence of such time delay in the parabolic Harnack inequality comes from the method of proof we employ (see \cite{Mos64}). For nonlocal equations, as for example the fractional heat equation, it can be neglected (see \cite{BSV17}, \cite{DKSZ20}).
\end{remark}

The second main result of this article concerns the corresponding stationary problems
\begin{align}
\label{ellPDE}\tag{$\text{ell-PDE}$}
-L u &= f ~~\text{ in } B_{2R},\\
\label{ellPDEdual}\tag{$\text{ell-}\widehat{\text{PDE}}$}
-\widehat{L} u &= f ~~\text{ in } B_{2R},
\end{align}
where $f \in L^{\infty}(B_{2R})$. We obtain an elliptic Harnack inequality for weak solutions:

\begin{theorem}
\label{thm:fHI}
Assume \eqref{K2}, \eqref{cutoff} and \eqref{elower} for some $\alpha \in (0,2)$.
\begin{itemize}
\item[(i)] Assume that \eqref{K1}, \eqref{UJS} hold true for some $\theta \in [\frac{d}{\alpha},\infty]$. Then there exists $c > 0$ such that for every $0 < R \le 1$ and every nonnegative, weak solution $u$ to \eqref{ellPDE} in $B_{2R}$, it holds
\begin{align}
\label{eq:fHI}
\sup_{B_{R/4}} u \le c \left(\inf_{B_{R/2}} u + R^{\alpha}\Vert f \Vert_{L^{\infty}}\right),
\end{align}
where $B_{2R} \subset \Omega \subset \R^d$.
\item[(ii)] Assume that \eqref{K1glob}, \eqref{UJSdual} hold true for some $\theta \in (\frac{d}{\alpha},\infty]$. Then there exists $c > 0$ such that for every $0 < R \le 1$ and every nonnegative, weak solution $u$ to \eqref{ellPDEdual} in $B_{2R}$, estimate \eqref{eq:fHI} holds true.
\end{itemize}
\end{theorem}

As in \eqref{eq:fHI}, for elliptic equations we are able to estimate the supremum of $u$ by local quantities, only. To this end, we prove a suitable estimate of the nonlocal tail term (see \autoref{cor:elltailest}).\\
In the parabolic case, the situation is more complicated since we require the tail estimate to be uniform in $t$. The same difficulty occurs in the symmetric case. We comment on possible corresponding extensions of \autoref{thm:apriorifHI} in \autoref{sec:parabolicFHI}.

\begin{remark}
All constants in \autoref{thm:apriorifHI} and \autoref{thm:fHI} depend only on $d,\alpha,\theta$ and the constants in \eqref{K1}, \eqref{K2}, \eqref{cutoff} \eqref{Poinc}, \eqref{Sob}, \eqref{UJS}, \eqref{eq:kuppershort}, \eqref{elower}.
\end{remark}

The contributions of this work can be summarized as follows:
\begin{itemize}
\item[(i)] The main accomplishment is the extension of elliptic and parabolic regularity results -- including full Harnack inequalities -- for nonlocal problems to  operators with \emph{nonsymmetric} jumping kernels. In light of example \eqref{eq:kernelclassV}, the operators under consideration include nonlocal counterparts of second order differential operators in divergence form with a drift term
\begin{align*}
-\mathcal{L}u = -\partial_i (a_{i,j} \partial_j u) + b_i \partial_i u, \qquad \text{resp. } -\widehat{\mathcal{L}}u = - \partial_i (a_{i,j} \partial_j u + b_i u).
\end{align*}
Our results are in align with the corresponding theory for local operators (see \cite{Sta65}, \cite{ArSe67}, \cite{LSU68}, \cite{GiTr01}).

\item[(ii)] As nonsymmetric kernels require a careful treatment, several parts of the energy methods for nonlocal operators are refined in this work. For instance, we give a new proof of local boundedness using the Moser iteration for positive exponents (see \autoref{sec:CaccMoser}).\\
Moreover, as illustrated in example \eqref{eq:kernelclasscone}, nonsymmetric jumping kernels might naturally involve terms of lower order, causing a difference between the growth behavior at zero and infinity. We introduce tail terms which take into account this phenomenon (see \autoref{sec:tails}).

\item[(iii)] Technical issues of minor importance in other works are clarified, e.g., the treatment of Steklov averages (see \autoref{sec:app}).

\end{itemize}


\subsection{Related literature}

The study of Harnack inequalities for symmetric nonlocal operators has become an active field of research in the past 20 years. It has been observed that a classical elliptic Harnack inequality of the form 
\begin{equation}
\label{eq:introSHI}
\sup_{B_r} u \le c \inf_{B_r} u
\end{equation}
fails even for harmonic functions $u$ with respect to the fractional Laplacian $(-\Delta)^{\alpha/2}$ in $B_{2r}$ if one merely assumes $u$ to be nonnegative in the solution domain $B_{2r}$ (see \cite{Kas07}). Indeed, due to the nonlocality it is necessary either to assume $u$ to be globally nonnegative -- as in \cite{Rie38} and in this article -- or to add the nonlocal tail of $u_-$ to the right hand side of \eqref{eq:introSHI}. Such estimate was proposed in \cite{Kas11}. We refer to both estimates as a \emph{Harnack inequality} in the context of this article.

A lot of research activity has centered around the challenge to establish a Harnack inequality for a larger class of nonlocal operators.
First, we comment on corresponding elliptic regularity results for symmetric nonlocal operators related to energy forms. A Harnack inequality and H\"older estimates were proved in \cite{DKP14}, \cite{DKP16} for operators with a jumping kernel that is pointwise comparable to the kernel of the fractional $p$-Laplacian by a nonlocal De Giorgi type iteration. This method was refined in \cite{Coz17} to allow for more general nonlinearities. \cite{Sch20} considers a class of linear integro-differential operators governed by jumping kernels satisfying an averaged integral bound instead of a pointwise lower bound.

However, it is well known that for the deduction of interior H\"older regularity estimates a weak Harnack inequality (see \autoref{thm:wHI}) is sufficient. Such inequalities hold true for a much larger class of operators. In fact, only comparability of the energy forms to the $H^{\alpha/2}$-seminorm on small scales and a suitable upper bound for the probability of large jumps are required (see \cite{DyKa20}). That is why operators with singular jumping measures that may be anisotropic (see \cite{ChKa20}, \cite{CKW19}) also satisfy H\"older regularity estimates. However, the Harnack inequality may fail for singular operators as was  already observed in \cite{BoSz05}. Hence it is an exciting (and still open) question to find equivalent conditions on the jumping kernel for a (weak) elliptic Harnack inequality to hold true. For $\alpha$-stable translation invariant operators conditions on the jumping kernel are established in \cite{BoSz05} that are equivalent to a Harnack inequality.

Second, we comment on parabolic Harnack inequalities of the form
\begin{align}
\label{eq:phi-localterms}
\sup_{I^{\ominus}_r \times B_r} u \le c \inf_{I^{\oplus}_r \times B_r} u
\end{align}
for globally nonnegative solutions $u$ to \eqref{PDE}. Note that such results imply corresponding estimates for weak solutions to the stationary equation \eqref{ellPDE}. 
So far, parabolic Harnack inequalities have not been obtained via purely analytic methods, not even in the symmetric case. A major challenge in the parabolic case seems to be the correct treatment of the time-dependence in the nonlocal tail terms. For a discussion of this issue we refer the reader to the discussion in \autoref{sec:prelim} and \autoref{sec:parabolicFHI}.\\
Parabolic H\"older estimates and local boundedness have been obtained via an adaptation of the nonlocal De Giorgi method in \cite{Str19a}, \cite{Kim19}, \cite{Kim20}, \cite{DZZ21}. A proof of H\"older estimates based on Moser's technique can be found in \cite{FeKa13}.\\
Using the corresponding Hunt process and its heat kernel, parabolic Harnack inequalities of the form \eqref{eq:phi-localterms} were first proved for symmetric Dirichlet forms with jumping measures pointwise comparable to the $\alpha$-stable kernel in \cite{BaLe02}, \cite{ChKu03}. The authors also obtain two-sided heat kernel bounds. Numerous articles have analyzed the exact relationship between parabolic and elliptic Harnack inequalities, heat kernel bounds, and H\"older regularity estimates for nonlocal operators in connection to the geometry of the underlying metric measure space. Such program was carried out in a series of papers (\cite{CKW19b}, \cite{CKW20}, \cite{GHKS14}, \cite{GHKS15}, \cite{GHH18}). On $\R^d$ it turns out that \eqref{eq:phi-localterms} is equivalent to a Poincar\'e inequality (see \eqref{Poinc}), a pointwise upper bound of the jumping kernel, and \eqref{UJS}.\\

In contrast to the symmetric case, for nonlocal operators associated with nonsymmetric forms, pointwise estimates  have not yet been studied systematically. Some results have been obtained making use of a sector-type-condition. Well-posedness of the Dirichlet problem is proved in \cite{FKV15}. In the present article and in  \cite{KaWe22} we provide Harnack inequalities and interior H\"older regularity estimates for nonlocal operators that contain a nonlocal drift term of lower order. These results can be regarded as nonlocal counterparts of the famous regularity results for local equations by Aronson-Serrin (see \cite{ArSe67}) and Ladyzhenskaya-Solonnikov-Ural'tceva (see \cite{LSU68}) in the linear case.
H\"older estimates for kinetic integro-differential equations including certain nonlocal operators with nonsymmetric jumping kernels are established in \cite{ImSi20} using an adaptation of the De Giorgi iteration. The class of nonsymmetric kernels in their work does not contain the class of kernels in our work, and vice versa.\\
Note that, as an application of the regularity estimates in \cite{KaWe22}, it is possible to establish Markov chain approximation results for diffusion processes with drift terms, but also for certain nonsymmetric jump processes (see \cite{Wei22}).
In light of \cite{CKW20} and \cite{GHH18}, we consider it an interesting problem  to establish heat kernel estimates for nonlocal operators associated with nonsymmetric forms, and to investigate their stability on general doubling metric measure spaces, as well as their connection to Harnack inequalities.

\subsection{Outline}
This article is structured as follows: In \autoref{sec:prelim} we state and discuss our assumptions and the notion of a weak solution to \eqref{PDE} and \eqref{PDEdual}. A Caccioppoli-type estimate for nonsymmetric forms and an a priori $L^{\infty}-L^2$-estimate involving the nonlocal tail is proved in \autoref{sec:Cacc} using a nonsymmetric version of the De Giorgi iteration. An analogous result is established in \autoref{sec:CaccMoser} using a nonlocal adaptation of the Moser iteration technique for large positive exponents. Note that \autoref{sec:Cacc} and \autoref{sec:CaccMoser} are fully independent of one another.
In \autoref{sec:LTE} we establish an upper bound for the nonlocal tails of supersolution to \eqref{PDE} and \eqref{PDEdual}. Our two main results \autoref{thm:apriorifHI} and \autoref{thm:fHI} are proved in \autoref{sec:FHI}.

\section{Preliminaries}
\label{sec:prelim}

In this section we state and discuss the assumptions in our main results (see \autoref{sec:mainass}). Moreover, we provide the notion of a (super/sub)solution to \eqref{PDE} and \eqref{PDEdual}, as well as the corresponding stationary equations \eqref{ellPDE} and \eqref{ellPDEdual} (see \autoref{sec:weaksol}). Another goal of this section is to introduce nonlocal tail terms which suit the class of nonsymmetric operators under consideration and are designed in such a way that they are compatible with the iteration techniques carried out in the remainder of this article (see \autoref{sec:tails}).

We introduce the following notation: Given a set $M \subset \R^d \times \R^d$, we write
\begin{equation*}
\cE_M(u,v) := \int \int_{M} (u(x)-u(y))v(x) K(x,y) \d x \d y.
\end{equation*}
Analogously, we define $\cEs_M, \cEa_M$. If $M := B_r \times B_r$ for a ball $B_r \subset \R^d$, we write $\cE_{B_r} = \cE_{B_r \times B_r}$.

\subsection{Discussion of main assumptions}
\label{sec:mainass}

In this section, we list and discuss the assumptions which are imposed on the jumping kernels $K$ in the course of this article. Except for \eqref{UJS}, all other assumptions have already been discussed in detail in \cite{KaWe22}.

First, we assume throughout this article that $K_s$ satisfies the L\'evy-integrability condition
\begin{equation}
\label{eq:Levy}
\left(x \mapsto \int_{\R^d} \left(\vert x-y \vert^2 \wedge 1 \right) K_s(x,y) \d y  \right) \in L^1_{loc}(\R^d).
\end{equation}

In the following, let $\Omega \subset \R^d$ be an open set. Let us now fix $\alpha \in (0,2)$ and $\theta \in [\frac{d}{\alpha},\infty]$. 
The first two assumptions were introduced and discussed in \cite{KaWe22}.

\begin{assumption*}[\textbf{K1}]
\label{ass:K1}
Let $J : \R^d \times \R^d \to [0,\infty]$ be a symmetric jumping kernel satisfying \eqref{cutoff} and let $\theta \in [\frac{d}{\alpha},\infty]$.
\begin{itemize}
\item $K$ satisfies \eqref{K1} if there is $C > 0$ such that for every ball $B_{2r} \subset \Omega$ with $r \le 1$:
\begin{equation}
\label{K1}\tag{$\text{K1}_{loc}$}
\left\Vert \int_{B_{2r}} \frac{\vert K_a(\cdot,y)\vert^2}{J(\cdot,y)} \d y  \right\Vert_{L^{\theta}(B_{2r})} \le C, \qquad \cE_{B_{2r}}^{J}(v,v) \le C \cE_{B_{2r}}^{K_s}(v,v), ~~ \forall v \in L^2(B_{2r}).
\end{equation}
\item $K$ satisfies \eqref{K1glob} if there is $C > 0$ such that for every ball $B_{2r} \subset \Omega$ with $r \le 1$:
\begin{equation}
\label{K1glob}\tag{$\text{K1}_{glob}$}
\left\Vert \int_{\R^d} \frac{\vert K_a(\cdot,y)\vert^2}{J(\cdot,y)} \d y  \right\Vert_{L^{\theta}(\R^d)} \le C, \qquad \cE_{B_{2r}}^{J}(v,v) \le C \cE_{B_{2r}}^{K_s}(v,v), ~~ \forall v \in L^2(B_{2r}).
\end{equation}
\end{itemize}
\end{assumption*}

\begin{assumption*}[\textbf{K2}]
\label{ass:K2}
There exist $C > 0$, $D < 1$ and a symmetric jumping kernel $j$ such that for every ball $B_{2r} \subset \Omega$ with $r \le 1$ and every $v \in L^2(B_{2r})$ with $\cE_{B_{2r}}^{K_s}(v,v) < \infty$:
\begin{align*}
\label{K2}\tag{$\text{K2}$}
K(x,y) \ge (1-D)j(x,y), ~~ \forall x,y \in B_{2r}, \qquad \cE_{B_{2r}}^{K_s}(v,v) &\le C \cE_{B_{2r}}^{j}(v,v).
\end{align*}
\end{assumption*}

\begin{remark}
\begin{itemize}
\item[(i)] \eqref{K1} ensures that the quantities in \eqref{eq:supercal} and \eqref{PDE} are well-defined (see \autoref{lemma:welldef}) and simultaneously determines $\cEa$ to be a term of lower order. It gives rise to a nonlocal drift, analogous to $(b, \nabla u)$, where $b \in L^{2\theta}(\R^d)$ with $\theta \in [\frac{d}{2},\infty]$.
\item[(ii)] \eqref{K2} is only needed in the proof of the weak Harnack inequality (see \autoref{thm:wHI}). It ensures that the symmetric kernel $K_s - |K_a|$ is locally coercive with respect to $\cE^{K_s}$. 
\item[(iii)] For a detailed discussion of \eqref{K1} and \eqref{K2} including their redundancy, we refer the reader to \cite{KaWe22}. 
\eqref{K1} and \eqref{K2} are verified for the examples $K_1$, $K_2$ and $K_3$ from above in Section 8 in \cite{KaWe22}.
\item[(iv)] In the simplest case, \eqref{K1} (and \eqref{K1glob}) and \eqref{K2} hold true with $J = j = K_s$. However, allowing for general symmetric kernels $J,j$ significantly increases the class of admissible operators.
\end{itemize}
\end{remark}

The following two assumptions on $K$ only depend on the symmetric part. They are standard in the regularity for nonlocal operators associated with symmetric forms.

\begin{assumption*}[\textbf{Cutoff}]
\label{ass:Cutoff}
There is $c > 0$ such that for every $0 < \rho \le r \le 1$,  $z \in \Omega$ such that $B_{r+\rho}(z) \subset \Omega$ there is a radially decreasing function $\tau = \tau_{z,r,\rho}$ centered at $z \in \R^d$ with $\supp(\tau) \subset \overline{B_{r+\rho}(z)}$, $0 \le \tau \le 1$, $\tau \equiv 1$ on $B_r(z)$ and $\vert \nabla \tau \vert \le \frac{3}{2}\rho^{-1}$
\begin{equation}
\label{cutoff}\tag{$\text{Cutoff}$}
\sup_{x \in B_{r+\rho}(z)} \Gamma^{K_s}(\tau,\tau)(x) \le c \rho^{-\alpha},
\end{equation}
where $\Gamma^{K_s}(\tau,\tau)(x) := \int_{\R^d} (\tau(x)-\tau(y))^2 K_s(x,y) \d y$ is the carr\'e du champ associated with $\cEs$.
\end{assumption*}

\begin{assumption*}[$\mathbf{\cE_{\ge}}$]
\label{ass:elower}
There exists $c > 0$ such that for every ball $B_{2r} \subset \Omega$ and every $v \in L^2(B_{2r})$:
\begin{align}
\label{elower}\tag{$\cE_{\ge}$}
\cE^{K_s}_{B_{2r}}(v,v) \ge c [u]^2_{H^{\alpha/2}(B_{2r})}.
\end{align}
\end{assumption*}

\begin{remark}
\begin{itemize}
\item[(i)] A sufficient condition for \eqref{cutoff} to hold true for every $\tau_{z,r,\rho}$ is (see \cite{KaWe22}):
There is $c > 0$ such that for every $0 < \zeta \le \rho \le r \le 1$, $z\in \R^d$ with $B_{r+\rho}(z) \subset \Omega$:
\begin{equation}
\label{eq:suffcutoff}
\sup_{x \in B_{r+\rho}(z)} \left( \int_{\R^d \setminus B_{\zeta}(x)} K_s(x,y) \d y \right) \le c \zeta^{-\alpha}.
\end{equation}

\item[(ii)] \eqref{elower} is a classical coercivity condition on $K_s$. It is significantly weaker than a pointwise lower bound of the form $K_s(x,y) \ge c |x-y|^{-d-\alpha}$ since it allows for non fully supported kernels such as $K_3$ (see \eqref{eq:kernelclasscone}).
\item[(iii)] Under \eqref{elower}, we have the following Poincar\'e -- and Sobolev inequality:\\
There is $c > 0$ such that for every ball $B_{r+\rho} \subset \Omega$ with $0 < \rho \le r \le 1$ and $v \in L^2(B_{r+\rho})$
\begin{align}
\label{Sob}\tag{$\text{Sob}$}
\Vert v^2 \Vert_{L^{\frac{d}{d-\alpha}}(B_r)} &\le c\cEs_{B_{r+\rho}}(v,v) + c \rho^{-\alpha}\Vert v^2\Vert_{L^{1}(B_{r+\rho})},\\
\label{Poinc}\tag{$\text{Poinc}$}
\int_{B_r} \left(v(x) - [v]_{B_r}\right)^2 \d x &\le c r^{\alpha} \cEs_{B_r}(v,v),
\end{align}
where $[v]_{B_r} = \dashint_{B_r} v(x) \d x$.
\eqref{Poinc} is not explicitly needed in any of the proofs of this article. Nevertheless it is required for \autoref{thm:wHI} to hold and therefore appears in the assumptions of our main result \autoref{thm:apriorifHI}.
\end{itemize}
\end{remark}

The following assumption did not appear in \cite{KaWe22} and is designed to estimate  nonlocal tails of supersolutions to \eqref{PDE} from above. It is required for the proof of the Harnack inequality.

\begin{assumption*}[\textbf{UJS}]
\label{ass:UJS}
\begin{itemize}
\item $K$ satisfies \eqref{UJS} if there exists $c > 0$ such that for every $x,y \in \R^d$ and every $r \le \left(\frac{1}{4} \wedge \frac{\vert x-y \vert}{4}\right)$ with $B_r(x) \subset \Omega$ it holds
\begin{align}
\label{UJS}\tag{$\text{UJS}$}
K(x,y) &\le c \dashint_{B_r(x)} K(z,y) \d z.
\end{align}
\item $K$ satisfies \eqref{UJSdual} if there exists $c > 0$ such that for every $x,y \in \R^d$ and every $r \le \left(\frac{1}{4} \wedge \frac{\vert x-y \vert}{4}\right)$ with $B_r(x) \subset \Omega$ it holds
\begin{align}
\label{UJSdual}\tag{$\widehat{\text{UJS}}$}
K(y,x) &\le c \dashint_{B_r(x)} K(y,z) \d z.
\end{align}
\end{itemize}
\end{assumption*}

\begin{remark}
\begin{itemize}  
\item[(i)] If $K$ satisfies both conditions, \eqref{UJSdual}  and \eqref{UJS}, then $K_s$ satisfies  \eqref{UJS}.
\item[(ii)] Also for symmetric kernels the conditions \eqref{cutoff}, \eqref{Poinc}, \eqref{Sob} are known to be insufficient for a Harnack inequality to hold (see \cite{BoSz05}). 
\item[(iii)] Analogs to \eqref{UJS} for symmetric jumping kernels appeared in \cite{Sch20} and \cite{CKW20}. A pointwise version of \eqref{UJS} was considered in \cite{BaKa05}.
\end{itemize}
\end{remark}

\begin{remark}
\begin{itemize}
\item[(i)] \eqref{UJS} clearly holds if $K(x,y)$ is pointwise comparable to $|x-y|^{-d-\alpha}$ for every $x,y \in \R^d$. However, \eqref{UJS} neither implies nor is implied by \eqref{elower}.
\item[(ii)] Assume a global version of \eqref{K2}, namely
\begin{equation}
\label{eq:K2glob}
\vert K_a(x,y)\vert \le D K_s(x,y),~~\forall x \in \Omega, y \in \R^d.
\end{equation}
Then, $(1-D)K_s \le K \le 2K_s$ and therefore \eqref{UJS} is equivalent to
\begin{equation*}
K_s(x,y) \le c \dashint_{B_r(x)} K_s(z,y) \d z
\end{equation*}
for $x,y \in \R^d$ with $r \le \left(\frac{1}{4} \wedge \frac{\vert x-y \vert}{4}\right)$, $B_r(x) \subset \Omega$, i.e., it remains to verify \eqref{UJS} for $K_s$.
\item[(iii)] In \cite{Sch20} it was proved that kernels of the form
\begin{equation*}
K_s(x,y) = \mathbbm{1}_S(x-y)\vert x-y \vert^{-d-\alpha}
\end{equation*} 
satisfy \eqref{UJS} if $S = -S$ and there exists $c > 0$ such that for every $x \in S$, $r \le \left(\frac{\vert x \vert}{4} \wedge \frac{1}{4} \right)$ it holds that $\vert B_r(x)\vert \le c \vert B_r(x) \cap S\vert$.
\end{itemize}
\end{remark}

We provide sufficient conditions for \eqref{UJS} to hold true for the examples $K_1$, $K_2$, $K_3$ in \eqref{eq:kernelclassf}, \eqref{eq:kernelclassV} and \eqref{eq:kernelclasscone}. 

\begin{example}
\begin{itemize}
\item[(i)] Let $K_1(x,y) = g(x,y)|x-y|^{-d-\alpha}$ be as in \eqref{eq:kernelclassf}. It was shown in \cite{KaWe22} that \eqref{eq:K2glob} holds true for $K$ with $D = \frac{\Lambda - \lambda}{\Lambda + \lambda} < 1$. As
\begin{equation*}
2K_s(x,y) = (g(x,y)+g(y,x)) |x-y|^{-d-\alpha},
\end{equation*}
it follows that \eqref{UJS} holds true for $K$.

\item[(ii)] Let $K_2$ be as in \eqref{eq:kernelclassV}. Then, the antisymmetric part of $K_2$ is given by
\begin{equation*}
K_a(x,y) = (V(x)-V(y))\mathbbm{1}_{\{\vert x-y \vert \le L\}}(x,y)\vert x-y \vert^{-d-\alpha} \le K_s(x,y) = \vert x-y \vert^{-d-\alpha} .
\end{equation*}
Therefore, \eqref{UJS} holds true if there exists $c > 0$ such that for every $x,y \in \R^d$ and $r \le \left(\frac{\vert x-y  \vert}{4} \wedge \frac{1}{4} \right)$ with $B_r(x) \subset \Omega$:
\begin{equation}
\label{eq:sufftailV}
1  + (V(x)-V(y))\mathbbm{1}_{\{\vert x-y \vert \le L\}} \le c \dashint_{B_r(x)} 1 + (V(z)-V(y))\mathbbm{1}_{\{\vert z-y \vert \le L\}} \d z.
\end{equation}

\item[(iii)] We claim that \eqref{UJS} holds true for $K_3$. Let us prove the following more general statement: Let $S \subset \R^d$ with $0 \in S$ and $c > 0$ such that for every $x \in S$ and $r < \frac{1}{4}$ it holds $\frac{|S \cap B_r(x)|}{r^d} \ge c$. Then, $K(x,y) = \mathbbm{1}_S (x-y) |x-y|^{-d-\alpha}$ satisfies \eqref{UJS} and \eqref{UJSdual}.\\
In fact it suffices to prove that
\begin{align}
\label{eq:coneUJShelp1}
\mathbbm{1}_S(x-y) \le c \dashint_{B_r(x)} \mathbbm{1}_S(z-y) \d z
\end{align}
in order to deduce \eqref{UJS}. Note that \eqref{UJSdual} follows by consideration of $-S$. We compute
\begin{align*}
\mathbbm{1}_S(x-y) \le c \frac{|B_r(x-y) \cap S|}{r^d} = c \frac{|B_r(x) \cap (y+S)|}{r^d} = c\dashint_{B_r(x)} \mathbbm{1}_{S}(z-y) \d z.
\end{align*}
\end{itemize}
\end{example}

Finally, we introduce the assumption of an upper bound of the jumping kernel which will be used only to prove an $L^{\infty}-L^2+\tail$-estimate (see \autoref{thm:LB1}) and is not required for the proof of the main theorems. However it follows from \eqref{UJS} and \eqref{cutoff}.

\begin{assumption*}[$K_{loc}^{\le}$]
There exists $c > 0$ such that for every ball $B_{2r} \subset \Omega$ with $r \le 1$ and every  $x,y \in B_{2r}$:
\begin{align}
\label{eq:kuppershort}\tag{$K^{\le}_{loc}$}
K(x,y) \le c |x-y|^{-d-\alpha}.
\end{align}
\end{assumption*}

\begin{remark}
Note that \eqref{eq:kuppershort} follows from \eqref{UJS} and \eqref{cutoff}. Indeed, for any $x,y \in \R^d$ with $|x-y| \le 4$ and $r = \frac{|x-y|}{16} \le \left(\frac{1}{4} \wedge \frac{|x-y|}{4}\right)$ it holds $B_r(x) \subset B_r(y)^c$ and therefore
\begin{align*}
K(x,y) \le c_1 \dashint_{B_r(x)} K(z,y) \d z \le c_2 r^{-d} \int_{B_r(y)^c} K(z,y) \d z \le c_3 r^{-d-\alpha} \le c_4 |x-y|^{-d-\alpha}
\end{align*}
for some constants $c_1, c_2, c_3, c_4 > 0$.
\end{remark}

\subsection{Weak solution concept}
\label{sec:weaksol}

We introduce the following function spaces for $\Omega \subset \R^d$
\begin{align*}
V(\Omega|\R^d) &= \left\lbrace v : \R^d \to \R : v \mid_{\Omega} \in L^2(\Omega) : (v(x)-v(y))K_s^{1/2}(x,y) \in L^2(\Omega \times \R^d)\right\rbrace,\\
H_{\Omega}(\R^d) &= \left\lbrace v \in V(\R^d|\R^d) : v = 0 \text{ on } \R^d \setminus \Omega \right\rbrace
\end{align*}
equipped with
\begin{align*}
\Vert v \Vert_{V(\Omega|\R^d)}^2 &= \Vert v \Vert_{L^2(\Omega)}^2 + \int_{\Omega} \int_{\R^d}(v(x)-v(y))^2K_s(x,y) \d y \d x,\\
\Vert v \Vert_{H_{\Omega}(\R^d)}^2 &= \Vert v \Vert_{L^2(\R^d)}^2 + \cEs(v,v).
\end{align*}
We emphasize that both spaces are completely determined by the symmetric part of the jumping kernel $K_s$.
Moreover, for $\alpha \in (0,2)$, we define $V^{\alpha}(\Omega|\R^d)$ and $H_{\Omega}^{\alpha}(\R^d)$ as the corresponding function spaces associated with $K_s(x,y) = |x-y|^{-d-\alpha}$.

We are ready to define the notion of a weak solution to \eqref{PDE} and \eqref{PDEdual}.

\pagebreak[3]
\begin{definition}
\label{def:parsol}
Let $\Omega \subset \R^d$ be a bounded domain, $I \subset \R$ a finite interval and $f \in L^{\infty}(I \times \Omega)$. 
\begin{itemize}
\item[(i)] We say that $u \in L^2_{loc}(I;V(\Omega|\R^d))$ is a weak supersolution $u$ to \eqref{PDE} in $I \times \Omega$ if the weak $L^2(\Omega)$-derivative $\partial_t u$ exists, $\partial_t u \in L^1_{loc}(I;L^2(\Omega))$, 
\begin{equation}
\label{eq:supercal}
(\partial_t u(t),\phi) + \cE(u(t),\phi) \le (f(t),\phi), \qquad \forall t \in I,~ \forall \phi \in H_{\Omega}(\R^d) \text{ with } \phi \le 0.
\end{equation}
$u$ is called a weak subsolution if \eqref{eq:supercal} holds true for every $\phi \ge 0$. $u$ is called a weak solution, if it is a supersolution and a subsolution.

\item[(ii)]  We say that $u \in L^2_{loc}(I;V(\Omega|\R^d)\cap L^{2\theta'}(\R^d))$ is a weak supersolution $u$ to \eqref{PDEdual} in $I \times \Omega$ if the weak $L^2(\Omega)$-derivative $\partial_t u$ satisfies the same properties as before and
\begin{equation*}
(\partial_t u(t),\phi) + \widehat{\cE}(u(t),\phi) \le (f(t),\phi), \qquad \forall t \in I,~ \forall \phi \in H_{\Omega}(\R^d) \text{ with } \phi \le 0.
\end{equation*}
Weak (sub)-solutions to \eqref{PDEdual} are defined in analogy with (i).
\end{itemize}
\end{definition}

Next, we introduce the solution concept for stationary equations.
\begin{definition}
\label{def:ellsol}
Let $\Omega \subset \R^d$ be a bounded domain and $f \in L^{\infty}(\Omega)$
\begin{itemize}
\item[(i)] We say that $u \in V(\Omega|\R^d)$ is a weak supersolution $u$ to \eqref{ellPDE} in $\Omega$ if
\begin{equation}
\label{eq:superharm}
\cE(u,\phi) \le (f,\phi), \qquad \forall \phi \in H_{\Omega}(\R^d) \text{ with }  \phi \le 0
\end{equation}
We call $u$ a weak subsolution if \eqref{eq:superharm} holds true for every $\phi \ge 0$. $u$ is called a weak solution if $u$ is a supersolution and a subsolution.
\item[(ii)] We say that $u \in V(\Omega|\R^d) \cap L^{2\theta'}(\R^d)$ is a weak supersolution $u$ to \eqref{ellPDEdual} in $\Omega$ if
\begin{equation*}
\widehat{\cE}(u,\phi) \le (f,\phi), \qquad \forall \phi \in H_{\Omega}(\R^d) \text{ with } \phi \le 0.
\end{equation*}
(Sub)solutions to \eqref{ellPDEdual} are defined in analogy with (i).
\end{itemize}
\end{definition}

Let us point out that the solution concept also makes sense under much weaker assumptions on $u$ without any change in the proofs being needed (see \cite{FeKa13}). In particular, one can drop the condition that the weak time derivative $\partial_t u$ exists.

We will only consider solutions on special time-space cylinders $I_R(t_0) \times B_{2R}$, where $B_{2R} \subset \Omega$ is a ball, $I_R(t_0) = (t_0 - R^{\alpha},t_0 + R^{\alpha})$, $0 < R \le 1$, $t_0 \in \R$. Moreover:
\begin{equation*}
I^{\ominus}_R(t_0) := (t_0 - R^{\alpha}, t_0), ~~ I^{\oplus}_R(t_0) := (t_0, t_0 + R^{\alpha}).
\end{equation*}

Recall the following lemma, which was proved in \cite{KaWe22}. It ensures that the expressions in \autoref{def:parsol} and \autoref{def:ellsol} are well-defined. 

\begin{lemma}[see Lemma 2.2 in \cite{KaWe22}]
\label{lemma:welldef}
Let $0 < \rho \le r \le 1$, $B_{2r} \subset \Omega$.
\vspace{-0.2cm}
\begin{itemize}
\item[(i)] Assume that one of the following is true:
\begin{itemize}
\item \eqref{K1} holds true with $\theta = \infty$,
\item \eqref{K1} holds with $\theta \in [\frac{d}{\alpha}, \infty)$ and \eqref{Sob}   holds true.
\end{itemize} 
Then $\cE(u,\phi)$ is well-defined for $u \in V(B_{r+\rho}|\R^d)$, $\phi \in H_{B_{r+\rho}}(\R^d)$.
\item[(ii)] Assume that \eqref{K1glob} holds true with $\theta \in [\frac{d}{\alpha}, \infty]$.\\
Then $\widehat{\cE}(u,\phi)$ is well-defined for $u \in V(B_{r+\frac{\rho}{2}}|\R^d) \cap L^{2\theta'}(\R^d)$ and $\phi \in H_{B_{r+\frac{\rho}{2}}}(\R^d)$.
\end{itemize}
\end{lemma}

The following lemma is of central importance in the proofs of the Caccioppoli estimates for nonsymmetric nonlocal operators. Note that the proof in the special case $\theta = \infty$ is trivial. 

\begin{lemma}[see Lemma 2.4 in \cite{KaWe22}]
\label{lemma:K1-lemma}
\begin{itemize}
\item[(i)] Assume that \eqref{K1} holds true for some $\theta \in [\frac{d}{\alpha},\infty]$. Moreover, assume \eqref{Sob} if $\theta < \infty$. Then, there exists $c_1 > 0$ such that for every $\delta > 0$ there is $C(\delta) > 0$ such that for every $v \in L^2(B_{r+\rho})$ with $\supp(v) \subset B_{r + \frac{\rho}{2}}$, and every ball $B_{2r} \subset \Omega$ with $0 < \rho \le r \le 1$ it holds
\begin{align}
\label{eq:K1consequence}
\int_{B_{r+\rho}} v^2(x) \left(\int_{B_{r+\rho}} \frac{|K_a(x,y)|^2}{J(x,y)} \d y\right) \d x \le \delta \cE^{K_s}_{B_{r+\rho}}(v,v) + c_1(C(\delta) + \delta \rho^{-\alpha}) \Vert v^2 \Vert_{L^1(B_{r+\rho})}.
\end{align}
Moreover, if $\theta \in (\frac{d}{\alpha},\infty]$, the constant $C(\delta)$ has the following form:
\begin{align}
\label{eq:quantifiedK1consequence}
C(\delta) = \begin{cases}
\Vert W \Vert_{L^{\infty}(B_{r+\rho})}&, ~~ \theta = \infty,\\
\delta^{\frac{d}{d-\theta \alpha}} \Vert W \Vert_{L^{\theta}(B_{r+\rho})}^{\frac{\theta \alpha}{\theta\alpha - d}}&, ~~ \theta \in (\frac{d}{\alpha},\infty),
\end{cases} \quad
\text{ where } W(x) := \int_{B_{r+\rho}} \frac{|K_a(x,y)|^2}{J(x,y)} \d y.
\end{align}

\item[(ii)] Assume that \eqref{K1glob} holds true for some $\theta \in [\frac{d}{\alpha},\infty]$. Moreover, assume \eqref{Sob} if $\theta < \infty$.
then \eqref{eq:K1consequence} and \eqref{eq:quantifiedK1consequence} hold true with $\left(\int_{\R^d} \frac{|K_a(x,y)|^2}{J(x,y)} \d y\right)$ instead of $\left(\int_{B_{r+\rho}} \frac{|K_a(x,y)|^2}{J(x,y)} \d y\right)$.
\end{itemize}
\end{lemma}

\subsection{Nonlocal tail terms}
\label{sec:tails}

Due to the nonlocality of the problems under consideration, certain nonlocal tail terms naturally enter the picture. For references concerning the treatment of tail terms in the study of symmetric nonlocal operators, we refer the reader to \cite{DKP14}, \cite{DKP16}, and \cite{CKW20}. It is crucial for our analysis to make sure that the respective tail terms are finite for any weak solution under reasonable assumptions on $K$ and that the tail terms are compatible with the iteration techniques carried out in the remainder of this article.

Given any ball $B_{2r}(x_0) \subset \Omega$, a function $v \in V(B_{2r}(x_0)|\R^d)$ and $0 < r_1 < r_2 \le 2r$ we define
\begin{align*}
\tail_K(v,r_1,r_2,x_0) &:= \sup_{x \in B_{r_1}(x_0)} \int_{B_{r_2}(x_0)^{c}} \vert v(y) \vert K(x,y) \d y,\\
\widehat{\tail}_K(v,r_1,r_2,x_0) &:= \sup_{x \in B_{r_1}(x_0)} \int_{B_{r_2}(x_0)^{c}} \vert v(y) \vert K(y,x) \d y.
\end{align*}

\begin{remark}
\begin{itemize}
\item[(i)] For $0 < \rho_1 \le r_1$ and $0 < \rho_2 \le r_2$: $\tail_K(v,\rho_1,r_2) \le \tail_K(v,r_1,\rho_2)$.
\item[(ii)] Note that $\tail_K$ has been introduced in \cite{Sch20} for symmetric kernels.
\end{itemize}
\end{remark}

We would like to point out that $\tail_K$ will naturally appear in the proofs of the Caccioppoli estimates in \autoref{sec:Cacc} and \autoref{sec:CaccMoser}. However, it is not suitable for De Giorgi-type and Moser-type iteration arguments. Therefore, we introduce another nonlocal tail term defined as follows:
\begin{align*}
\tail_{K,\alpha}(u,R,x_0) := R^{\alpha}\int_{B_{2R}(x_0) \setminus B_{\frac{R}{2}}(x_0)} \frac{|u(y)|}{|x_0 - y|^{d+\alpha}} \d y + \sup_{x \in B_{\frac{3R}{2}}(x_0)} \int_{B_{2R}(x_0)^c} |u(y)| K(x,y) \d y,\\
\widehat{\tail}_{K,\alpha}(u,R,x_0) := R^{\alpha}\int_{B_{2R}(x_0) \setminus B_{\frac{R}{2}}(x_0)} \frac{|u(y)|}{|x_0 - y|^{d+\alpha}} \d y + \sup_{x \in B_{\frac{3R}{2}}(x_0)} \int_{B_{2R}(x_0)^c} |u(y)| K(y,x) \d y.
\end{align*}

$\tail_{K,\alpha}$ can be regarded as a hybrid between a tail term for general kernels introduced in \cite{Sch20} and a tail term for rotationally symmetric kernels as in \cite{DKP16} and \cite{CKW20}.

The advantage of $\tail_{K,\alpha}$ is that it fits the iteration schemes since for short connections, the weight is a radial function. Moreover, it still takes into account the correct decay of the jumping kernel $K$ for long jumps, which might be of lower order due to the presence of a nonlocal drift term (see $K_3$ in \eqref{eq:kernelclasscone}). Since we do not want to impose any pointwise upper bound on $K$ for long jumps, the second summand contains the supremum in $x$.

We have the following connection between $\tail_K$ and $\tail_{K,\alpha}$:

\begin{lemma}
Assume \eqref{eq:kuppershort}. Let $0 < \rho \le r \le r + \rho \le R \le 1$, $x_0 \in \R^d$ and $v \in V(B_{R}(x_0)|\R^d)$. Then we have
\begin{align}
\label{eq:tailtotail}
\tail_K(v,r,r+\rho,x_0) &\le c \rho^{-\alpha} \left( \frac{r+\rho}{\rho} \right)^d \tail_{K,\alpha}(u,R,x_0),\\
\label{eq:tailtotaildual}
\widehat{\tail}_K(v,r,r+\rho,x_0) &\le c \rho^{-\alpha} \left( \frac{r+\rho}{\rho} \right)^d \widehat{\tail}_{K,\alpha}(u,R,x_0).
\end{align}
\end{lemma}

\begin{proof}
We use that for $x \in B_{r}(x_0)$ and $y \in B_{r+\rho}(x_0)^c \cap B_{2R}(x_0)$, $z \in B_{r+\rho}(x_0)^c \cap B_{2R}(x_0)^c = B_{2R}(x_0)^c$
\begin{align*}
|y - x_0| \le \frac{r+\rho}{\rho}|y-x|, ~~ |z - x_0| \le 2|z-x|,
\end{align*}
which implies upon \eqref{eq:kuppershort} that for every $x \in B_{r}(x_0)$
\begin{align*}
\int_{B_{r+\rho}(x_0)^{c}}v(y) K(x,y) \d y &\le \int_{B_{r+\rho}(x_0)^{c} \cap B_{2R}(x_0)}v(y) K(x,y) \d y + \int_{B_{2R}(x_0)^c} v(z) K(x,z) \d z\\
&\le c_1 \left( \frac{r+\rho}{\rho} \right)^{d+\alpha} \int_{B_{r+\rho}(x_0)^{c} \cap B_{2R}(x_0)}v(y) |x_0 - y|^{-d-\alpha} \d y\\
&\quad+ c_1 \int_{B_{2R}(x_0)^{c}}v(z) K(x,z) \d z\\
&\le c_{2} \rho^{-\alpha} \left(\frac{r+\rho}{\rho}\right)^d \tail_{K,\alpha}(v,R),
\end{align*}
where $c_1,c_2 > 0$. This proves \eqref{eq:tailtotail}, as desired. The proof of \eqref{eq:tailtotaildual} works in the same way.
\end{proof}

Moreover, $\tail_{K,\alpha}(u,R,x_0)$ and $\widehat{\tail}_{K,\alpha}(u,R,x_0)$ are finite for any $u \in V(B_{2R}(x_0)|\R^d)$ under natural and nonrestrictive assumptions on $K$. This property is of some importance to us since it allows us to work with the natural function space $V(B_{2R}(x_0)|\R^d)$ associated with $K$.

\begin{lemma}
\label{lemma:tailfinite}
Assume \eqref{cutoff} and \eqref{elower}.
\begin{itemize}
\item[(i)] If \eqref{UJS} holds true, then $\tail_{K,\alpha}(u,R,x_0) < \infty$ for every $u \in V(B_{2R}(x_0) | \R^d)$,
\item[(ii)] If \eqref{UJSdual} holds true, then $\widehat{\tail}_{K,\alpha}(u,R,x_0) < \infty$ for every $u \in V(B_{2R}(x_0) | \R^d)$.
\end{itemize}
\end{lemma} 

\begin{proof}
We restrict ourselves to proving (i). The proof of (ii) follows via analogous arguments. By \eqref{cutoff}, it clearly suffices to prove that
\begin{align}
\label{eq:finitetailhelp1}
\int_{B_{2R}(x_0) \setminus B_{\frac{R}{2}}(x_0)} |u(y)|^2 |x_0 - y|^{-d-\alpha} \d y + \sup_{x \in B_{\frac{3R}{2}}(x_0)} \int_{B_{2R}(x_0)^c} |u(y)|^2 K(x,y) \d y < \infty.
\end{align}
We start by proving finiteness of the first summand. This can be achieved by the same argument as in the proof of Proposition 12 in \cite{DyKa19}.
We compute using that $|x-y| \le 3|x_0-y|$ for every $x \in B_{\frac{R}{4}}(x_0)$:
\begin{align*}
\int_{B_{2R}(x_0) \setminus B_{\frac{R}{2}}(x_0)}& |u(y)|^2 |x_0 - y|^{-d-\alpha} \d y \le 3 \int_{B_{2R}(x_0) \setminus B_{\frac{R}{2}}(x_0)} \dashint_{B_{\frac{R}{4}}(x_0)} |u(y)|^2 |x-y|^{-d-\alpha} \d x \d y\\
&\le c\int_{B_{2R}(x_0)} \dashint_{B_{\frac{R}{4}}(x_0)} |u(y) - u(x)|^2 |x-y|^{-d-\alpha} \d x \d y\\
&+ c\dashint_{B_{\frac{R}{4}}(x_0)} |u(x)|^2  \left(\int_{B_{2R}(x_0) \setminus B_{\frac{R}{2}}(x_0)} |x-y|^{-d-\alpha}  \d y \right) \d x\\
&\le c R^{-d} \cE^{\alpha}_{B_{2R}(x_0)}(u,u) + c \dashint_{B_{\frac{R}{4}}(x_0)} |u(x)|^2  \left(\int_{B_{\frac{R}{4}}(x)^c} |x-y|^{-d-\alpha} \d y \right) \d x\\
&\le c R^{-d} \cE^{\alpha}_{B_{2R}(x_0)}(u,u) + c R^{-d-\alpha} \Vert u \Vert_{L^2(B_{\frac{R}{4}}(x_0))}^2 < \infty.
\end{align*}
Finiteness of the quantity on the right follows from \eqref{elower} and since $u \in V(B_{2R}(x_0) | \R^d)$.\\
For the second summand in \eqref{eq:finitetailhelp1}, we estimate using \eqref{UJS} and \eqref{cutoff} that for every $x \in B_{\frac{3R}{2}}(x_0)$:
\begin{align*}
&\int_{B_{2R}(x_0)^c} |u(y)|^2 K(x,y) \d y \le \int_{B_{2R}(x_0)^c} \dashint_{B_{\frac{R}{4}}(x)} |u(y)|^2 K(z,y) \d z \d y\\
&\quad\le cR^{-d}\int_{B_{2R}(x_0)^c} \dashint_{B_{2R}(x_0)} \hspace{-0.2cm}|u(y) - u(z)|^2 K_s(z,y) \d z \d y + 2\dashint_{B_{\frac{R}{4}}(x)}\hspace{-0.2cm} |u(z)|^2  \left(\int_{B_{2R}(x_0)^c} \hspace{-0.3cm} K_s(z,y) \d y \right) \d z\\
&\quad\le c R^{-d} [u]^2_{V(B_{2R}(x_0)| \R^d)} + c \dashint_{B_{\frac{R}{4}}(x)} |u(z)|^2  \left(\int_{B_{\frac{R}{4}}(z)^c}  K_s(z,y) \d y \right) \d z\\
&\quad\le c R^{-d} [u]^2_{V(B_{2R}(x_0)| \R^d)} + c R^{-d-\alpha} \Vert u \Vert_{L^2(B_{2R}(x_0))}^2 < \infty.
\end{align*}
Here we used that $B_{\frac{R}{4}}(x) \subset B_{2R}(x_0)$ for every $x \in B_{\frac{3R}{2}}(x_0)$.
\end{proof}

\begin{remark}
Note that \eqref{UJS} and \eqref{UJSdual} are not necessary for $\tail_{K,\alpha}(u,R,x_0)$ and $\widehat{\tail}_{K,\alpha}(u,R,x_0)$ to be finite, respectively. Consider for example a jumping kernel $K$ whose symmetric part satisfies global versions of \eqref{elower}, \eqref{eq:kuppershort}, namely
\begin{align*}
\cE^{K_s}(u,u) \ge c [u]_{H^{\alpha/2}(B_r)}^2, ~~ \forall v \in L^2(B_r),~ r > 0, \qquad
K(x,y) \le c |x-y|^{-d-\alpha}, ~~ \forall x,y \in \R^d,
\end{align*}
then we have that $V(B_{2R}|\R^d) = V^{\alpha}(B_{2R}|\R^d)$. Therefore,
\begin{align*}
\tail_{K,\alpha}(u,R,x_0) \le c\tail_{\alpha}(u,R,x_0) = R^{\alpha} \int_{B_{\frac{R}{2}}(x_0)} |u(y)||x_0-y|^{-d-\alpha} \d y < \infty, ~~ \forall u \in V(B_{2R}| \R^d).
\end{align*}
\end{remark}

\begin{remark}
\begin{itemize}
\item[(i)] Later, we will require finiteness of $\tail_{K,\alpha}(u,R,x_0)$ and $\widehat{\tail}_{K,\alpha}(u,R,x_0)$ in order to deduce local boundedness of weak solutions to \eqref{ellPDE} and \eqref{ellPDEdual} from \autoref{thm:LB1} and \autoref{thm:LB1dual}, respectively. The above lemma shows that under the natural assumptions \eqref{cutoff}, \eqref{elower} and \eqref{UJS}, or \eqref{UJSdual}, finiteness of the tail terms for weak solutions follows already from the solution concept.

\item[(ii)] For parabolic equations, the aforementioned assumptions merely imply finiteness of $\tail_{K,\alpha}(u(t),R,x_0)$ and $\widehat{\tail}_{K,\alpha}(u(t),R,x_0)$ for a.e. $t$, but do not yield a uniform upper bound in $t$. 

\item[(iii)] Since parabolic tails of the form $\sup_{t \in I}\tail_{K}(u(t),r,r+\rho,x_0)$ and $\sup_{t \in I}\widehat{\tail}_{K}(u(t),r,r+\rho,x_0)$ naturally appear in the analysis of solutions to \eqref{PDE} and \eqref{PDEdual}, respectively, it is an important research question to investigate these quantities and to derive suitable estimates. First results have been obtained in \cite{Str19a}, where an estimate for $\sup_{t \in I}\tail_{K}(u(t),r,r+\rho,x_0)$ is derived for global solutions $u$ to \eqref{PDE} in the symmetric case under pointwise bounds for $K$. Another attempt has been made in \cite{Kim19} for solutions to a parabolic boundary value problem with given continuous, bounded data. However, the proof of Lemma 5.3 in \cite{Kim19} is not complete.
\end{itemize}
\end{remark}

\section{Local boundedness via De Giorgi iteration}
\label{sec:Cacc}

The goal of this section is to prove that the supremum of a weak subsolution $u$ to \eqref{PDE}, or to \eqref{PDEdual}, can locally be estimated from above by the $L^2$-norm of $u$ and a nonlocal tail term (see \autoref{thm:LB1}). Under the assumption that the tail term is finite, this result is the key to proving the Harnack inequality. The strategy of proof is based on the De Giorgi iteration for nonlocal operators, as it was adopted in \cite{DKP14}, \cite{DKP16} and \cite{Coz17}.

\subsection{Caccioppoli estimates}
In this section nonlocal Caccioppoli estimates are established. They are derived by testing the weak formulation of \eqref{PDE}, or of \eqref{PDEdual}, with a test function of the form $\tau^2 (u - k)_+$. The lack of symmetry of the jumping kernel $K$ calls for a refinement of the existing proofs for symmetric operators. The main technical ingredient is \autoref{lemma:K1-lemma}.
Such estimates will be used in \autoref{sec:LBDG} to set up a De Giorgi-type iteration scheme which allows us to prove \autoref{thm:LB1}.

The following lemma can be regarded as a generalization of Proposition 8.5 in \cite{Coz17} to nonsymmetric jumping kernels.

\begin{lemma}
\label{lemma:Cacc}
Assume that \eqref{K1}, and \eqref{cutoff} hold true for some $\theta \in [\frac{d}{\alpha},\infty]$. Moreover, assume \eqref{Sob} if $\theta < \infty$. Then there exist $c_1,c_2 > 1$ such that for every $0 < \rho \le r \le 1$, every $l \in \R$, and every function $u \in V(B_{r+\rho}|\R^d)$ it holds
\begin{equation}
\label{eq:Cacc}
\begin{split}
\cEs_{B_{r+\rho}}(\tau w_+,\tau w_+) -& \cE_{B_{r+\rho}}(w_-,\tau w_+) \le c_1\cE(u,\tau^2 w_+) + c_2 \rho^{-\alpha} \Vert w^2_+\Vert_{L^{1}(B_{r+\rho})}\\
&\qquad+ c_2 \Vert w_+ \Vert_{L^1(B_{r+\rho})} \tail_K(w_+,r+\frac{\rho}{2},r+\rho),
\end{split}
\end{equation}
where $B_{2r} \subset \Omega$, $w = u-l$, $\tau = \tau_{r,\frac{\rho}{2}}$.
\end{lemma}

\begin{proof}
Step 1: We claim that there exists a constant $c > 0$ such that
\begin{align}
\label{eq:Caccstep1}
\cEs_{B_{r+\rho}}(\tau w_+,\tau w_+) - \cEs_{B_{r+\rho}}(w_-,\tau w_+) &\le \cE^{K_s}_{B_{r+\rho}}(u,\tau^2 w_+) + c \rho^{-\alpha} \Vert w^2_+\Vert_{L^{1}(B_{r+\rho})}.
\end{align}
Observe that by the following algebraic identities,
\begin{align*}
a-b &= ((a-l)_+ - (b-l)_+) - ((a-l)_- - (b-l)_-),\\
(w_1 - w_2)(\tau_1 w_1 - \tau_2 w_2)& = (\tau_1 w_1 - \tau_2 w_2)^2 - w_1 w_2 (\tau_1 - \tau_2)^2,
\end{align*}
we have that 
\begin{align*}
\cEs_{B_{r+\rho}}(\tau w_+,\tau w_+) &- \cEs_{B_{r+\rho}}(w_-,\tau w_+) \\
&=\cE^{K_s}_{B_{r+\rho}}(u,\tau^2 w_+) +  \int_{B_{r+\rho}}\int_{B_{r+\rho}} w_+(x)w_+(y) (\tau(x) - \tau(y))^2 K_s(x,y) \d y \d x.
\end{align*}
Thus, \eqref{eq:Caccstep1} follows immediately from \eqref{cutoff}.

Step 2: For every $\delta > 0$ there exists $c > 0$ such that
\begin{align}
\label{eq:Caccstep2}
\cEa_{B_{r+\rho}}(u,\tau^2 w_+) \ge -\cE^{K_a}_{B_{r+\rho}}(w_-,\tau^2 w_+) - \delta \cEs_{B_{r+\rho}}(\tau w_+,\tau w_+) - c \rho^{-\alpha}\Vert w_+^2 \Vert_{L^1(B_{r+\rho})}.
\end{align}
For the proof, we first observe the following algebraic identity:
\begin{align*}
(w_1 - w_2)(\tau_1 w_1 + \tau_2 w_2)& = (\tau_1^2 w_1^2 - \tau_2^2 w_2^2) + w_1 w_2 (\tau_2^2 - \tau_1^2).
\end{align*}
Thus, we obtain
\begin{align*}
\cEa_{B_{r+\rho}}(u,\tau^2 w_+) &= -\cE^{K_a}_{B_{r+\rho}}(w_-,\tau^2 w_+) + \int_{B_{r+\rho}}\int_{B_{r+\rho}} (\tau^2 w^2_+(x) - \tau^2 w^2_+(y)) K_a(x,y)\d y \d x\\
&+ \int_{B_{r+\rho}}\int_{B_{r+\rho}} w_+(x) w_+(y) (\tau^2(y) - \tau^2(x)) K_a(x,y) \d y \d x =: I_1 + I_2 + I_3.
\end{align*}
For $I_2$, we estimate using \eqref{K1} and \eqref{eq:K1consequence}
\begin{align*}
I_2 &\ge - \frac{\delta}{2}\cE^{K_s}_{B_{r+\rho}}(\tau w_+ ,\tau w_+) - c \int_{B_{r+\rho}} \tau^2(x) w_+^2(x) \left(\int_{B_{r+\rho}} \frac{|K_a(x,y)|^2}{J(x,y)} \right) \d x\\
&\ge -\delta \cE^{K_s}_{B_{r+\rho}}(\tau w_+ ,\tau w_+) - c \rho^{-\alpha} \Vert w_+^2 \Vert_{L^1(B_{r+\rho})}.
\end{align*}
For $I_3$, we proceed as follows, using that by the following standard estimate
\begin{align}
\label{eq:taualgebra}
(\tau^2(x) - \tau^2(y)) \le  2(\tau(x) - \tau(y))^2 + 2(\tau(x) - \tau(y))(\tau(x) \wedge \tau(y)),
\end{align}
and \eqref{eq:KaKs}, \eqref{cutoff} and \eqref{K1}:
\begin{align*}
I_3 &\ge - 2\int_{B_{r+\rho}}\int_{B_{r+\rho}} (w_+^2(x) \vee w_+^2(y)) (\tau(x) - \tau(y))^2 K_s(x,y) \d y \d x\\
&- 2\int_{B_{r+\rho}}\int_{B_{r+\rho}} (w_+^2(x) \vee w_+^2(y)) (\tau(x) \wedge \tau(y))|\tau(x) - \tau(y)| |K_a(x,y)| \d y \d x\\
&\ge - c \rho^{-\alpha} \Vert w_+^2 \Vert_{L^1(B_{r+\rho})} - \int_{B_{r+\rho}} \tau^2(x)  w_+^2(x) \left( \int_{B_{r+\rho}}\frac{|K_a(x,y)|^2}{J(x,y)} \d y \right) \d x\\
&\ge - c \rho^{-\alpha} \Vert w_+^2 \Vert_{L^1(B_{r+\rho})} - \delta \cE^{K_s}_{B_{r+\rho}}(\tau w_+ ,\tau w_+).
\end{align*}
This proves \eqref{eq:Caccstep2}.

Step 3: Next, let us demonstrate how to prove
\begin{equation}
\label{eq:Caccstep3}
-\cE_{(B_{r+\rho} \times B_{r+\rho})^{c}}(u,\tau^2w_+) \le 2 \left(\int_{B_{r+\rho}} w_+(x) \d x\right)\tail_K(w_+,r+\frac{\rho}{2},r+\rho).
\end{equation}
We estimate
\begin{align*}
-\cE_{(B_{r+\rho} \times B_{r+\rho})^{c}}(u,\tau^2w_+) &= 2\int_{B_{r+\frac{\rho}{2}}}\int_{B_{r+\rho}^{c}}(u(y)-u(x))\tau^2 w_+(x) K(x,y) \d y \d x\\
&\le 2\int_{B_{r+\frac{\rho}{2}}}\int_{B_{r+\rho}^{c}}(u(y)-u(x))_+ \tau^2 w_+(x) K(x,y) \d y \d x\\
&\le 2\int_{B_{r+\frac{\rho}{2}}}\int_{B_{r+\rho}^{c}}(u(y)-l)_+ \tau^2 w_+(x) K(x,y) \d y \d x\\
&\le 2\int_{B_{r+\frac{\rho}{2}}}w_+(x) \sup_{z \in B_{r+\frac{\rho}{2}}}\left(\int_{B_{r+\rho}^{c}}w_+(y) K(z,y) \d y\right) \d x,
\end{align*}
where we used that $K$ is nonnegative and that $\tau \equiv 0$ in $B_{r+\frac{\rho}{2}}^{c}$.

Step 4: We will now combine \eqref{eq:Caccstep1}, \eqref{eq:Caccstep2} and \eqref{eq:Caccstep3}. 
Observe:
\begin{equation*}
\cEs_{B_{r+\rho}}(u,\tau^2w_+) = \cE(u,\tau^2w_+) - \cEa_{B_{r+\rho}}(u,\tau^2w_+) - \cE_{(B_{r+\rho} \times B_{r+\rho})^{c}}(u,\tau^2w_+).
\end{equation*}

Altogether, we immediately obtain the desired result by choosing $\delta > 0$ from Step 2 small enough.
\end{proof}

Note that $-\cE_{B_{r+\rho}}(w_-,\tau^2 w_+) \ge 0$ since $K \ge 0$. Thus, we have the following corollary of \autoref{lemma:Cacc}:

\begin{corollary}
\label{cor:Cacc}
Assume that \eqref{K1} and \eqref{cutoff} hold true for some $\theta \in [\frac{d}{\alpha},\infty]$. Moreover, assume \eqref{Sob} if $\theta < \infty$. Then there exist $c_1,c_2 > 0$ such that for every $0 < \rho \le r \le 1$, every $l \in \R$, and every function $u \in V(B_{r+\rho}|\R^d)$ it holds
\begin{equation}
\label{eq:Cacc2}
\begin{split}
\cEs_{B_{r+\rho}}(\tau w_+,\tau w_+) &\le c_1 \cE(u,\tau^2 w_+) + c_2 \rho^{-\alpha}\Vert w^2_+\Vert_{L^{1}(B_{r+\rho})}\\
&+ c_2 \Vert w_+ \Vert_{L^1(B_{r+\rho})} \tail_K(w_+,r+\frac{\rho}{2},r+\rho),
\end{split}
\end{equation}
where $B_{2r} \subset \Omega$, $w = u-l$ and $\tau = \tau_{r,\frac{\rho}{2}}$.
\end{corollary}

\begin{remark}
Let us point out that both Caccioppoli-type inequalities \eqref{eq:Cacc} and \eqref{eq:Cacc2} appear in the literature for symmetric jumping kernels. \eqref{eq:Cacc} was introduced in \cite{Coz17} (see also \cite{CCV11}, \cite{Coz19}) and is used to prove H\"older estimates for small $\alpha$. For our purposes, \eqref{eq:Cacc2} is sufficient.
\end{remark}

Next, we present a Caccioppoli inequality that is tailored to subsolutions to \eqref{PDEdual}. Note that due to the different shape of the bilinear form, we obtain an additional summand on the right hand side of the estimate.

\begin{lemma}
\label{lemma:Caccdual}
Assume that \eqref{K1}, and \eqref{cutoff} hold true for some $\theta \in [\frac{d}{\alpha},\infty]$. Moreover, assume \eqref{Sob} if $\theta < \infty$. Then there exist $c_1,c_2 > 1$ such that for every $0 < \rho \le r \le 1$, every $l \in \R$, and every function $u \in V(B_{r+\rho}|\R^d)$ it holds
\begin{equation}
\label{eq:Caccdual}
\begin{split}
&\cEs_{B_{r+\rho}}(\tau w_+,\tau w_+) - \cE_{B_{r+\rho}}(w_-,\tau w_+) \le c_1\widehat{\cE}(u,\tau^2 w_+) + c \rho^{-\alpha} \Vert w^2_+\Vert_{L^{1}(B_{r+\rho})}\\
&+ c_2 l^2 \rho^{-\alpha} \left[|A(l,r+\rho)|+ |B_{r+\rho}|\left(\frac{|A(l,r+\rho)|}{|B_{r+\rho}|}\right)^{1/\theta'}\right] + c_2 \Vert w_+ \Vert_{L^1(B_{r+\rho})} \widehat{\tail}_K(u,r+\frac{\rho}{2},r+\rho),
\end{split}
\end{equation}
where $B_{2r} \subset \Omega$, $w = u-l$, $\tau = \tau_{r,\frac{\rho}{2}}$ and $A(l,r+\rho) = \{x \in B_{r+\rho} : w_+ > 0 \}$.
\end{lemma}

\begin{proof}
The proof follows the structure from the proof of \autoref{lemma:Cacc}.

Step 1: As before, there exists a constant $c > 0$ such that
\begin{align}
\label{eq:Caccstep1dual}
\cEs_{B_{r+\rho}}(\tau w_+,\tau w_+) - \cE^{K_s}_{B_{r+\rho}}(w_-,\tau w_+) &\le \cE^{K_s}_{B_{r+\rho}}(u,\tau^2 w_+) + c \rho^{-\alpha} \Vert w^2_+\Vert_{L^{1}(B_{r+\rho})}.
\end{align}

Step 2: We claim that for every $\delta > 0$ there exists $c > 0$ such that
\begin{align}
\label{eq:Caccstep2dual}
\begin{split}
\widehat{\cE}^{K_a}_{B_{r+\rho}}(u&,\tau^2 w_+) \ge -\widehat{\cE}^{K_a}_{B_{r+\rho}}(w_-,\tau^2 w_+) - \delta \cEs_{B_{r+\rho}}(\tau w_+,\tau w_+)\\
&- c \rho^{-\alpha}\Vert w_+^2 \Vert_{L^1(B_{r+\rho})} - c l^2 \rho^{-\alpha} \left[|A(l,r+\rho)|+ |B_{r+\rho}|\left(\frac{|A(l,r+\rho)|}{|B_{r+\rho}|}\right)^{1/\theta'}\right].
\end{split}
\end{align}
This is the main part of the proof and it differs from Step 2 in \autoref{lemma:Cacc}. 
First, we observe:
\begin{align*}
\widehat{\cE}^{K_a}_{B_{r+\rho}}(u,\tau^2 w_+) = \cE^{K_a}_{B_{r+\rho}}(\tau^2 w_+ , u) = - \cE^{K_a}_{B_{r+\rho}}(\tau^2 w_+ , w_-) + \cE^{K_a}_{B_{r+\rho}}(\tau^2 w_+ , w_+) + \cE^{K_a}_{B_{r+\rho}}(\tau^2 w_+ , l).
\end{align*}
To estimate the second term, observe
\begin{align*}
(\tau_1^2 w_1 - \tau_2^2 w_2)(w_1 + w_2) = (\tau_1^2 w_1^2 - \tau_2^2 w_2^2) + w_1 w_2 (\tau_1^2 - \tau_2^2).
\end{align*}
Thus, we note that for every $\delta > 0$ there exists $c > 0$ such that:
\begin{align*}
\cE^{K_a}_{B_{r+\rho}}(\tau^2 w_+ , w_+) &= \int_{B_{r+\rho}}\int_{B_{r+\rho}} (\tau^2 w^2_+(x) - \tau^2 w^2_+(y)) K_a(x,y)\d y \d x\\
&+ \int_{B_{r+\rho}}\int_{B_{r+\rho}} w_+(x) w_+(y) (\tau^2(x) - \tau^2(y)) K_a(x,y) \d y \d x\\
&\ge - \delta \cE^{K_s}_{B_{r+\rho}}(\tau w_+ ,\tau w_+) - c \rho^{-\alpha} \Vert w_+^2 \Vert_{L^1(B_{r+\rho})}.
\end{align*}
The estimate in the last step works exactly like in the estimation of the quantities $I_2, I_3$ in the proof of \autoref{lemma:Cacc}.\\
The estimate of the remaining term $\cE^{K_a}_{B_{r+\rho}}(\tau^2 w_+ , l)$ goes as follows:
\begin{align*}
\cE^{K_a}_{B_{r+\rho}}(\tau^2 w_+ , l) &= 2 l \int_{B_{r+\rho}} \int_{B_{r+\rho}} (\tau^2 w_+(x) - \tau^2 w_+(y)) K_a(x,y) \d y \d x\\
&= 2 l \int_{B_{r+\rho}} \int_{B_{r+\rho}} (\tau(x) - \tau(y)) (\tau w_+(x) + \tau w_+(y)) K_a(x,y) \d y \d x\\
&+ 2 l \int_{B_{r+\rho}} \int_{B_{r+\rho}} (\tau(x) + \tau(y)) (\tau w_+(x) - \tau w_+(y)) K_a(x,y) \d y \d x\\
&=: J_1 + J_2.
\end{align*}
To estimate $J_1$, we apply \eqref{cutoff} and \eqref{eq:K1consequence}:
\begin{align*}
J_1 &\ge - 4l \int_{A(l,r+\rho)} \int_{B_{r+\rho}} |\tau(x) - \tau(y)| \tau w_+(x)  |K_a(x,y)| \d y \d x \\
&\ge - c l^2 \int_{A(l,r+\rho)}  \Gamma^{J}(\tau,\tau)(x) \d x - c \int_{A(l,r+\rho)} \int_{B_{r+\rho}} \tau^2 w_+^2(x) \frac{|K_a(x,y)|^2}{J(x,y)} \d y \d x\\
&\ge -c \rho^{-\alpha} l^2 |A(l,r+\rho)| - \delta \cE^{K_s}_{B_{r+\rho}}(\tau w_+ ,\tau w_+) - c \rho^{-\alpha} \Vert w_+^2 \Vert_{L^1(B_{r+\rho})}.
\end{align*}
$J_2$ can also be estimated with the help of \eqref{cutoff} and \eqref{K1}:
\begin{align*}
J_2 &\ge -4 l \int_{A(l,r+\rho)} \int_{B_{r+\rho}} (\tau(x) + \tau(y)) |\tau w_+(x) - \tau w_+(y)| |K_a(x,y)| \d y \d x\\
&\ge -8 l \int_{A(l,r+\rho)} \int_{B_{r+\rho}} |\tau(x) - \tau(y)| |\tau w_+(x) - \tau w_+(y)| |K_s(x,y)| \d y \d x \\
&-8 l \int_{A(l,r+\rho)} \int_{B_{r+\rho}} (\tau(x) \wedge \tau(y)) |\tau w_+(x) - \tau w_+(y)| |K_a(x,y)| \d y \d x \\
&\ge - c l^2 \int_{A(l,r+\rho)}  \Gamma^{K_s}(\tau,\tau)(x) \d x - \delta \cE^{K_s}_{B_{r+\rho}}(\tau w_+ , \tau w_+)\\
&- \delta \cE^{J}_{B_{r+\rho}}(\tau w_+ , \tau w_+) - c l^2 \int_{A(l,r+\rho)} \int_{B_{r+\rho}} (\tau^2(x) \wedge \tau^2(y)) \frac{|K_a(x,y)|^2}{J(x,y)} \d y \d x\\
&\ge -c \delta \cE^{K_s}_{B_{r+\rho}}(\tau w_+ , \tau w_+) - c l^2 \rho^{-\alpha} |A(l,r+\rho)| - c l^2 \rho^{-\alpha} |B_{r+\rho}| \left(\frac{|A(l,r+\rho)|}{|B_{r+\rho}|}\right)^{1/\theta'}.
\end{align*}
Here, we used that by \eqref{K1} and H\"older's inequality,
\begin{align*}
l^2 \int_{A(l,r+\rho)} \int_{B_{r+\rho}} (\tau^2(x) \wedge \tau^2(y)) \frac{|K_a(x,y)|^2}{J(x,y)} \d y \d x &\le l^2 \int_{A(l,r+\rho)}\hspace{-0.2cm} \tau^2(x) \left(\int_{B_{r+\rho}} \frac{|K_a(x,y)|^2}{J(x,y)} \d y \right) \d x\\
&\le c l^2 \Vert \tau^2 \Vert_{L^{\theta'}(A(l,r+\rho))}\\
&\le c l^2 \rho^{-\alpha} |B_{r+\rho}| \left(\frac{|A(l,r+\rho)|}{|B_{r+\rho}|}\right)^{1/\theta'},
\end{align*}
since $1 \le c |B_{r+\rho}|^{-\frac{\alpha}{d} + 1-\frac{1}{\theta'}} \le c \rho^{-\alpha} |B_{r+\rho}|^{1-\frac{1}{\theta'}}$ for some constant $c > 0$ since $\theta \ge \frac{d}{\alpha}$, which implies that $-\frac{\alpha}{d} + 1 - \frac{1}{\theta'} \in [-\frac{\alpha}{d},0)$ and $\rho \le r \le 1$.

Step 3: Next, let us demonstrate how to prove
\begin{equation}
\label{eq:Caccstep3dual}
-\widehat{\cE}_{(B_{r+\rho} \times B_{r+\rho})^{c}}(u,\tau^2w_+) \le 2 \left(\int_{B_{r+\rho}} w_+(x) \d x\right)\widehat{\tail}_K(u,r+\frac{\rho}{2},r+\rho).
\end{equation}
We estimate
\begin{align*}
-\widehat{\cE}_{(B_{r+\rho} \times B_{r+\rho})^{c}}(u,\tau^2w_+) &= 2 \int_{B_{r+\rho}^{c}} \int_{B_{r+\frac{\rho}{2}}} \tau^2 w_+(y) u(x) K(x,y) \d y \d x\\
&- 2\int_{B_{r+\frac{\rho}{2}}}\int_{B_{r+\rho}^{c}} \tau^2 w_+(x)u(x) K(x,y) \d y \d x\\
&\le 2\int_{B_{r+\frac{\rho}{2}}} w_+(y) \left(\int_{\R^d \setminus B_{r+\rho}} u(x) K(x,y) \d x \right) \d y,
\end{align*}
where we used that $K$ is nonnegative and that $\tau \equiv 0$ in $B_{r+\frac{\rho}{2}}^{c}$. Note that the second summand in the first step is negative since $w_+(x)u(x) \ge 0$, and can therefore be neglected.

Step 4: We will now combine \eqref{eq:Caccstep1dual}, \eqref{eq:Caccstep2dual} and \eqref{eq:Caccstep3dual}. 
Observe that:
\begin{equation*}
\cEs_{B_{r+\rho}}(u,\tau^2w_+) = \widehat{\cE}(u,\tau^2w_+) - \widehat{\cE}^{K_a}_{B_{r+\rho}}(u,\tau^2w_+) - \widehat{\cE}_{(B_{r+\rho} \times B_{r+\rho})^{c}}(u,\tau^2w_+).
\end{equation*}

Altogether, we immediately obtain the desired result by choosing $\delta > 0$ from Step 2 small enough.
\end{proof}

\begin{corollary}
\label{cor:Caccdual}
Assume that \eqref{K1glob} and \eqref{cutoff} hold true for some $\theta \in [\frac{d}{\alpha},\infty]$. Moreover, assume \eqref{Sob} if $\theta < \infty$. Then there exist $c_1,c_2 > 0$ such that for every $0 < \rho \le r \le 1$, every $l \in \R$, and every function $u \in V(B_{r+\rho}|\R^d)$ it holds
\begin{equation}
\label{eq:Cacc2dual}
\begin{split}
&\cEs_{B_{r+\rho}}(\tau w_+,\tau w_+) \le c_1 \widehat{\cE}(u,\tau^2 w_+) + c_2 \rho^{-\alpha}\Vert w^2_+\Vert_{L^{1}(B_{r+\rho})}\\
&+ c_2 l^2 \rho^{-\alpha} \left[|A(l,r+\rho)|+ |B_{r+\rho}|\left(\frac{|A(l,r+\rho)|}{|B_{r+\rho}|}\right)^{1/\theta'}\right] + c_2 \Vert w_+ \Vert_{L^1(B_{r+\rho})} \widehat{\tail}_K(w_+,r+\frac{\rho}{2},r+\rho),
\end{split}
\end{equation}
where $B_{2r} \subset \Omega$, $w = u-l$ and $\tau = \tau_{r,\frac{\rho}{2}}$.
\end{corollary}

\subsection{Local boundedness}
\label{sec:LBDG}

The following theorem is the main result of this section. It yields a priori local boundedness of subsolutions to \eqref{PDE}, or to \eqref{PDEdual} if the nonlocal tail is finite.

\pagebreak[3]
\begin{theorem}
\label{thm:LB1}
Assume that \eqref{eq:kuppershort}, \eqref{cutoff} and \eqref{Sob} hold true.
\begin{itemize}
\item[(i)] Assume that \eqref{K1} holds true for some $\theta \in [\frac{d}{\alpha},\infty]$. Then there exists $c > 0$ such that for every $0 < R \le 1$, every $\delta \in (0,1]$ and every nonnegative, weak subsolution $u$ to \eqref{PDE} in $I_{R}^{\ominus}(t_0) \times B_{2R}$
\begin{align*}
\sup_{I_{R/8}^{\ominus} \times B_{R/2}} u \le c \delta^{-\frac{d+\alpha}{2\alpha}} \left(\dashint_{I_{R/4}^{\ominus}} \dashint_{B_R} u^{2}(t,x)\d x \d t\right)^{1/2} + \delta \sup_{t \in I_{R/4}^{\ominus}}\tail_{K,\alpha}(u(t),R) +\delta R^{\alpha}\Vert f\Vert_{L^{\infty}},
\end{align*}
where $B_{2R} \subset \Omega$.

\item[(ii)] Assume that \eqref{K1glob} holds true for some $\theta \in (\frac{d}{\alpha},\infty]$. Then there exists $c > 0$ such that for every $0 < R \le 1$, every $\delta \in (0,1]$ and every nonnegative, weak subsolution $u$ to \eqref{PDEdual} in $I_{R}^{\ominus}(t_0) \times B_{2R}$
\begin{align*}
\sup_{I_{R/8}^{\ominus} \times B_{R/2}} u \le c \delta^{-\frac{\widetilde{\kappa}'}{2}} \left(\dashint_{I_{R/4}^{\ominus}} \left(\dashint_{B_R} u^{2\theta'}(t,x) \d x\right)^{\frac{1}{\theta'}} \hspace{-0.2cm} \d t\right)^{1/2} \hspace{-0.3cm} + \delta \sup_{t \in I_{R/4}^{\ominus}}\widehat{\tail}_{K,\alpha}(u(t),R) +\delta R^{\alpha}\Vert f\Vert_{L^{\infty}},
\end{align*}
where $B_{2R} \subset \Omega$ and $\widetilde{\kappa} = 1 + \frac{\alpha}{d} - \frac{1}{\theta} > 1$.
\end{itemize}

\end{theorem}

\begin{proof}
We first explain how to prove (i). Let $l > 0$ and define $w_l := (u-l)_+$. Let $r,\rho > 0$ such that $R/2 \le r \le R$ and $\rho \le r \le r+\rho \le R$. 
Let $\tau = \tau_{r,\frac{\rho}{2}}$. Moreover, we define $\chi \in C^1(\R)$ to be a function satisfying $0 \le \chi \le 1$, $\Vert \chi' \Vert_{\infty} \le 16 ((r+\rho)^{\alpha} - r^{\alpha})^{-1}$, $\chi(t_0 - ((r+\rho)/4)^{\alpha}) = 0$, $\chi \equiv 1$ in $I_{r/4}^{\ominus}(t_0)$. Since $u$ is a weak subsolution to \eqref{PDE}, \autoref{lemma:steklovDG} yields that for any $t \in I_{r/4}^{\ominus}(t_0)$
\begin{align*}
\int_{B_{r+\rho}}&\chi^2(t) \tau^2(x) w_l^2(t,x) \d x + \int^t_{t_0 - ((r+\rho)/4)^{\alpha}}\chi^2(s) \cE(u(t),\tau^2 w_l(t))\\
&\le \int^t_{t_0 - ((r+\rho)/4)^{\alpha}}\hspace{-0.2cm}\chi^2(s) (f(s),\tau^2 w_l(s)) \d s + 2 \int^t_{t_0 - ((r+\rho)/4)^{\alpha}}\hspace{-0.2cm}\chi(s) \vert \chi'(s)\vert \int_{B_{r+\rho}}  \tau^2(x) w^2_l(s,x) \d x \d s\\
&\le \Vert f \Vert_{L^{\infty}} \int_{I_{(r+\rho)/4}^{\ominus}} \Vert w_l(s) \Vert_{L^1(B_{r+\rho})} \d s + c_1 ((r+\rho)^{\alpha}-r^{\alpha})^{-1} \int_{I_{(r+\rho)/4}^{\ominus}}\Vert w_l^2(s)\Vert_{L^{1}(B_{r+\rho})} \d s
\end{align*}
for some constant $c_1 > 0$. By application of \autoref{cor:Cacc}, we obtain
\begin{align}
\label{eq:LB1prep15}
\begin{split}
\sup_{t \in I_{r/4}^{\ominus}} &\int_{B_r}w_l^2(t,x) \d x + \int_{I_{r/4}^{\ominus}} \cEs_{B_{r+\rho}}(\tau w_l(s),\tau w_l(s)) \d s\\
&\le c_2 \left(\rho^{-\alpha} \vee ((r+\rho)^{\alpha}-r^{\alpha})^{-1}\right) \int_{I_{(r+\rho)/4}^{\ominus}}\Vert w_l^2(s)\Vert_{L^{1}(B_{r+\rho})} \d s\\
&+ c_2 \Vert w_l\Vert_{L^1(I_{(r+\rho)/4}^{\ominus} \times B_{r+\rho})} \left(\sup_{t \in I_{(r+\rho)/4}^{\ominus}}\tail_K(u(t),r+\frac{\rho}{2},r+\rho)  + \Vert f\Vert_{L^{\infty}} \right)
\end{split}
\end{align}
for some $c_2 > 0$. Recall $\kappa = 1 + \frac{\alpha}{d} > 1$. H\"older interpolation and Sobolev inequality \eqref{Sob} yield
\begin{align}
\label{eq:LB1prep2}
\begin{split}
&\Vert w^2_l \Vert_{L^{\kappa}(I_{r/4}^{\ominus} \times B_r)} \le \left(\sup_{t \in I_{r/4}^{\ominus}}\Vert w^2_l(t) \Vert_{L^{1}(B_r)}^{\kappa-1} \int_{I_{r/4}^{\ominus}}\Vert w^2_l(s) \Vert_{L^{\frac{d}{d-\alpha}}(B_r)} \d s\right)^{1/\kappa}\\
&\qquad\le c_3 \sigma(r,\rho) \Vert w_l^2\Vert_{L^{1}(I_{(r+\rho)/4}^{\ominus} \times B_{r+\rho})}\\
&\qquad+ c_3 \Vert w_l\Vert_{L^1(I_{(r+\rho)/4}^{\ominus} \times B_{r+\rho})} \left(\sup_{t \in I_{(r+\rho)/4}^{\ominus}}\tail_K(u(t),r+\frac{\rho}{2},r+\rho) + \Vert f\Vert_{L^{\infty}}\right),
\end{split}
\end{align}
where $c_3 > 0$ and we used that there is $c > 0$ such that $(\rho^{-\alpha} \vee ((r+\rho)^{\alpha}-r^{\alpha})^{-1}) \le c\rho^{-(\alpha \vee 1)}(r+\rho)^{(\alpha \vee 1) - \alpha} =: \sigma(r,\rho)$. Furthermore, set $|A(l,r)| :=  \int_{I_{r/4}^{\ominus}}\vert\lbrace x \in B_r : u(s,x) > l\rbrace\vert \d s$. Then by application of H\"older's inequality with $\kappa, \frac{\kappa}{\kappa-1}$, both in time and in space, and \eqref{eq:LB1prep2}
\begin{align}
\label{eq:LB1help1}
\begin{split}
&\Vert w^2_l\Vert_{L^{1}(I_{r/4}^{\ominus} \times B_r)} \le \vert A(l,r) \vert^{\frac{1}{\kappa'}} \Vert w^2_l \Vert_{L^{\kappa}(I_{r/4}^{\ominus} \times B_r)}\\
&\qquad\le c_4\vert A(l,r) \vert^{\frac{1}{\kappa'}} \Bigg[\sigma(r,\rho) \Vert w_l^2\Vert_{L^{1}(I_{(r+\rho)/4}^{\ominus} \times B_{r+\rho})}\\
&\qquad+ \Vert w_l\Vert_{L^1(I_{(r+\rho)/4}^{\ominus} \times B_{r+\rho})} \left(\sup_{t \in I_{(r+\rho)/4}^{\ominus}}\tail_K(u(t),r+\frac{\rho}{2},r+\rho)  + \Vert f\Vert_{L^{\infty}} \right) \Bigg],
\end{split}
\end{align}
where $c_4 > 0$ is a constant. 
Let now $0< k < l$ be arbitrary. Then the following holds true:
\begin{align}
\label{eq:LB1help2}
\begin{split}
\Vert w^2_l \Vert_{L^{1}(I_{(r+\rho)/4}^{\ominus} \times B_{r+\rho})} &\le \Vert w^2_k \Vert_{L^{1}(I_{(r+\rho)/4}^{\ominus} \times B_{r+\rho})},\\
\Vert w_l \Vert_{L^1(I_{(r+\rho)/4}^{\ominus} \times B_{r+\rho})} &\le \frac{\Vert w^2_k \Vert_{L^{1}(I_{(r+\rho)/4}^{\ominus} \times B_{r+\rho})}}{l-k},\\
\vert A(l,r)\vert &\le \frac{\Vert w^2_k \Vert_{L^{1}(I_{(r+\rho)/4}^{\ominus} \times B_{r+\rho})}}{(l-k)^2}.
\end{split}
\end{align}
By combining \eqref{eq:LB1help1} and \eqref{eq:LB1help2} we obtain
\begin{align*}
&\Vert w^2_l\Vert_{L^{1}(I_{r/4}^{\ominus} \times B_r)} \\
&\le c_5 \vert A(l,r)\vert^{\frac{1}{\kappa'}} \left( \sigma(r,\rho) + \frac{\sup_{t \in I_{(r+\rho)/4}^{\ominus}}\hspace{-0.2cm}\tail_K(u(t),r+\frac{\rho}{2},r+\rho) + \Vert f\Vert_{L^{\infty}}}{l-k} \right) \Vert w^2_k \Vert_{L^{1}(I_{(r+\rho)/4}^{\ominus} \times B_{r+\rho})} \\
&\le c_6 (l-k)^{-\frac{2}{\kappa'}}\left(\sigma(r,\rho) + \frac{\sup_{t \in I_{(r+\rho)/4}^{\ominus}}\hspace{-0.2cm}\tail_K(u(t),r+\frac{\rho}{2},r+\rho) + \Vert f\Vert_{L^{\infty}} }{l-k}\right) \Vert w^2_k \Vert_{L^{1}(I_{(r+\rho)/4}^{\ominus} \times B_{r+\rho})} ^{1 + \frac{1}{\kappa'}}
\end{align*}
for some $c_5, c_6 > 0$. 
The plan for the remainder of the proof is to iterate the above estimate. Recall \eqref{eq:tailtotail}, which we will apply in the sequel. Let us now set up the iteration scheme.
For this purpose, we define two sequences $l_i = M(1-2^{-i})$ and $\rho_i = 2^{-i-1}R$, $i \in \N$, where $M > 0$ is to be determined later. We also set $r_0 = R$, $r_{i+1} = r_{i} - \rho_{i+1} = \frac{R}{2}(1 + \left(\frac{1}{2}\right)^{i+1})$ and $l_0 = 0$. Then $r_i \searrow R/2$ and $l_i \nearrow M$ as $i \to \infty$.\\
Note that, $\sigma(r_{i},\rho_{i}) \le c_{7} R^{-\alpha}2^{2i}$ for some $c_{7} > 0$. Define $A_i = \Vert w^2_{l_i}\Vert_{L^{1}(I_{r_i/4}^{\ominus} \times B_{r_i})}$. It holds
\begin{align}
\label{eq:LB1help3}
\begin{split}
A_{i} &\le c_{8} \frac{1}{(l_{i} - l_{i-1})^{\frac{2}{\kappa'}}}\left( \sigma(r_i,\rho_i) +  \frac{ \sup_{t \in I_{r_i/4}^{\ominus}}\tail_{K}(u(t),r_i+\frac{\rho_i}{2},r_{i}+\rho_i) + \Vert f\Vert_{L^{\infty}}}{l_{i}-l_{i-1}} \right)A_{i-1}^{1+\frac{1}{\kappa'}}\\
&\le c_{9}\frac{1}{(l_{i} - l_{i-1})^{\frac{2}{\kappa'}}}\left( \sigma(r_i,\rho_i) + \rho_i^{-\alpha}\left(\frac{r_i}{\rho_i}\right)^d \frac{ \sup_{t \in I_{R/4}^{\ominus}}\tail_{K,\alpha}(u(t),R) + R^{\alpha}\Vert f\Vert_{L^{\infty}}}{l_{i}-l_{i-1}} \right)A_{i-1}^{1+\frac{1}{\kappa'}}\\
&\le c_{10}\frac{2^{\frac{2i}{\kappa'}}}{M^{\frac{2}{\kappa'}}}\left(\frac{2^{2i}}{R^{\alpha}}  + \frac{2^{(1+\alpha+d)i}}{R^{\alpha}} \frac{\sup_{t \in I_{R/4}^{\ominus}}\tail_{K,\alpha}(u(t),R) + R^{\alpha}\Vert f\Vert_{L^{\infty}}}{M} \right)A_{i-1}^{1+\frac{1}{\kappa'}}\\
&\le \frac{c_{11}}{R^{\alpha} M^{\frac{2}{\kappa'}}} 2^{\gamma i} \left(1 + \frac{ \sup_{t \in I_{R/4}^{\ominus}}\tail_{K,\alpha}(u(t),R) + R^{\alpha}\Vert f\Vert_{L^{\infty}}}{M}\right)A_{i-1}^{1+\frac{1}{\kappa'}}
\end{split}
\end{align}
for $c_{8}, c_{9}, c_{10}, c_{11} > 0$, $\gamma > 1$. Note that here we also applied \eqref{eq:tailtotail}. If, given $\delta \in (0,1]$, we choose $M \ge \delta \left(\sup_{t \in I_{R/4}^{\ominus}}\tail_{K,\alpha}(u(t),R)+ R^{\alpha}\Vert f\Vert_{L^{\infty}}\right)$ then,
\begin{equation*}
A_i \le \frac{c_{12}}{\delta R^{\alpha}M^{\frac{2}{\kappa'}}}C^i A_{i-1}^{1+\frac{1}{\kappa'}},
\end{equation*}
where $C := 2^{\frac{2}{\kappa'} + 2} > 1$ and $c_{12} > 0$.
We choose
\begin{equation*}
M := \delta \left(\sup_{t \in I_{R/4}^{\ominus}}\tail_{K,\alpha}(u(t),R) + R^{\alpha}\Vert f\Vert_{L^{\infty}}\right) + C^{\frac{\kappa'^2}{2}} c_{12}^{\frac{\kappa'}{2}}\delta^{-\frac{\kappa'}{2}}R^{-\frac{\alpha\kappa'}{2}} A_0^{1/2}.
\end{equation*}
It follows that
\begin{equation*}
A_0 \le c_{12}^{-\kappa'} \delta^{\kappa'} R^{\alpha\kappa'} M^2 C^{-\kappa'^2} = \left(\frac{c_{12}}{\delta R^{\alpha} M^{\frac{2}{\kappa'}}}\right)^{-\kappa'} C^{-\kappa'^2},
\end{equation*}
and therefore we know from Lemma 7.1 in \cite{Giu03} that $A_i \searrow 0$ as $i \to \infty$, i.e.,
\begin{align*}
\sup_{I_{R/8}^{\ominus} \times B_{R/2}} u &\le M = \delta \left(\sup_{t \in I_{R/4}^{\ominus}}\tail_{K,\alpha}(u(t),R) + R^{\alpha}\Vert f\Vert_{L^{\infty}}\right) + C^{\frac{\kappa'^2}{2}} c_{12}^{\frac{\kappa'}{2}}\delta^{-\frac{\kappa'}{2}}R^{-\frac{\alpha\kappa'}{2}} A_0^{1/2}\\
&=\delta \left(\sup_{t \in I_{R/4}^{\ominus}}\tail_{K,\alpha}(u(t),R)+ R^{\alpha}\Vert f\Vert_{L^{\infty}}\right) + c_{13}\delta^{-\frac{\kappa'}{2}}\left(R^{-\alpha\kappa'} \int_{I_{R/4}^{\ominus}} \int_{B_R}\hspace{-0.2cm} u^2(t,x) \d x \d t\right)^{1/2}
\end{align*}
for $c_{13} > 0$. Note that by definition of $\kappa$ it holds that $\alpha\kappa' = \alpha + d$. Therefore, 
\begin{align*}
\sup_{I_{R/8}^{\ominus} \times B_{R/2}} u &\le \delta \sup_{t \in I_{R/4}^{\ominus}}\tail_{K,\alpha}(u(t),R) + \delta R^{\alpha}\Vert f\Vert_{L^{\infty}} + c_{14}\delta^{-\frac{\kappa'}{2}}\left(\dashint_{I_{R/4}^{\ominus}}\dashint_{B_R} u^{2}(t,x)\d x \d t\right)^{1/2}
\end{align*}
for some $c_{14} > 0$. This proves (i).

To prove (ii), observe that instead of \eqref{eq:LB1prep15}, the following estimate follows by applying \autoref{cor:Caccdual} to a weak subsolution $u$ to \eqref{PDEdual}
\begin{align*}
\sup_{t \in I_{r/4}^{\ominus}}\int_{B_{r}} w_l^2(t,x) \d x &+ \int_{I_{r/4}^{\ominus}}\cEs_{B_{r+\rho}}(\tau w_l(t),\tau w_l(t)) \le c_1 \sigma(r,\rho) \Vert w^2_l(t)\Vert_{L^{1}(I_{(r+\rho)/4}^{\ominus} \times B_{r+\rho})}\\
& + c_1 l^2 \rho^{-\alpha} \left[|A(l,r+\rho)| + |B_{r+\rho}|^{\frac{1}{\theta}} \int_{I_{(r+\rho)/4}^{\ominus}} |B_{r+\rho} \cap \{u(t,x) > l\}|^{\frac{1}{\theta'}} \d t \right]\\
&+ c_1 \Vert w_l(t) \Vert_{L^1(I_{(r+\rho)/4}^{\ominus} \times B_{r+\rho})} \left(\sup_{t \in I_{(r+\rho)/4}^{\ominus} }\widehat{\tail}_K(u(t),r+\frac{\rho}{2},r+\rho) + \Vert f\Vert_{L^{\infty}}\right)
\end{align*} 
for some $c_1 > 0$.
By proceeding as in the proof of (i), we derive the following estimate as a replacement of \eqref{eq:LB1help1}, where $\widetilde{\kappa} := \kappa -\frac{1}{\theta} > 1$,
\begin{align*}
&\int_{I_{r/4}^{\ominus}} \Vert w^2_l(t)\Vert_{L^{\theta'}(B_r)} \d t \le \vert A(l,r) \vert^{\frac{1}{\widetilde{\kappa}'}} \left( \int_{I_{r/4}^{\ominus}} \Vert w^2_l(t) \Vert^{\widetilde{\kappa}}_{L^{\widetilde{\kappa} \theta'}(B_r)} \d t \right)^{1/\widetilde{\kappa}}\\
&\le \vert A(l,r) \vert^{\frac{1}{\widetilde{\kappa}'}} \left(\sup_{t \in I_{r/4}^{\ominus}}\Vert w^2_l(t) \Vert_{L^{1}(B_r)}^{\widetilde{\kappa}-1} \int_{I_{r/4}^{\ominus}}\Vert w^2_l(s) \Vert_{L^{\frac{d}{d-\alpha}}(B_r)} \d s\right)^{1/\widetilde{\kappa}}\\
&\le c_2\vert A(l,r+\rho) \vert^{\frac{1}{\widetilde{\kappa}'}} \Bigg[\sigma(r,\rho) \Vert w_l^2\Vert_{L^{1}(I_{(r+\rho)/4}^{\ominus} \times B_{r+\rho})}\\
&\qquad\qquad\qquad\qquad\qquad+ l^2 \rho^{-\alpha} \left[|A(l,r+\rho)| + |B_{r+\rho}|^{\frac{1}{\theta}} \int_{I_{(r+\rho)/4}^{\ominus}} |B_{r+\rho} \cap \{u(t,x) > l\}|^{\frac{1}{\theta'}} \d t \right]  \\
&\qquad\qquad\qquad\qquad\qquad+ \Vert w_l\Vert_{L^1(I_{(r+\rho)/4}^{\ominus} \times B_{r+\rho})} \left(\sup_{t \in I_{(r+\rho)/4}^{\ominus}}\tail_K(u(t),r+\frac{\rho}{2},r+\rho)  + \Vert f\Vert_{L^{\infty}} \right) \Bigg]\\
&\le c_3 |B_{r+\rho}|^{\frac{1}{\theta}} |A(l,r+\rho)|^{\frac{1}{\widetilde{\kappa}'}} \Bigg[ \sigma(r,\rho) \left(1 +  \left(\frac{l}{l-k}\right)^2 \right)\\
&\qquad + \frac{\sup_{t \in I_{(r+\rho)/4}^{\ominus}}\tail_K(u(t),r+\frac{\rho}{2},r+\rho) + \Vert f\Vert_{L^{\infty}}}{l-k} \Bigg] \int_{I_{(r+\rho)/4}^{\ominus}}\Vert w^2_k(t) \Vert_{L^{\theta'}(B_{r+\rho})} \d t \\
&\le c_4 \frac{|B_{r+\rho}|^{\frac{1}{\theta}}}{(l-k)^{\frac{2}{\widetilde{\kappa}'}}} \left[ \sigma(r,\rho) \left(1 + \left(\frac{l}{l-k}\right)^2 \right) + \frac{\sup_{t \in I_{(r+\rho)/4}^{\ominus}}\tail_K(u(t),r+\frac{\rho}{2},r+\rho) + \Vert f\Vert_{L^{\infty}}}{l-k} \right]\\ 
&\qquad\left(\int_{I_{(r+\rho)/4}^{\ominus}}\Vert w^2_k(t) \Vert_{L^{\theta'}(B_{r+\rho})} \d t\right)^{1+\frac{1}{\widetilde{\kappa'}}}
\end{align*}
for some $c_2,c_3,c_4 > 0$ and we used
\begin{align}
\label{eq:reasonforthetanorms}
\int_{I_{(r+\rho)/4}^{\ominus}} |B_{r+\rho} \cap \{u(t,x) > l\}|^{\frac{1}{\theta'}} \d t \le (l-k)^{-2} \int_{I_{(r+\rho)/4}^{\ominus}}\Vert w^2_k(t) \Vert_{L^{\theta'}(B_{r+\rho})} \d t,
\end{align}
and applied \eqref{eq:LB1help2}. From here, the proof basically proceeds as before. We define sequences $(l_i),(\rho_i),(r_i)$ as before, denote $A_i = \int_{I_{r_i/4}^{\ominus}} \Vert w_{l_i}(t) \Vert_{L^{\theta'}(B_{r_i})} \d t$ and deduce that for any $\delta \in (0,1]$, by choosing $M \ge \delta \left(\sup_{t \in I_{R/4}^{\ominus}}\tail_{K,\alpha}(u(t),R)+ R^{\alpha}\Vert f\Vert_{L^{\infty}}\right)$
\begin{equation*}
A_i \le \frac{c_{5}}{\delta R^{\alpha - \frac{d}{\theta}}M^{\frac{2}{\widetilde{\kappa}'}}}C^i A_{i-1}^{1+\frac{1}{\widetilde{\kappa}'}},
\end{equation*}
where $C > 1$ and $c_5 > 0$ are constants.
We choose
\begin{equation*}
M := \delta \left(\sup_{t \in I_{R/4}^{\ominus}}\tail_{K,\alpha}(u(t),R) + R^{\alpha}\Vert f\Vert_{L^{\infty}}\right) + C^{\frac{\widetilde{\kappa}'^2}{2}} c_{5}^{\frac{\widetilde{\kappa}'}{2}}\delta^{-\frac{\widetilde{\kappa}'}{2}} R^{-\frac{\left(\alpha - \frac{d}{\theta}\right)\widetilde{\kappa}'}{2}} A_0^{1/2}.
\end{equation*}
It follows
\begin{equation*}
A_0 \le c_{5}^{-\widetilde{\kappa}'} \delta^{\widetilde{\kappa}'} R^{\left(\alpha -\frac{d}{\theta}\right)\widetilde{\kappa}'} M^2 C^{-\widetilde{\kappa}'^2} = \left(\frac{c_{5}}{\delta R^{\alpha-\frac{d}{\theta}} M^{\frac{d}{\widetilde{\kappa}'}}}\right)^{-\widetilde{\kappa}'} C^{-\widetilde{\kappa}'^2},
\end{equation*}
and therefore we know from Lemma 7.1 in \cite{Giu03} that $A_i \searrow 0$ as $i \to \infty$, i.e.,
\begin{align*}
\sup_{I_{R/8}^{\ominus} \times B_{R/2}} u &\le M = \delta \left(\sup_{t \in I_{R/4}^{\ominus}}\tail_{K,\alpha}(u(t),R) + R^{\alpha}\Vert f\Vert_{L^{\infty}}\right) + C^{\frac{\widetilde{\kappa}'^2}{2}} c_{5}^{\frac{\widetilde{\kappa}'}{2}}\delta^{-\frac{\widetilde{\kappa}'}{2}} R^{-\frac{\left(\alpha - \frac{d}{\theta}\right)\widetilde{\kappa}'}{2}} A_0^{1/2}\\
&=\delta \left(\sup_{t \in I_{R/4}^{\ominus}}\tail_{K,\alpha}(u(t),R)+ R^{\alpha}\Vert f\Vert_{L^{\infty}}\right) + c_{6} \delta^{-\frac{\widetilde{\kappa}'}{2}} \left(\dashint_{I_{R/4}^{\ominus}} \hspace{-0.2cm}\left(\dashint_{B_R} u^{2\theta'}(t,x) \d x\right)^{\frac{1}{\theta'}} \hspace{-0.2cm}\d t\right)^{\frac{1}{2}}
\end{align*}
for $c_{6} > 0$, where we used $\left(\alpha - \frac{d}{\theta}\right)\widetilde{\kappa}' = \alpha + \frac{d}{\theta'}$.
\end{proof} 

\begin{remark}
Let us comment on the appearance of the quantity $L^{2,2\theta'}_{t,x}$-norm of $u$ in the estimate (ii) for subsolutions to \eqref{PDEdual}. In fact, this term appears since we iterate the quantities $L^{2,2\theta'}_{t,x}$-norms of $w_{l_i}$ in the proof of (ii). In fact, upon estimating 
\begin{align*}
|B_{r+\rho}|^{\frac{1}{\theta}} \int_{I_{(r+\rho)/4}^{\ominus}} |B_{r+\rho} \cap \{u(t,x) > l\}|^{\frac{1}{\theta'}} \d t \le c |I_{(r+\rho)/4}^{\ominus} \times B_{r+\rho}|^{\frac{1}{\theta}} |A(l,r+\rho)|^{\frac{1}{\theta'}},
\end{align*}
instead of \eqref{eq:reasonforthetanorms}, we could iterate the $L^{2,2}$-norms of $w_{l_i}$, as in the proof of (i), however, only as long as $\mu := \frac{1}{\kappa'} - \frac{1}{\theta} = \frac{\alpha}{d+\alpha} - \frac{1}{\theta} > 0$. This means that we would have to restrict ourselves to the suboptimal range $\theta \in (\frac{d+\alpha}{\alpha},\infty]$. In the local case, an analogous phenomenon appears in Chapter VI.13 in \cite{Lie96}.\\
Note that for subsolutions \eqref{ellPDEdual}, the analogous condition reads $\mu := \frac{\alpha}{d} - \frac{1}{\theta} > 0$, which allows us to estimate the supremum of $u$ by the $L^2$-norm, as expected for the full range $\theta \in (\frac{d}{\alpha},\infty]$.
\end{remark}

We now state the analog to \autoref{thm:LB1} for stationary solutions.

\begin{theorem}
\label{thm:ellLB1}
Assume that \eqref{eq:kuppershort}, \eqref{cutoff} and \eqref{Sob} hold true.
\begin{itemize}
\item[(i)] Assume that \eqref{K1} holds true for some $\theta \in [\frac{d}{\alpha},\infty]$. Then there exists $c > 0$ such that for every $0 < R \le 1$, every $\delta \in (0,1]$ and every nonnegative, weak subsolution $u$ to \eqref{ellPDE} in $B_{2R} \subset \Omega$
\begin{align}
\label{eq:ellLB1}
\sup_{B_{R/2}} u \le c \delta^{-\frac{d}{2\alpha}} \left( \dashint_{B_R} u^{2}(x) \d x \right)^{1/2} + \delta \tail_{K,\alpha}(u,R) + R^{\alpha}\Vert f\Vert_{L^{\infty}}.
\end{align}

\item[(ii)] Assume that \eqref{K1glob} holds true for some $\theta \in (\frac{d}{\alpha},\infty]$. Then there exists $c > 0$ such that for every $0 < R \le 1$, every $\delta \in (0,1]$ and every nonnegative, weak subsolution $u$ to \eqref{ellPDEdual} in $B_{2R} \subset \Omega$
\begin{align*}
\sup_{B_{R/2}} u \le c \delta^{-\frac{1}{2\mu}} \left( \dashint_{B_R} u^{2}(x) \d x \right)^{1/2} + \delta \widehat{\tail}_{K,\alpha}(u,R) + R^{\alpha}\Vert f\Vert_{L^{\infty}},
\end{align*}
where $\mu := \frac{\alpha}{d} - \frac{1}{\theta} \in (0,\frac{\alpha}{d}]$.
\end{itemize}
\end{theorem}

The first estimate can be read off from \autoref{thm:LB1} (i). The proof of (ii) works similar to the proof of \autoref{thm:LB1} (ii) up to small modifications in the sense of the aforementioned remark. The factor $\delta^{-\frac{d}{2\alpha}}$ in \eqref{eq:ellLB1} stems from defining $\kappa = \frac{d}{d-\alpha}$, $\kappa' = \frac{d}{\alpha}$ in the stationary case.

\section{Local boundedness via Moser iteration}
\label{sec:CaccMoser}

The goal of this section is to give another proof of \autoref{thm:LB1} via the Moser iteration for positive exponents (see \autoref{thm:LB1dual}). For our main result there is no need of a second proof. However, we consider this independent approach interesting due to the wide range of applicability of the Moser iteration. While local boundedness for symmetric nonlocal operators has been established in numerous works by the De Giorgi iteration technique, the following proof of local boundedness (see \autoref{thm:LB1dual}) using a Moser iteration scheme seems to be new. 

The Moser iteration for positive exponents is arguably more complicated than for negative exponents due to the following two reasons. Roughly speaking, one would like to test the equation with test-functions of the form $\phi = \tau^2 u^{2q-1}$ for $q > 1$. Unfortunately, $\phi$ a priori does not belong to the correct function space unless $u$ is bounded. Since boundedness of $u$ is one of the main goals of this section, such an assumption is illegal. Instead, we truncate the monomial $u^{2q-1}$ in an adequate way, similar to \cite{ArSe67}. 
The second reason concerns the appearance of nonlocal tail terms (see \autoref{sec:Cacc}) due to the nonlocality of the equation. These quantities require special treatment in order to make the iteration work. 

Note that \autoref{sec:Cacc} and \autoref{sec:CaccMoser} are fully independent of each other.

\subsection{Algebraic estimates}

The first step is to establish suitable algebraic estimates, which can be seen and will be used as nonlocal analogs to the chain rule. Note that an estimate similar to \eqref{eq:posaux1} was established in \cite{BrPa16}. We also refer to \cite{KaWe22}, where the Moser iteration schemes were established for negative and small positive exponents for the same class of nonsymmetric nonlocal operators.

\begin{lemma}
\label{lemma:posaux}
Let $g : [0,\infty) \to [0,\infty)$ be continuously differentiable. Assume that $g$ is increasing and that $g(0) = 0$. Set $G(t) := \int_0^t g'(\tau)^{1/2} \d \tau$. Then for every $s,t \ge 0$:
\begin{align}
\label{eq:posaux1}
(t-s)(g(t)-g(s)) &\ge (G(t) - G(s))^2,\\
\label{eq:posaux2}
\frac{(g(t) \wedge g(s))|t-s|}{|G(t) - G(s)|} &\le G(t) \wedge G(s),\\
\label{eq:posaux3}
\frac{|g(t) - g(s)|}{|G(t) - G(s)|} &\le g'(t \vee s)^{1/2}.
\end{align}
\end{lemma}

\begin{proof}
Note that by assumption, $t \mapsto G'(t) = g'(t)^{1/2}$ is nonnegative. Let us assume without loss of generality that $s \le t$.
First, we compute with the help of Jensen's inequality
\begin{align*}
(t-s)(g(t)-g(s)) &= (t-s)\int_s^t g'(\tau) \d \tau = (t-s)\int_s^t G'(\tau)^2 \d \tau\\
&\ge \left(\int_s^t G'(\tau) \d \tau \right)^2 = (G(t) - G(s))^2,
\end{align*}
which proves \eqref{eq:posaux1}. Next,
\begin{align*}
\frac{|G(t) - G(s)|}{|t-s|} = \dashint_s^t G'(\tau) \d \tau \ge G'(s).
\end{align*}
Moreover, we compute
\begin{align*}
g(s) = \int_0^s g'(\tau) \d \tau \le g'(s)^{1/2} \int_0^s g'(\tau)^{1/2} \d \tau = G'(s)G(s).
\end{align*}
This implies,
\begin{align*}
\frac{|G(t) - G(s)|}{|t-s|} \ge \frac{g(s)}{G(s)},
\end{align*}
which proves \eqref{eq:posaux2}. For \eqref{eq:posaux3}, we compute using chain rule and again the fact $G'(t) = g'(t)^{1/2}$ is non-decreasing:
\begin{align*}
\frac{|g(t) - g(s)|}{|G(t) - G(s)|} = \left| \dashint_{G(s)}^{G(t)} \left[g \circ G^{-1} \right]'(\tau) \d \tau \right| = \dashint_{G(s)}^{G(t)} g'(G^{-1}(\tau))^{1/2} \d \tau \le g'(t)^{1/2}.
\end{align*}
\end{proof}

The following lemma has already been established and applied in \cite{KaWe22} (see Lemma 3.2).

\begin{lemma}
Let $G : [0,\infty) \to \R$. Then for any $\tau_1,\tau_2 \ge 0$ and $t,s > 0$:
\begin{align}
\label{eq:posaux4}
(\tau_1^2 \wedge \tau_2^2)|G(t) - G(s)|^2 &\ge \frac{1}{2} |\tau_1 G(t) - \tau_2 G(s)|^2 - (\tau_1 - \tau_2)^2 (G^2(t) \vee G^2(s)),\\
\label{eq:posaux5}
(\tau_1^2 \vee \tau_2^2)|G(t) - G(s)|^2 &\le 2 |\tau_1 G(t) - \tau_2 G(s)|^2 +2 (\tau_1 - \tau_2)^2 (G^2(t) \vee G^2(s)).
\end{align}
\end{lemma}

From now on, let us define for $M > 0$ and $q \ge 1$ the function $g : [0,\infty) \to [0,\infty)$ and $G(t) = \int_0^t g'(s)^{1/2} \d s$ via 
\begin{align*}
g(t) &= \begin{cases}
t^{2q-1} &, ~~  t \le M,\\
M^{2q-1} + (2q-1)M^{2q-2}(t-M) &, ~~ t > M,
\end{cases}\\
G(t) &= \begin{cases}
\frac{\sqrt{2q-1}}{q} t^q &, ~~ t \le M,\\
\frac{\sqrt{2q-1}}{q} M^q + \sqrt{2q-1}(t-M)M^{q-1} &, ~~ t > M.
\end{cases}
\end{align*}
One easily checks that $g$ is continuously differentiable, increasing and satisfies $g(0) = 0$.  Therefore $g$ satisfies the assumptions of \autoref{lemma:posaux}. Moreover, note that $g$ is convex.

The following lemma is a direct consequence of the definition of $g$:

\begin{lemma}
For every $t \ge 0$:
\begin{align}
\label{eq:posaux33}
G'(t) = g'(t)^{1/2}  &\le  q \frac{G(t)}{t},\\
\label{eq:posaux6}
g(t)t &\le \frac{q^2}{2q-1}G^2(t).
\end{align}
\end{lemma}

\begin{proof}
Let us start by proving the first estimate.
In case $t \le M$, a direct computation shows:
\begin{align*}
g'(t)^{1/2} = \sqrt{2q-1}t^{q-1} = q \frac{\sqrt{2q-1}}{q} t^{q-1} = q \frac{G(t)}{t}.
\end{align*}
For $t > M$, we obtain using that $\sqrt{2q-1} \le q$:
\begin{align*}
g'(t)^{1/2} &= \sqrt{2q-1} M^{q-1} = q \frac{\frac{\sqrt{2q-1}}{q} (M^q + (t-M)M^{q-1})}{t}\\
&\le  q \frac{\frac{\sqrt{2q-1}}{q}M^q + \sqrt{2q-1}(t-M)M^{q-1}}{t} = q \frac{G(t)}{t}.
\end{align*}
This proves \eqref{eq:posaux33}. For \eqref{eq:posaux6}, we compute in case $t \le M$:
\begin{align*}
g(t)t = t^{2q} = \frac{q^2}{2q-1}G^2(t).
\end{align*}
In case $t > M$, we use that $\sqrt{2q-1} \le q$ to compute:
\begin{align*}
g(t)t = t^2 M^{2q-2} = \frac{q^2}{2q-1}\left(\frac{\sqrt{2q-1}}{q}(M^q + (t-M)M^{q-1}) \right)^2 \le \frac{q^2}{2q-1} G^2(t).
\end{align*}
\end{proof}

\begin{remark}
Note that \eqref{eq:posaux33} already implies a slightly weaker version of the estimate in \eqref{eq:posaux6}. Indeed, by \eqref{eq:posaux33}
\begin{align*}
q^2 G^2(t) \ge (G'(t)t)^2 = g'(t) t^2 \ge g(t) t,
\end{align*}
where we used convexity and $g(0) = 0$ in the last estimate.
\end{remark}

\begin{lemma}
\label{lemma:Gincreasing}
Let $q \ge 1$. Then, for every $s,t \ge 0$ it holds
\begin{equation*}
(G(t) - G(s))^2 \nearrow \frac{2q-1}{q^2}(t^q - s^q)^2, ~~ \text{ as } M \nearrow \infty.
\end{equation*}
\end{lemma}

\begin{proof}
Clearly, $(G(t) - G(s))^2 \to \frac{2q-1}{q^2}(t^q - s^q)^2$, as $M \to \infty$, since for $t,s < M$ it already holds:
\begin{align*}
(G(t) - G(s))^2 = \frac{2q-1}{q^2}(t^q - s^q)^2.
\end{align*}
It remains to prove that the convergence is monotone. Let us fix $t> s > 0$. First, we observe that $M \mapsto (G(t)-G(s))^2$ is continuous. Now, clearly, for $M < t < s$, we have
\begin{align*}
(G(t) - G(s))^2 = (2q-1)M^{2q-2}(t-s)^2,
\end{align*}
which is increasing in $M$. In case $s < M < t$:
\begin{align*}
(G(t) - G(s))^2 = \left( \frac{\sqrt{2q-1}}{q} M^q + (t-M)\sqrt{2q-1}M^{q-1} - \frac{\sqrt{2q-1}}{q}s^q \right)^2.
\end{align*}
This expression is clearly monotone in $M$, as long as $t > M$ since
\begin{align*}
\frac{d}{dM}\frac{\sqrt{2q-1}}{q} M^q + (t-M)\sqrt{2q-1}M^{q-1} = (q-1)\sqrt{2q-1}(t-M)M^{q-2} \ge 0.
\end{align*}
This proves the desired result.
\end{proof}

\subsection{Caccioppoli estimates}

Now, we are in the position to prove the following Caccioppoli-type estimate. We emphasize that $\tau^2 g(\U) \in H_{B_{r+\rho}}(\R^d)$ in the lemma below, whenever $u \in V(B_{r+\rho}|\R^d)$. This is a direct consequence of the definition of $g$.

\begin{lemma}
\label{lemma:MSpos}
Assume that \eqref{K1} and \eqref{cutoff} hold true for some $\theta \in [\frac{d}{\alpha},\infty]$. Moreover, assume \eqref{Sob} if $\theta < \infty$. Then there exist $c_1,c_2 > 0$ such that for every $0 < \rho \le r \le 1$ and every nonnegative function $u \in V(B_{r+\rho}|\R^d)$, and every $q \ge 1$, it holds
\begin{align*}
\cE^{K_s}_{B_{r+\rho}}(\tau G(\U),\tau G(\U)) &\le c_1 \cE(u,\tau^{2}g(\U)) + c_2 \rho^{-\alpha} \Vert G(\U)^2 \Vert_{L^{1}(B_{r+\rho})}\\
&+ c_2 \Vert g(\U) \Vert_{L^1(B_{r+\rho})} \tail_K(u,r+\frac{\rho}{2},r+\rho),
\end{align*}
where $B_{2r} \subset \Omega$, $\tau = \tau_{r,\frac{\rho}{2}}$, and $\U = u + R^{\alpha}\Vert f \Vert_{L^{\infty}}$.
\end{lemma}

\begin{proof}

We define $M := \{(x,y) \in B_{r+\rho} \times B_{r+\rho} : u(x) > u(y) \}$. Note that for $(x,y) \in M$ it holds $g(u(x)) \ge g(u(y))$ and $G(u(x)) \ge G(u(y))$. The proof is divided into several steps.

Step 1: First, we claim that for some $c_1, c_2 > 0$:
\begin{align}
\label{eq:MSstep1}
\cE_{B_{r+\rho}}(u,\tau^2 g(\U)) &\ge c_1 \cE_{B_{r+\rho}}^{K_s}(\tau G(\U),\tau G(\U)) -c_2 \rho^{-\alpha} \Vert G(\U)^2\Vert_{L^1(B_{r+\rho})}.
\end{align}

For the symmetric part, we compute using the symmetry of $K_s$ (see also Lemma 2.3 in \cite{KaWe22}):
\begin{align*}
\cE^{K_s}_{B_{r+\rho}}(u,\tau^{2}g(\U)) &= 2 \iint_{M} (\U(x) - \U(y))(\tau^2(x)g(\U(x)) - \tau^2(y)g(\U(y))) K_s(x,y) \d y \d x\\
&=2 \iint_{M} (\U(x) - \U(y))(g(\U(x)) - g(\U(y)))\tau^2(x) K_s(x,y) \d y \d x\\\
&+ 2 \iint_{M} (\U(x) - \U(y))g(\U(y))(\tau^2(x) - \tau^2(y)) K_s(x,y) \d y \d x\\
&= I_s + J_s.
\end{align*}
For the nonsymmetric part, we compute using the antisymmetry of $K_a$ and with the help of Lemma 2.3 in \cite{KaWe22}:
\begin{align*}
\cE^{K_a}_{B_{r+\rho}}(u,\tau^{2}g(\U)) &= 2 \iint_{M} (\U(x) - \U(y))(\tau^2(x)g(\U(x)) + \tau^2(y)g(\U(y))) K_a(x,y) \d y \d x\\
&=2 \iint_{M} (\U(x) - \U(y))(g(\U(x)) - g(\U(y)))\tau^2(x) K_a(x,y) \d y \d x\\\
&+ 2 \iint_{M} (\U(x) - \U(y))g(\U(y))(\tau^2(x) + \tau^2(y)) K_a(x,y) \d y \d x\\
&= I_a + J_a.
\end{align*}
By adding up $I_s + I_a$ and using \eqref{eq:posaux1}, \eqref{eq:posaux4}, as well as \eqref{cutoff}, we obtain
\begin{align*}
I_s + I_a &= 2 \iint_{M} (\U(x) - \U(y))(g(\U(x)) - g(\U(y)))\tau^2(x) K(x,y) \d y \d x\\
&\ge \iint_{M} (G(\U(x)) - G(\U(y)))^2 ( \tau^2(x) \wedge \tau^2(y) )  K(x,y) \d y \d x\\
&\ge \frac{1}{2} \cE_{B_{r+\rho}}^{K_s}(\tau G(\U) , \tau G(\U)) - c \rho^{-\alpha} \Vert G(\U)^2\Vert_{L^{1}(B_{r+\rho})}\\
&-\frac{1}{2} \int_{B_{r+\rho}} \int_{B_{r+\rho}} (\tau G(\U(x)) - \tau G(\U(y)))^2 |K_a(x,y)| \d y \d x.
\end{align*}
For the nonsymmetric part, we find using \eqref{K1} and \eqref{eq:K1consequence} that for every $\eps > 0$ there is $c> 0$ such that  
\begin{align*}
&\int_{B_{r+\rho}} \int_{B_{r+\rho}} (\tau G(\U(x)) - \tau G(\U(y)))^2 |K_a(x,y)| \d y \d x \\
&\le \eps \cE_{B_{r+\rho}}^{K_s}(\tau G(\U) , \tau G(\U)) + c \int_{B_{r+\rho}} \tau^2(x) G^2(\U(x)) \left(\int_{B_{r+\rho}}  \frac{|K_a(x,y)|^2}{J(x,y)} \right) \d x\\
&\le 2\eps \cE_{B_{r+\rho}}^{K_s}(\tau G(\U) , \tau G(\U)) + c \rho^{-\alpha} \Vert G(\U)^2 \Vert_{L^1(B_{r+\rho})}.
\end{align*}
Consequently,
\begin{equation*}
I_s + I_a \ge \frac{1}{4} \cE_{B_{r+\rho}}^{K_s}(\tau G(\U) , \tau G(\U)) - c \rho^{-\alpha} \Vert G(\U)^2 \Vert_{L^1(B_{r+\rho})}.
\end{equation*}

To estimate $J_s$, we use \eqref{eq:posaux2}, \eqref{eq:posaux5}, and \eqref{cutoff} to prove that for every $\eps > 0$ there exists $c > 0$ such that
\begin{align*}
J_s &\ge - \iint_{M} |G(\U(x)) - G(\U(y))| G(\U(y)) |\tau(x) - \tau(y)| ( \tau(x) \vee \tau(y) ) K_s(x,y) \d y \d x\\
&\ge - \eps \iint_M  (G(\U(x)) - G(\U(y)))^2 ( \tau^2(x) \vee \tau^2(y) ) K_s(x,y) \d y \d x - c\rho^{-\alpha} \Vert G(\U)^2\Vert_{L^{1}(B_{r+\rho})}\\
&\ge - \eps\cE_{B_{r+\rho}}^{K_s}(\tau G(\U) , \tau G(\U)) - c\rho^{-\alpha} \Vert G(\U)^2 \Vert_{L^{1}(B_{r+\rho})}.
\end{align*}

Next, we estimate $J_a$, and prove using \eqref{eq:taualgebra}, \eqref{eq:KaKs}, \eqref{eq:posaux2}, \eqref{cutoff}, and \eqref{eq:posaux5} that for every $\eps > 0$ there is $c> 0$ such that
\begin{align*}
J_a &\ge -8\iint_{M} |\U(x) - \U(y)|g(\U(y)) ( \tau^2(x) \wedge \tau^2(y) ) |K_a(x,y)| \d y \d x\\
&- 8\iint_{M} |\U(x) - \U(y)|g(\U(y)) (\tau(x) - \tau(y))^2 K_s(x,y) \d y \d x\\
&\ge - \eps \iint_M  (G(\U(x)) - G(\U(y)))^2 ( \tau^2(x) \wedge \tau^2(y) ) J(x,y) \d y \d x\\
&- c \iint_M G(\U(y)) ( \tau^2(x) \wedge \tau^2(y) ) \frac{|K_a(x,y)|^2}{J(x,y)} \d y \d x - c \rho^{-\alpha} \Vert G(\U)^2 \Vert_{L^{1}(B_{r+\rho})}\\
&\ge - 2\eps \cE_{B_{r+\rho}}^{K_s}(\tau G(\U) , \tau G(\U)) - c \rho^{-\alpha} \Vert G(\U)^2 \Vert_{L^{1}(B_{r+\rho})},
\end{align*}
where we used \eqref{K1} and \eqref{eq:K1consequence} in the last step to estimate
\begin{align*}
c &\iint_M G^2(\U(y)) ( \tau^2(x) \wedge \tau^2(y) ) \frac{|K_a(x,y)|^2}{J(x,y)} \d y \d x \le 2\eps\cE^{K_s}_{B_{r+\rho}}(\tau G(\U) , \tau G(\U)) + c \rho^{-\alpha} \Vert G(\U)^2 \Vert_{L^{1}(B_{r+\rho})},
\end{align*}
and used Lemma 2.6 in \cite{KaWe22}, \eqref{K1}, \eqref{eq:posaux4} and \eqref{cutoff} to estimate
\begin{align}
\label{eq:absorbJtau}
\begin{split}
\iint_M  &(G(\U(x)) - G(\U(y)))^2 ( \tau^2(x) \wedge \tau^2(y) ) J(x,y) \d y \d x \\
&\le c \int_{B_{r+\rho}}\int_{B_{r+\rho}} (G(\U(x)) - G(\U(y)))^2 ( \tau^2(x) \wedge \tau^2(y) ) K_s(x,y) \d y \d x\\
&\le c \cE_{B_{r+\rho}}^{K_s}(\tau G(\U) , \tau G(\U)) + c \rho^{-\alpha} \Vert G(\U)^2 \Vert_{L^1(B_{r+\rho})}.
\end{split}
\end{align}
Altogether, we obtain:
\begin{align*}
\cE_{B_{r+\rho}}(u,\tau^{2}g(\U)) &\ge \left[\frac{1}{4} - 2\eps\right]\cE_{B_{r+\rho}}^{K_s}(\tau G(\U), \tau G(\U)) - c \rho^{-\alpha} \Vert G(\U)^2 \Vert_{L^{1}(B_{r+\rho})}.
\end{align*}
The desired estimate \eqref{eq:MSstep1} now follows by choosing $\eps > 0$ small enough.

Step 2: In addition, we claim
\begin{equation}
\label{eq:MSstep2}
- \cE_{(B_{r+\rho} \times B_{r+\rho})^{c}}(u,\tau^{2} g(\U)) \le 2 \Vert g(\U) \Vert_{L^1(B_{r+\rho})} \sup_{z \in B_{r+\frac{\rho}{2}}}\left( \int_{B_{r+\rho}^c} u(y) K(z,y) \d y\right).
\end{equation}
To see this, we compute
\begin{align*}
- &\cE_{(B_{r+\rho} \times B_{r+\rho})^{c}}(u,\tau^{2} g(\U)) = -2 \int_{(B_{r+\rho}\times B_{r+\rho})^{c}} (u(x)-u(y))\tau^{2}(x)g(\U(x)) K(x,y) \d y \d x\\
&= -2 \int_{B_{r+\rho}} \hspace{-0.3cm} \tau^{2}(x)u(x)g(\U(x)) \left(\int_{B_{r+\rho}^{c}} \hspace{-0.4cm} K(x,y) \d y \right) \d x +2 \int_{B_{r+\rho}} \hspace{-0.3cm} \tau^{2}(x) g(\U(x)) \left(\int_{B_{r+\rho}^{c}} \hspace{-0.4cm} u(y) K(x,y) \d y \right) \d x\\
&\le 2 \Vert g(\U) \Vert_{L^1(B_{r+\rho})} \sup_{z \in B_{r+\frac{\rho}{2}}}\left( \int_{B_{r+\rho}^c} u(y) K(z,y) \d y\right),
\end{align*}
using that $u,K \ge 0$ and $\supp(\tau) \subset B_{r + \frac{\rho}{2}}$.

Step 3: Observe that
\begin{equation*}
\begin{split}
\cEs_{B_{r+\rho}}(u,\tau^{2}g(\U)) &= \cE(u,\tau^{2} g(\U)) - \cE_{(B_{r+\rho}\times B_{r+\rho})^{c}}(u,\tau^{2} g(\U)).
\end{split}
\end{equation*} 
Therefore, combining \eqref{eq:MSstep1}, \eqref{eq:MSstep2} yields the desired result.
\end{proof}

The following Caccioppoli-type estimate is designed for the dual equation:

\begin{lemma}
\label{lemma:MSposdual}
Assume that \eqref{K1glob} and \eqref{cutoff} hold true for some $\theta \in (\frac{d}{\alpha},\infty]$. Moreover, assume \eqref{Sob} if $\theta < \infty$. Then there exist $c_1,c_2,\gamma > 0$ such that for every $0 < \rho \le r \le 1$ and every nonnegative function $u \in V(B_{r+\rho}|\R^d) \cap L^{2\theta'}(\R^d)$, and every $q \ge 1$, it holds
\begin{align*}
\cE^{K_s}_{B_{r+\rho}}(\tau G(\U),\tau G(\U)) &\le c_1 \widehat{\cE}(u,\tau^{2}g(\U)) + c_2 q^{\gamma} \rho^{-\alpha} \Vert G(\U)^2 \Vert_{L^{1}(B_{r+\rho})}\\
&+ c_2 \Vert g(\U) \Vert_{L^1(B_{r+\rho})} \widehat{\tail}_K(u,r+\frac{\rho}{2},r+\rho),
\end{align*}
where $B_{2r} \subset \Omega$, $\tau = \tau_{r,\frac{\rho}{2}}$, and $\U = u + R^{\alpha}\Vert f \Vert_{L^{\infty}}$.
\end{lemma}

\begin{proof}
Step 1: We claim that there exists $c > 0$ such that for some $\gamma \ge 1$:
\begin{align}
\label{eq:dualMSstep1}
\begin{split}
\widehat{\cE}_{B_{r+\rho}}(u,\tau^2 g(\U)) &\ge c_1 \cE^{K_s}_{B_{r+\rho}}(\tau G(\U) , \tau  G(\U))-c_2 q^{\gamma} \rho^{-\alpha} \Vert G(\U)^2 \Vert_{L^1(B_{r+\rho})}.
\end{split}
\end{align}
Let $M$ be as in the proof of \autoref{lemma:MSpos}. 
Moreover, we observe the following algebraic identity:
\begin{align*}
(a+b)(\tau_1^2 g(\widetilde{a}) - \tau_2^2 g(\widetilde{b})) = (\widetilde{a}-\widetilde{b})(g(\widetilde{a}) - g(\widetilde{b}))\tau_1^2 +2b(g(\widetilde{a}) - g(\widetilde{b}))\tau_1^2 + (a+b) g(\widetilde{b}) (\tau_1^2 - \tau_2^2).
\end{align*}
We use again Lemma 2.3 in \cite{KaWe22} to estimate
\begin{align*}
\widehat{\cE}_{M}^{K_a}(u,\tau^{2}g(\U)) &= 2\iint_M (\U(x) - \U(y))(g(\U(x)) - g(\U(y)))\tau^2(x)  K_a(x,y) \d y \d x\\
&+ 4 \iint_M u(y) (g(\U(x)) - g(\U(y))) \tau^2(x) K_a(x,y) \d y \d x\\
&+  4 \iint_M (u(x) + u(y)) g(\U(y)) (\tau^2(x) - \tau^2(y))  K_a(x,y) \d y \d x\\
&\ge 2\iint_M  (\U(x) - \U(y))(g(\U(x)) - g(\U(y)))\tau^2(x)  K_a(y,x) \d y \d x\\
& - 4 \iint_M \U(x) |g(\U(x)) - g(\U(y))| \tau^2(x) |K_a(x,y)| \d y \d x\\
&-  8 \iint_M g(\U(x))\U(x) |\tau^2(x) - \tau^2(y)|  |K_a(x,y)| \d y \d x\\
&= I_a + M_a + N_a,
\end{align*}
where we used that $u(x) \ge u(y)$ and $g(u(x)) \ge g(u(y))$ on $M$, as well as $u \le \U$.
As in the proof of \autoref{lemma:MSpos}, we can decompose $\cE^{K_s}_{B_{r+\rho}}(u,\tau^2 g(\U)) = I_s + J_s$. Then, using \eqref{eq:posaux1}, \eqref{eq:K1consequence}, and \eqref{cutoff}, we estimate
\begin{align*}
I_s + I_a &= 2\iint_M (\U(x) - \U(y))(g(\U(x)) - g(\U(y)))\tau^2(x) K(x,y) \d y \d x\\
& \ge \frac{1}{4}\cE^{K_s}_{B_{r+\rho}}(\tau G(\U),\tau G(\U)) - c \rho^{-\alpha} \Vert G(\U)^2 \Vert_{L^{1}(B_{r+\rho})}. 
\end{align*}

For $M_a$, we obtain using, \eqref{eq:posaux3}, \eqref{eq:posaux33}, \eqref{eq:posaux5} and \eqref{cutoff}
\begin{align*}
M_a &\ge - 4 q\iint_M \left|G(\U(x)) - G(\U(y)) \right| G(\U(x)) \tau^2(x) |K_a(x,y)| \d y \d x\\
&\ge - \eps \iint_M (G(\U(x)) - G(\U(y)) )^2 (\tau^2(y) \vee \tau^2(x))  J(x,y) \d y \d x\\
&- c q^2 \iint_M \tau^2(x)G^2(\U(x)) \frac{|K_a(x,y)|^2}{J(x,y)} \d y \d x\\
&\ge - c\eps \cE^{K_s}_{B_{r+\rho}}(\tau G(\U) , \tau G(\U)) - c q^{\gamma_1} \rho^{-\alpha} \Vert G(\U)^2 \Vert_{L^{1}(B_{r+\rho})},
\end{align*}

for some $\gamma_1 > 0$, where we used that by \eqref{K1glob} and \eqref{eq:quantifiedK1consequence} applied with some $\delta \le \frac{\eps}{c q^2}$
\begin{align}
\label{eq:applyquantifiedK1consequence}
\begin{split}
c q^2 \iint_M \tau^2(x)G^2(\U(x)) \frac{|K_a(x,y)|^2}{J(x,y)} \d y \d x &\le \eps \cE^{K_s}_{B_{r+\rho}}(\tau G(\U) , \tau G(\U))\\
&+ c q^2 (\delta^{-\gamma_2} + \delta)\rho^{-\alpha} \Vert G(\U)^2 \Vert_{L^{1}(B_{r+\rho})}
\end{split}
\end{align}
for some $\gamma_2 > 0$ and moreover, by \eqref{eq:taualgebra}, \eqref{cutoff}, and using the same argument as in \eqref{eq:absorbJtau}:
\begin{align*}
\eps &\iint_M (G(\U(x)) - G(\U(y)) )^2 (\tau^2(y) \vee \tau^2(x)) J(x,y) \d y \d x\\
&\le 2\eps \int_{B_{r+\rho}}\int_{B_{r+\rho}} (G(\U(x)) - G(\U(y)) )^2 (\tau^2(y) \wedge \tau^2(x)) J(x,y) \d y \d x + c \rho^{-\alpha}\Vert G(\U)^2 \Vert_{L^1(B_{r+\rho})}\\
&\le c\eps \cE^{K_s}_{B_{r+\rho}}(\tau G(\U) , \tau G(\U)) + c \rho^{-\alpha}\Vert G(\U)^2 \Vert_{L^1(B_{r+\rho})}.
\end{align*}
For $N_a$, we compute using \eqref{eq:posaux6}, \eqref{eq:taualgebra}, \eqref{eq:KaKs}
\begin{align*}
N_a &\ge -cq \iint_M G^2(\U(x))(\tau(x)-\tau(y))^2 K_s(x,y) \d y \d x\\
&-cq \iint_M G^2(\U(x))(\tau(x) \wedge \tau(y))|\tau(x)-\tau(y)| |K_a(x,y)| \d y \d x\\
&\ge -c q \rho^{-\alpha} \Vert G(\U)^2 \Vert_{L^{1}(B_{r+\rho})}- cq^2\int_{B_{r+\rho}} \int_{B_{r+\rho}} \tau^2(x)G^2(\U(x)) \frac{|K_a(x,y)|^2}{J(x,y)} \d y \d x\\
&\ge- c q^{\gamma_3} \rho^{-\alpha} \Vert G(\U)^2 \Vert_{L^{1}(B_{r+\rho})} - \eps \cE^{K_s}_{B_{r+\rho}}(\tau G(\U), \tau G(\U))
\end{align*}
for some $\gamma_3 > 0$, where we applied \eqref{cutoff} and used the same argument as in \eqref{eq:applyquantifiedK1consequence} to estimate the second summand in the last step.
Altogether, we have shown
\begin{align*}
\widehat{\cE}_{B_{r+\rho}}(u,\tau^2 g(\U)) &\ge \left[ \frac{1}{4} - 3\eps \right] \cE^{K_s}_{B_{r+\rho}}(\tau G(\U),\tau G(\U)) - c q^{\gamma} \rho^{-\alpha} \Vert G(\U)^2 \Vert_{L^{1}(B_{r+\rho})}.
\end{align*}
Thus, by choosing $\eps > 0$ small enough, we obtain \eqref{eq:dualMSstep1}, as desired.

Step 2: Moreover, it holds
\begin{equation}
\label{eq:dualMSstep2}
- \widehat{\cE}_{(B_{r+\rho}\times B_{r+\rho})^{c}}(u,\tau^2 g(\U)) \le  c \Vert g(\U) \Vert_{L^1(B_{r+\rho})} \sup_{z \in B_{r+\frac{\rho}{2}}}\left( \int_{B_{r+\rho}^c} u(y) K(y,z) \d y\right).
\end{equation}
The proof works similar to the proof of Step 2 in \autoref{lemma:MSpos}
\begin{align*}
- &\widehat{\cE}_{(B_{r+\rho} \times B_{r+\rho})^{c}}(u,\tau^{2} g(\U)) = -2 \int_{(B_{r+\rho}\times B_{r+\rho})^{c}} (\tau^{2}g(\U(x)) - \tau^{2}g(\U(y))) u(x) K(x,y) \d y \d x\\
&= -2 \int_{B_{r+\rho}}\hspace{-0.4cm} \tau^{2}(x)u(x)g(\U(x)) \left(\int_{B_{r+\rho}^{c}}\hspace{-0.4cm} K(x,y) \d y \right) \d x +2 \int_{B_{r+\rho}}\hspace{-0.4cm} \tau^{2}(y) g(\U(y)) \left(\int_{B_{r+\rho}^{c}}\hspace{-0.4cm} u(x) K(x,y) \d x \right) \d y\\
&\le 2 \Vert g(\U) \Vert_{L^1(B_{r+\rho})} \sup_{z \in B_{r+\frac{\rho}{2}}}\left( \int_{B_{r+\rho}^c} u(y) K(y,z) \d y\right),
\end{align*}
using that $u,K \ge 0$ and $\supp(\tau) \subset B_{r + \frac{\rho}{2}}$.
\end{proof}

\subsection{Local boundedness}

Now, we will show how to prove \autoref{thm:LB1} via the Moser iteration. Note that we get a slightly better bound for subsolutions to \eqref{PDEdual} compared to \autoref{thm:LB1} (ii).

\begin{theorem}
\label{thm:LB1dual}
Assume that \eqref{eq:kuppershort}, \eqref{cutoff} and \eqref{Sob} hold true.
\begin{itemize}
\item[(i)] Assume \eqref{K1} holds true for some $\theta \in [\frac{d}{\alpha},\infty]$. Then there exists $c > 0$ such that for every $0 < R \le 1$, and every nonnegative, weak subsolution $u$ to \eqref{PDE} in $I_{R}^{\ominus}(t_0) \times B_{2R}$
\begin{align}
\label{eq:parLBMoser}
\sup_{I_{R/8}^{\ominus} \times B_{R/2}} u \le c \left(\dashint_{I_{R/4}^{\ominus}} \dashint_{B_R} u^{2}(t,x)\d x \d t\right)^{1/2} \hspace{-0.3cm} + c\sup_{t \in I_{R/4}^{\ominus}}\tail_{K,\alpha}(u(t),R) + c R^{\alpha}\Vert f \Vert_{L^{\infty}},
\end{align}
where $B_{2R} \subset \Omega$.

\item[(ii)] Assume \eqref{K1glob} holds true for some $\theta \in (\frac{d}{\alpha},\infty]$. Then there exists $c > 0$ such that for every $0 < R \le 1$, and every nonnegative, weak subsolution $u$ to \eqref{PDEdual} in $I_{R}^{\ominus}(t_0) \times B_{2R}$
\begin{align}
\label{eq:parLBMoserdual}
\sup_{I_{R/8}^{\ominus} \times B_{R/2}} u \le c \left(\dashint_{I_{R/4}^{\ominus}} \dashint_{B_R} u^{2}(t,x)\d x \d t\right)^{1/2} \hspace{-0.3cm} + c\sup_{t \in I_{R/4}^{\ominus}}\widehat{\tail}_{K,\alpha}(u(t),R)+ c R^{\alpha}\Vert f \Vert_{L^{\infty}},
\end{align}
where $B_{2R} \subset \Omega$.
\end{itemize}
\end{theorem} 

\begin{proof}
We will only demonstrate the proof of (ii). The proof of (i) goes via the same arguments, but uses \autoref{lemma:MSpos} instead of \autoref{lemma:MSposdual}. Let $0 < \rho \le r \le r+\rho \le R$ and $q \ge 1$. 
We obtain by applying \autoref{lemma:MSposdual}
\begin{align}
\label{eq:LB1Moserprep1}
\begin{split}
c&\int_{B_{r+\rho}} \tau^2(x) \partial_t u(t,x) g(\U(t,x)) \d x + \cEs_{B_{r+\rho}}(\tau G(\U),\tau G(\U))\\
&\le c\left[ (\partial_{t} u(t) , \tau^2 g(\U(t))) + \cE(u(t),\tau^2 g(\U(t)))\right]\\
&+c q^{\gamma} \rho^{-\alpha} \Vert G(\U(t))^2 \Vert_{L^{1}(B_{r+\rho})} + c \Vert g(\U(t)) \Vert_{L^1(B_{r+\rho})} \widehat{\tail}_K(u(t),r+\frac{\rho}{2},r+\rho)\\
&\le c(f(t),\tau^2 g(\U(t)))\\
&+ c q^{\gamma} \rho^{-\alpha} \Vert G(\U(t))^2 \Vert_{L^{1}(B_{r+\rho})} + c \Vert g(\U(t)) \Vert_{L^1(B_{r+\rho})} \widehat{\tail}_K(u(t),r+\frac{\rho}{2},r+\rho)\\
&\le c q^{\gamma} \rho^{-\alpha} \Vert G(\U(t))^2 \Vert_{L^{1}(B_{r+\rho})} + c \Vert g(\U(t)) \Vert_{L^1(B_{r+\rho})} \widehat{\tail}_K(u(t),r+\frac{\rho}{2},r+\rho),
\end{split}
\end{align}
where $c > 0$ is the constant from \autoref{lemma:MSposdual} and we tested the equation with $\tau^2 g(u)$, where $\tau = \tau_{r,\frac{\rho}{2}}$. Moreover, we used that by definition of $\U$:
\begin{align*}
(f(t),\tau^2 g(\U(t))) \le c \rho^{-\alpha}\Vert G(\U(t))^2 \Vert_{L^1(B_{r+\rho})}.
\end{align*}
We observe that
\begin{align*}
(\partial_t u) g(\U) = \begin{cases}
\frac{1}{2q} \partial_t (\U^{2q})&, ~~u \le M,\\
\frac{1}{2} M^{2q-2} \partial_t(\U^2)&, ~~u > M.
\end{cases}
\end{align*}
Next, we define $\chi \in C^1(\R)$ to be a function satisfying $0 \le \chi \le 1$, $\Vert \chi' \Vert_{\infty} \le 16 ((r+\rho)^{\alpha} - r^{\alpha})^{-1}$, $\chi(t_0 - ((r+\rho)/4)^{\alpha}) = 0$, $\chi \equiv 1$ in $I_{r/4}^{\ominus}(t_0)$. By multiplying \eqref{eq:LB1Moserprep1} with $\chi^2$ and integrating over $(t_0 - ((r+\rho)/4)^{\alpha},t)$ for some arbitrary $t \in I_{r/4}^{\ominus}(t_0)$, we obtain
\begin{align*}
\int_{B_{r+\rho}}&\chi^2(t) \tau^2(x) H(\U(t,x)) \d x + \int^t_{t_0 - ((r+\rho)/4)^{\alpha}}\chi^2(s) \cEs_{B_{r+\rho}}(\tau G(\U(s)),\tau G(\U(s))) \d s\\
&\le c_2 q^{\gamma} \rho^{-\alpha} \int^t_{t_0 - ((r+\rho)/4)^{\alpha}} \chi^2(s) \Vert G(\U(s))^2 \Vert_{L^{1}(B_{r+\rho})} \d s\\
&+ c_2 \int^t_{t_0 - ((r+\rho)/4)^{\alpha}}\chi(s) \vert \chi'(s)\vert \int_{B_{r+\rho}}  \tau^2(x) H(\U(s,x)) \d x \d s\\
&+ c_2 \int^t_{t_0 - ((r+\rho)/4)^{\alpha}}\chi^2(s) \Vert g(\U(s)) \Vert_{L^1(B_{r+\rho})} \widehat{\tail}_K(u(s),r+\frac{\rho}{2},r+\rho) \d s
\end{align*}
for some $c_2 > 0$, where
\begin{align*}
H(t) = \begin{cases}
\frac{1}{2q} t^{2q}&, ~~t \le M,\\
\frac{1}{2} M^{2q-2} t^2&, ~~t > M.
\end{cases}
\end{align*}
Consequently,
\begin{align*}
\sup_{t \in I_{r/4}^{\ominus}} &\int_{B_r}H(\U(t,x)) \d x + \int_{I_{r/4}^{\ominus}} \cEs_{B_{r+\rho}}(\tau G(\U(s)),\tau G(\U(s))) \d s\\
&\le c_3 q^{\gamma} \left(\rho^{-\alpha} \vee ((r+\rho)^{\alpha}-r^{\alpha})^{-1}\right) \left(\Vert H(\U)\Vert_{L^{1}(I_{(r+\rho)/4}^{\ominus} \times B_{r+\rho})} + \Vert G(\U)^2\Vert_{L^{1}(I_{(r+\rho)/4}^{\ominus} \times B_{r+\rho})}\right)\\
&+ c_3 \Vert g(\U)\Vert_{L^1(I_{(r+\rho)/4}^{\ominus} \times B_{r+\rho})} \sup_{t \in I_{(r+\rho)/4}^{\ominus}}\widehat{\tail}_K(u(t),r+\frac{\rho}{2},r+\rho)
\end{align*}
for some $c_3 > 0$. Now, we send $M \nearrow \infty$. By monotone convergence, the definition of $g,G,H$ and \eqref{lemma:Gincreasing}
\begin{align*}
\sup_{t \in I_{r/4}^{\ominus}} &\int_{B_r} \U^{2q}(t,x) \d x + \int_{I_{r/4}^{\ominus}} \cEs_{B_{r+\rho}}(\tau \U^q(s),\tau \U^q(s) ) \d s\\
&\le c_4 q^{\gamma} \left(\rho^{-\alpha} \vee ((r+\rho)^{\alpha}-r^{\alpha})^{-1}\right) \Vert \U^{2q}\Vert_{L^{1}(I_{(r+\rho)/4}^{\ominus} \times B_{r+\rho})}\\
&+ c_4 q \Vert \U^{2q-1}\Vert_{L^1(I_{(r+\rho)/4}^{\ominus} \times B_{r+\rho})} \sup_{t \in I_{(r+\rho)/4}^{\ominus}}\widehat{\tail}_K(u(t),r+\frac{\rho}{2},r+\rho)
\end{align*}
for some $c_4 > 0$. Recall that $\kappa = 1 + \frac{\alpha}{d} > 1$. By H\"older interpolation and Sobolev inequality \eqref{Sob}, we derive the following result:
\begin{align}
\label{eq:posMoserItStep}
\begin{split}
\Vert \U^{2q} \Vert_{L^{\kappa}(I_r^{\ominus} \times B_r)}  &\le \left(\sup_{t \in I_r^{\ominus}} \Vert \U^{2q}(t) \Vert_{L^1(B_r)}^{\kappa-1} \int_{I_r^{\ominus}} \Vert \U^{2q}(s) \Vert_{L^{\frac{d}{d-\alpha}}(B_r)} \d s \right)^{1/\kappa}\\
&\le c q^{\gamma} (\rho^{-\alpha} \vee ((r+\rho)^{\alpha} -r^{\alpha})^{-1}) \Vert \U^{2q} \Vert_{L^1(I_{r+\rho}^{\ominus} \times B_{r+\rho})}\\
&+  cq \Vert \U^{2q-1}\Vert_{L^1(I_{(r+\rho)/4}^{\ominus} \times B_{r+\rho})} \sup_{t \in I_{(r+\rho)/4}^{\ominus}}\widehat{\tail}_K(u(t),r+\frac{\rho}{2},r+\rho).
\end{split}
\end{align}
We will now demonstrate how to perform the Moser iteration for positive exponents for nonlocal equations. \eqref{eq:posMoserItStep} is the key estimate for the iteration scheme. The main difficulty compared to the classical local case it the treatment of the tail term.

Let us define for $\eps > 0$ to be determined later and $i \in \N$: $c_i = 2^{-(i+1)\frac{d+\eps}{\alpha}} < 1$. By H\"older's and Young's inequality we have for each $i \in \N$ the following estimate:
\begin{align*}
q&\Vert \U^{2q-1}\Vert_{L^1(I_{(r+\rho)/4}^{\ominus} \times B_{r+\rho})} \sup_{t \in I_{(r+\rho)/4}^{\ominus}}\widehat{\tail}_K(u(t),r+\frac{\rho}{2},r+\rho)\\
&\le \left(q(c_i\rho)^{-\alpha}\Vert \U^{2q}\Vert_{L^1(I_{(r+\rho)/4}^{\ominus} \times B_{r+\rho})} \right)^{\frac{2q-1}{2q}}\\
&\qquad\left(q^{\frac{1}{2q}}(c_i \rho)^{\alpha \frac{2q-1}{2q}}  (r+\rho)^{\frac{d+\alpha}{2q}} \sup_{t \in I_{(r+\rho)/4}^{\ominus}}\widehat{\tail}_K(u(t),r+\frac{\rho}{2},r+\rho) \right)\\
&\le q (c_i \rho)^{-\alpha} \Vert \U^{2q} \Vert_{L^1(I_{(r+\rho)/4}^{\ominus} \times B_{r+\rho})}\hspace{-0.1cm}  + \hspace{-0.1cm} \left(q^{\frac{1}{2q}}(c_i \rho)^{\alpha \frac{2q-1}{2q}}  (r+\rho)^{\frac{d+\alpha}{2q}} \hspace{-0.3cm}\sup_{t \in I_{(r+\rho)/4}^{\ominus}}\hspace{-0.2cm}\widehat{\tail}_K(u(t),r+\frac{\rho}{2},r+\rho) \right)^{2q}\hspace{-0.2cm}.
\end{align*}
Combining this estimate with \eqref{eq:posMoserItStep} and taking both sides to the power $\frac{1}{2q}$ yields:
\begin{align*}
\Vert \U \Vert_{L^{2q\kappa}(I_r^{\ominus} \times B_r)} &\le c^{\frac{1}{2q}} q^{\frac{\gamma}{2q}} c_i^{-\frac{\alpha}{2q}} \left(\rho^{-\frac{\alpha}{2q}} \vee ((r+\rho)^{\alpha}-r^{\alpha})^{-\frac{1}{2q}}\right)\Vert \U \Vert_{L^{2q}(I_{r+\rho}^{\ominus} \times B_{r+\rho})}\\
&+ c^{\frac{1}{2q}} q^{\frac{1}{2q}} (c_i \rho)^{\alpha\frac{2q-1}{2q}} (r+\rho)^{\frac{d+\alpha}{2q}} \sup_{t \in I_{(r+\rho)/4}^{\ominus}}\widehat{\tail}_K(u(t),r+\frac{\rho}{2},r+\rho)\\
&\le c^{\frac{1}{2q}} q^{\frac{\gamma}{2q}} c_i^{-\frac{\alpha}{2q}} \left(\rho^{-\frac{\alpha}{2q}} \vee ((r+\rho)^{\alpha}-r^{\alpha})^{-\frac{1}{2q}}\right)\\
&\left(\Vert \U \Vert_{L^{2q}(I_{r+\rho}^{\ominus} \times B_{r+\rho})} + (c_i\rho)^{\alpha} (r+\rho)^{\frac{d+\alpha}{2q}} \sup_{t \in I_{(r+\rho)/4}^{\ominus}}\widehat{\tail}_K(u(t),r+\frac{\rho}{2},r+\rho) \right).
\end{align*}

Recall that by \eqref{eq:tailtotaildual}, we have the following estimate:
\begin{align*}
\widehat{\tail}_K(u(t),r+\frac{\rho}{2},r+\rho) \le c \rho^{-\alpha}\left(\frac{r+\rho}{\rho}\right)^d \widehat{\tail}_{K,\alpha}(u(t),R).
\end{align*}

We fix $q_0 \ge 1$, $q_i = q_0 \kappa^i$, and set $\rho_i = 2^{-i-1}R$ and $r_{i+1} = r_{i} - \rho_{i+1}$, $r_0 = R$. Note that $r_i \searrow R/2$. We obtain for every $i \in \N$, using that $(\rho_i^{-\frac{\alpha}{2q_{i-1}}} \vee ((r_i+\rho_i)^{\alpha} -r_i^{\alpha})^{-\frac{1}{2q_{i-1}}}) \le c^{\frac{1}{2q_{i-1}}} R^{-\frac{\alpha}{2q_{i-1}}} 2^{\frac{i+1}{q_{i-1}}}$:
\begin{align}
\label{eq:posMoserIthelp1}
\begin{split}
\Vert \U \Vert_{L^{q_{i}}(I^{\ominus}_{r_{i}} \times B_{r_{i}})} &\le c^{\frac{1}{2q_{i-1}}}q_{i-1}^{\frac{\gamma}{2q_{i-1}}} R^{-\frac{\alpha}{2q_{i-1}}} 2^{\frac{d+\eps+2}{2 q_{i-1}} (i+1)} \\
&\left( \Vert \U \Vert_{L^{2q_{i-1}}(I^{\ominus}_{r_{i-1}} \times B_{r_{i-1}})} + 2^{-(d+\eps)(i+1)}\rho_{i}^{\alpha} R^{\frac{d+\alpha}{2q_{i-1}}} \hspace{-0.2cm} \sup_{t \in I_{(r_{i}+\rho_{i})/4}^{\ominus}} \hspace{-0.3cm} \widehat{\tail}_{K}(u(t),r_i + \frac{\rho_i}{2},r_i+\rho_i) \right)\\
&\le c^{\frac{1}{2q_{i-1}}}q_{i-1}^{\frac{\gamma}{2q_{i-1}}} R^{-\frac{\alpha}{2q_{i-1}}} 2^{\frac{d+\eps+2}{2 q_{i-1}} (i+1)} \\
&\left( \Vert \U \Vert_{L^{2q_{i-1}}(I^{\ominus}_{r_{i-1}} \times B_{r_{i-1}})} + R^{\frac{d+\alpha}{2q_{i-1}}} 2^{-(i+1) \eps} \sup_{t \in I_{R/4}^{\ominus}} \widehat{\tail}_{K,\alpha}(u(t),R) \right).
\end{split}
\end{align}
Consequently,
\begin{align*}
\sup_{I_{R/2}^{\ominus} \times B_{R/2}} \U &\le \left( \prod_{i = 1}^{\infty} c^{\frac{1}{2q_{i-1}}}q_{i-1}^{\frac{\gamma}{2q_{i-1}}} R^{-\frac{\alpha}{2q_{i-1}}} 2^{\frac{d+\eps+2}{2 q_{i-1}} (i+1)}  \right) \Vert \U \Vert_{L^{2q_0}(I_{R}^{\ominus} \times B_{R})}\\
&+ \left[\sum_{i = 1}^{\infty} \left( \prod_{j = i}^{\infty} c^{\frac{1}{2q_{j-1}}} q_{j-1}^{\frac{\gamma}{2q_{j-1}}} R^{-\frac{\alpha}{2q_{j-1}}} 2^{\frac{d+\eps+2}{2 q_{j-1}} (j+1)}  \right) R^{\frac{d+\alpha}{2q_{i-1}}} 2^{-(i+1)\eps} \right] \sup_{t \in I_{R/4}^{\ominus}} \widehat{\tail}_{K,\alpha}(u(t),R).
\end{align*}
Note that $\sum_{i = 0}^{\infty} \kappa^{-i} = \frac{d+\alpha}{\alpha}$ and also $\sum_{i = 0}^{\infty} \frac{i}{\kappa^i} =: c_3 < \infty$. Therefore,
\begin{align*}
\prod_{i = 1}^{\infty} (c q_{i-1})^{\frac{\gamma}{2q_{i-1}}} &\le (c q_0)^{\frac{\gamma}{2q_0} \sum_{i = 0}^{\infty} \kappa^{-i}} \kappa^{\frac{\gamma}{2q_0} \sum_{i = 0}^{\infty} \frac{i}{\kappa^i}} \le c(q_0,\kappa,\gamma) \hspace{-0.2cm} < \infty,\\
\prod_{i = 1}^{\infty} 2^{\frac{d+\eps+2}{2 q_{i-1}} (i+1)} &\le 2^{\frac{d+\eps+2}{2q_0} \sum_{i = 0}^{\infty} \frac{i+2}{\kappa^i}} \le 2^{\frac{(d+\eps+2)c_4}{2q_0}} < \infty, \quad \prod_{j=i}^{\infty} R^{-\frac{\alpha}{2q_{j-1}}} &\hspace{-0.2cm}= R^{-\frac{\alpha}{2q_{i-1}}\sum_{j=0}^{\infty} \kappa^{-j}} \hspace{-0.2cm} = R^{-\frac{d+\alpha}{2q_{i-1}}}. 
\end{align*}
As a consequence,
\begin{align*}
\prod_{i = 1}^{\infty} c^{\frac{1}{2q_{i-1}}}q_{i-1}^{\frac{\gamma}{2q_{i-1}}} R^{-\frac{\alpha}{2q_{i-1}}} 2^{\frac{d+\eps+2}{2 q_{i-1}} (i+1)} \hspace{-0.1cm} \le \hspace{-0.1cm} c(q_0,\kappa,d) R^{-\frac{d+\alpha}{2q_k}} 2^{\frac{d+\eps+2}{2q_k} \sum_{i=0}^{\infty} \frac{i+k+2}{\kappa^i}} \le \hspace{-0.1cm} c(q_0,\kappa,d) R^{-\frac{d+\alpha}{2q_0}} 2^{\frac{(d+\eps+2)c_5}{2q_0}},
\end{align*}
\begin{align*}
\sum_{i = 1}^{\infty} &\left( \prod_{j = i}^{\infty} c^{\frac{1}{2q_{j-1}}} q_{j-1}^{\frac{\gamma}{2q_{j-1}}} R^{-\frac{\alpha}{2q_{j-1}}} 2^{\frac{d+\eps+2}{2 q_{j-1}} (j+1)}  \right) R^{\frac{d+\alpha}{2q_{i-1}}} 2^{-(i+1)\eps} &\le c \sum_{i = 1}^{\infty}  \left( 2^{\frac{(d+\eps + 2)c_5}{2q_{i-1}}(i+1)} \right) 2^{-(i+1)\eps}\\
&\le  c \left( 2^{\frac{(d+\eps+2)c_6}{2q_0} } \right) \sum_{i = 1}^{\infty} 2^{-(i+1)\eps} \le c(d,q_0,\kappa,\eps),
\end{align*}
where we used that $\frac{i+1}{\kappa^{i-1}} \le c_6$ is bounded from above by some constant $c_6 = c_6(\kappa)$. 
Therefore, choosing $\eps = 1$ and $q_0 = 1$, we deduce that for some $c > 0$:
\begin{equation*}
\sup_{I_{R/2}^{\ominus} \times B_{R/2}} \U \le c \left(\dashint_{I_R^{\ominus}}\dashint_{B_R} \U^{2}(t,x) \d x \d t  \right)^{\frac{1}{2}} + c \sup_{t \in I_{R/4}^{\ominus}} \widehat{\tail}_{K,\alpha}(u(t),R).
\end{equation*}
As a consequence, by using the definition of $\U$, as well as triangle inequality for the $L^{2}$-norm, we deduce
\begin{align*}
\sup_{I_{R/2}^{\ominus} \times B_{R/2}} u \le c \left(\dashint_{I_R^{\ominus}}\dashint_{B_R} u^{2}(t,x) \d x \d t \right)^{\frac{1}{2}} + c \sup_{t \in I_{R/4}^{\ominus}} \widehat{\tail}_{K,\alpha}(u(t),R) + c R^{\alpha} \Vert f \Vert_{L^{\infty}}.
\end{align*}
This proves the desired result.
\end{proof}

\section{Local tail estimate}
\label{sec:LTE}

In this section, local tail estimates for supersolutions to \eqref{PDE} and \eqref{PDEdual} (see \autoref{cor:elltailest}), as well as the corresponding stationary equations \eqref{ellPDE} and \eqref{ellPDEdual}  (see \autoref{cor:partailest}) are established. The main auxiliary results are \autoref{lemma:LTE} and \autoref{lemma:LTEdual} whose proofs use similar ideas as in \autoref{lemma:Cacc} and \autoref{lemma:Caccdual}. A central ingredient in the proof are the assumptions \eqref{UJS} and \eqref{UJSdual}, which allow us to derive local tail estimates without having to assume a pointwise lower bound of the jumping kernel. They are applied in a similar way as in \cite{Sch20}, where symmetric nonlocal operators are considered. 

\begin{lemma}
\label{lemma:LTE}
Assume that \eqref{K1}, \eqref{cutoff} and \eqref{UJS} hold true for some $\theta \in [\frac{d}{\alpha},\infty]$. Moreover, assume \eqref{Sob} if $\theta < \infty$. Then there exist $c_1,c_2 > 0$ such that for every $0 < \rho \le r \le 1$, every nonnegative function $u \in V(B_{2r}|\R^d)$ and every $S > 0$ with $S \ge \sup_{B_{r+\rho}} u$, it holds
\begin{align*}
\tail_K(u,r,r+\rho) \le c_1 \frac{1}{S \rho^d}\cE(u,\tau^2(u-2S)) + c_2 \left(\frac{r+\rho}{\rho}\right)^{d}\rho^{-\alpha} S,
\end{align*}
where $B_{2r} \subset \Omega$, $\tau = \tau_{r,\rho}$.
\end{lemma}

\begin{proof}

We define $w = u - 2S$. Note that by definition, $w \in [-2S,-S]$ in $B_{r+\rho}$.
We separate the proof into several steps.

Step 1: First, we claim that for some $c > 0$ it holds
\begin{equation}
\label{eq:LTEstep1}
\cEs_{B_{r+\rho}}(\tau w, \tau w) \le \cEs_{B_{r+\rho}} (u,\tau^2 w) + c S^2 (r+\rho)^d\rho^{-\alpha}.
\end{equation}
We compute
\begin{align*}
\cEs_{B_{r+\rho}} (u,\tau^2 w) &= \int_{B_{r+\rho}}\int_{B_{r+\rho}} (w(x) - w(y))(\tau^2 w(x) - \tau^2 w(y)) K_s(x,y) \d y \d x \\
&= \cEs_{B_{r+\rho}}(\tau w,\tau w) - \int_{B_{r+\rho}}\int_{B_{r+\rho}} w(x)w(y)(\tau(x)-\tau(y))^2 K_s(x,y)\d y \d x.
\end{align*}
We estimate using \eqref{cutoff}
\begin{align*}
\int_{B_{r+\rho}}&\int_{B_{r+\rho}} w(x)w(y)(\tau(x)-\tau(y))^2 K_s(x,y)\d y \d x \le 4S^2 \cE^{K_s}_{B_{r+\rho}}(\tau,\tau) \le c_1S^2 (r+\rho)^d\rho^{-\alpha}
\end{align*}
for some $c_1 > 0$, which directly implies \eqref{eq:LTEstep1}.

Step 2: Next, we claim that there exists $c > 0$ such that
\begin{equation}
\label{eq:LTEstep2}
-\cEa_{B_{r+\rho}}(u,\tau^2 w) \le \frac{1}{2}\cEs_{B_{r+\rho}}(\tau w, \tau w)  + cS^2(r+\rho)^d \rho^{-\alpha}.
\end{equation}
For the proof, we use the same arguments as in the proof of the Caccioppoli estimate:
\begin{align*}
-\cEa_{B_{r+\rho}}(u,\tau^2 w) &= \int_{B_{r+\rho}}\int_{B_{r+\rho}} (w(y)-w(x))(\tau^2 w(x)+\tau^2 w(y)) K_a(x,y) \d y \d x\\
&= \int_{B_{r+\rho}}\int_{B_{r+\rho}} (\tau w(y)-\tau w(x)) (\tau w(y)+\tau w(x)) K_a(x,y) \d y \d x\\
& +\int_{B_{r+\rho}}\int_{B_{r+\rho}} w(x)w(y)(\tau^2(x)-\tau^2(y)) K_a(x,y) \d y \d x\\
&=: J_1 + J_2.
\end{align*}
Using H\"older's and Young's inequality, as well as \eqref{K1} and \eqref{eq:K1consequence}, we obtain that for every $\delta > 0$
\begin{align*}
J_1 &\le \delta \cE^J_{B_{r+\rho}}(\tau w, \tau w)  + c_2\int_{B_{r+\rho}}\int_{B_{r+\rho}} (\tau w(y)+\tau w(x))^2 \frac{\vert K_a(x,y)\vert^2}{J(x,y)} \d y \d x\\
&\le c \delta \cEs_{B_{r+\rho}}(\tau w, \tau w)  + 2c_2\int_{B_{r+\rho}} \tau^2 w^2(x) \left( \int_{B_{r+\rho}}\frac{\vert K_a(x,y)\vert^2}{J(x,y)} \d y\right)\d x\\
&\le 2c \delta \cEs_{B_{r+\rho}}(\tau w, \tau w) + c_3 S^2 (r+\rho)^{d} \rho^{-\alpha}
\end{align*}
for $c_2, c_3 > 0$ depending on $\delta$. Again, by H\"older's and Young's inequality, as well as \eqref{K1}, \eqref{cutoff}, and \eqref{eq:K1consequence}, we estimate
\begin{align*}
J_2 &\le \frac{1}{2} \int_{B_{r+\rho}}\int_{B_{r+\rho}} \vert w(x)\vert \vert w(y)\vert (\tau(y)-\tau(x))^2 J(x,y) \d y \d x\\
&+ \frac{1}{2}\int_{B_{r+\rho}}\int_{B_{r+\rho}} \vert w(x)\vert \vert w(y)\vert (\tau(y)+\tau(x))^2 \frac{\vert K_a(x,y)\vert^2}{J(x,y)} \d y \d x\\
&\le 2S^2 \cE^J_{B_{r+\rho}}(\tau,\tau) + 8S^2\int_{B_{r+\rho}}\left( \tau^2(x) \int_{B_{r+\rho}} \frac{\vert K_a(x,y)\vert^2}{J(x,y)} \d y \right)\d x\\
&\le c_4 S^2 \cEs_{B_{r+\rho}}(\tau,\tau) + c_4 S^2 (r+\rho)^{d} \rho^{-\alpha}\\
&\le c_5 S^2 (r+\rho)^{d} \rho^{-\alpha}
\end{align*}
for $c_4, c_5 > 0$. From here, \eqref{eq:LTEstep2} directly follows.

Step 3: We claim that there exist constants $c, c' >0$ such that
\begin{equation}
\label{eq:LTEstep3}
-\cE_{(B_{r+\rho} \times B_{r+\rho})^{c}}(u,\tau^2 w) \le c S^2(r+\rho)^d\rho^{-\alpha} - c' S \rho^d \tail_K(u,r,r+\rho).
\end{equation}
First, we rewrite the term on the left hand side of the above line
\begin{align}
\label{eq:LTEhelp2}
\begin{split}
-\cE_{(B_{r+\rho} \times B_{r+\rho})^{c}}&(u,\tau^2 w) = -2\int\int_{(B_{r+\rho} \times B_{r+\rho})^{c}}(u(x)-u(y))\tau^2 w(x) K(x,y)\d y \d x\\
&= -2\int_{B_{r+\rho}}\int_{B_{r+\rho}^{c} \cap \lbrace u(y) \ge S \rbrace}\hspace{-0.2cm} (u(y)-u(x))\tau^2(x) (2S-u(x)) K(x,y)\d y \d x\\
& +2 \int_{B_{r+\rho}}\int_{B_{r+\rho}^{c} \cap \lbrace u(y) \le S \rbrace}(u(x)-u(y))\tau^2(x) (2S-u(x)) K(x,y)\d y \d x\\
&=: I_1 + I_2.
\end{split}
\end{align}
For $I_2$ we obtain:
\begin{align*}
I_2 &\le 4S\int_{B_{r+\rho}}\int_{B_{r+\rho}^{c} \cap \lbrace u(y) \le S \rbrace}(u(x)-u(y))_+\tau^2(x) K(x,y)\d y \d x\\
&\le 8S^2 \int_{B_{r+\rho}}\int_{B_{r+\rho}^{c}} (\tau(x)-\tau(y))^2 K(x,y) \d y \d x\\
&\le 8 S^2\int_{B_{r+\rho}}\Gamma^{K_s}(\tau,\tau)(x) \d x\\
&\le c_6 S^2 (r+\rho)^d \rho^{-\alpha}
\end{align*}
for some $c_6 > 0$, where we used \eqref{eq:KaKs}, \eqref{cutoff} and that $K \ge 0$, $u \ge 0$ globally. We treat $I_1$ in the following way (see \cite{Sch20})
\begin{align*}
I_1 &\le -2S\int_{B_{r+\rho}}\int_{B_{r+\rho}^{c} \cap \lbrace u(y) \ge S \rbrace} (u(y)-S)\tau^2(x) K(x,y)\d y \d x\\
&\le -2S\int_{B_{r+\rho}}\int_{B_{r+\rho}^{c}} (u(y)-S)\tau^2(x) K(x,y)\d y \d x\\
&\le -2S\int_{B_{r + \frac{\rho}{4}}}\int_{B_{r+\rho}^{c}} u(y)\tau^2(x) K(x,y)\d y \d x + 2S^2\int_{B_{r+\rho}}\int_{B_{r+\rho}^{c}} (\tau(x)-\tau(y))^2 K(x,y)\d y \d x\\
&\le -\frac{S}{8}\int_{B_{r + \frac{\rho}{4}}}\int_{B_{r+\rho}^{c}} u(y) K(x,y)\d y \d x + c_7 S^2 (r+\rho)^d \rho^{-\alpha}
\end{align*}
for some $c_7 > 0$, where we used that $u,K \ge 0$, $u \le S$ in $B_{r+\rho}$, $\tau^2 \ge \frac{1}{16}$ in $B_{r+\frac{\rho}{4}}$, \eqref{eq:KaKs} and \eqref{cutoff}.
Finally, note that due to \eqref{UJS}
\begin{align}
\label{eq:LTEhelp3}
\begin{split}
\rho^d \tail_K(u,r,r+\rho) &= \rho^d \sup_{x \in B_{r}} \int_{B_{r+\rho}^{c}} u(y) K(x,y) \d y\\
&\le c_8 \sup_{x \in B_r} \int_{B_{r+\rho}^{c}} u(y) \left( \int_{B_{\frac{\rho}{4}}(x)} K(z,y) \d z\right) \d y\\
&\le c_8 \int_{B_{r+\rho}^{c}} u(y) \left(\int_{B_{r+\frac{\rho}{4}}} K(x,y) \d x\right) \d y\\
&= c_8 \int_{B_{r + \frac{\rho}{4}}}\int_{B_{r+\rho}^{c}} u(y) K(x,y)\d y \d x
\end{split}
\end{align}
for some $c_8 > 0$. Consequently, 
\begin{equation*}
I_1 \le -c_9 S \rho^d \tail_K(u,r,r+\rho) + c_{10} S^2 (r+\rho)^d \rho^{-\alpha},
\end{equation*}
where $c_9, c_{10} > 0$ are constants. 

Step 4: Now, we want to combine \eqref{eq:LTEstep1}, \eqref{eq:LTEstep2} and \eqref{eq:LTEstep3}.
First, we observe that
\begin{equation*}
\cEs_{B_{r+\rho}}(u,\tau^2 w) = \cE(u,\tau^2 w) - \cE_{(B_{r+\rho}\times B_{r+\rho})^{c}}(u,\tau^2 w) -\cEa_{B_{r+\rho}}(u,\tau^2 w).
\end{equation*}
Together, we obtain
\begin{align*}
\cEs_{B_{r+\rho}}(\tau w, \tau w)  \le \cE(u,\tau^2 w) + c_{11} S^2(r+\rho)^d\rho^{-\alpha} - c_{12} S \rho^{d} \tail_K(u,r,r+\rho) + \frac{1}{2}\cEs_{B_{r+\rho}}(\tau w, \tau w) 
\end{align*}
for $c_{11}, c_{12} > 0$. Since $L \ge 0$, we conclude
\begin{align*}
\tail_K(u,r,r+\rho) \le c_{13}\frac{1}{S \rho^d}\cE(u,\tau^2 w) + c_{14} S \left(\frac{r+\rho}{\rho}\right)^d \rho^{-\alpha} ,
\end{align*}
where $c_{13}, c_{14} > 0$ are constants. This yields the desired result.
\end{proof}

Next, we prove a similar estimate for the dual form:

\begin{lemma}
\label{lemma:LTEdual}
Assume that \eqref{K1glob}, \eqref{cutoff} and \eqref{UJSdual} hold true for some $\theta \in [\frac{d}{\alpha},\infty]$. Moreover, assume \eqref{Sob} if $\theta < \infty$. Then there exist $c_1,c_2 > 0$ such that for every $0 < \rho \le r \le 1$, every nonnegative function $u \in V(B_{2r}|\R^d) \cap L^{2\theta'}(\R^d)$ and every $S \ge \sup_{B_{r+\rho}} u$, it holds
\begin{align*}
\widehat{\tail}_K(u,r,r+\rho) \le c_1 \frac{1}{S \rho^d}\widehat{\cE}(u,\tau^2(u-2S)) + c_2 \left(\frac{r+\rho}{\rho}\right)^{d}\rho^{-\alpha} S,
\end{align*}
where $B_{2r} \subset \Omega$, $\tau = \tau_{r,\rho}$.
\end{lemma}

\begin{proof}

As in the proof of \autoref{lemma:LTE}, we define $w = u - 2S$ and observe that $w \in [-S,-2S]$ in $B_{r+\rho}$. 
The proof is separated into several steps.

Step 1: First, we recall from the Step 1 in the proof of \autoref{lemma:LTE} that for some $c > 0$ it holds
\begin{equation}
\label{eq:LTEdualstep1}
\cEs_{B_{r+\rho}}(\tau w, \tau w) \le \cEs_{B_{r+\rho}} (u,\tau^2 w) + c S^2 (r+\rho)^d\rho^{-\alpha}.
\end{equation}

Step 2: In analogy with Step 2 in the proof of \autoref{lemma:LTE}, we claim that for some $c > 0$:
\begin{equation}
\label{eq:LTEdualstep2}
-\widehat{\cE}^{K_a}(u,\tau^2 w) \le \frac{1}{2} \cE^{K_s}_{B_{r+\rho}}(\tau w,\tau w) + cS^2(r+\rho)^d \rho^{\alpha}.
\end{equation}
To see this, we estimate
\begin{align*}
-\widehat{\cE}^{K_a}(u,\tau^2 w) &= \int_{B_{r+\rho}}\int_{B_{r+\rho}} (\tau^2 w(x) - \tau^2 w(y))(w(x)+w(y)) K_a(x,y) \d y \d x\\
&+ 4S \int_{B_{r+\rho}}\int_{B_{r+\rho}} (\tau^2 w(x) - \tau^2 w(y)) K_a(x,y) \d y \d x\\
&:= I_1 + I_2.
\end{align*}
For $I_1$, we compute
\begin{align*}
I_1 &= \int_{B_{r+\rho}}\int_{B_{r+\rho}} (\tau^2w^2(x) - \tau^2 w^2(y)) K_a(x,y) \d y \d x\\
&+ \int_{B_{r+\rho}}\int_{B_{r+\rho}}w(x)w(y)(\tau^2(x)-\tau^2(y))K_a(x,y) \d y \d x,
\end{align*}
and from the same arguments as in the proof of Step 2 from the proof of \autoref{lemma:LTE}, we conclude
\begin{equation*}
I_1 \le \frac{1}{4}\cE^{K_s}_{B_{r+\rho}}(\tau w,\tau w) + cS^2(r+\rho)^d\rho^{-\alpha},
\end{equation*}
using \eqref{K1glob}, \eqref{cutoff}.
For $I_2$ we observe:
\begin{align*}
I_2 &= 2S \int_{B_{r+\rho}}\int_{B_{r+\rho}}(\tau w(x) - \tau w(y))(\tau(x) + \tau(y))K_a(x,y) \d y \d x\\
&+ 2S \int_{B_{r+\rho}}\int_{B_{r+\rho}} (\tau w(x) + \tau w(y))(\tau(x) - \tau(y))K_a(x,y) \d y \d x\\
&=: I_{2,1} + I_{2,2}.
\end{align*}
Now, using \eqref{K1glob}, \eqref{eq:K1consequence} and \eqref{cutoff},
\begin{align*}
I_{2,1} &\le \frac{1}{8}\cE^{K_s}_{B_{r+\rho}}(\tau w,\tau w) + cS^2 \int_{B_{r+\rho}} \tau^2(x) \left(\int_{B_{r+\rho}} \frac{\vert K_a(x,y)\vert^2}{J(x,y)} \d y\right) \d x\\
&\le \frac{1}{8}\cE^{K_s}_{B_{r+\rho}}(\tau w,\tau w) + cS^2 \cE^{K_s}_{B_{r+\rho}}(\tau,\tau) + c S^2 \rho^{-\alpha} \int_{B_{r+\rho}} \tau^2(x) \d x\\
&\le \frac{1}{8}\cE^{K_s}_{B_{r+\rho}}(\tau w,\tau w) + cS^2 (r+\rho)^d \rho^{-\alpha},
\end{align*}
and, again using \eqref{K1glob}, \eqref{eq:K1consequence} and \eqref{cutoff}:
\begin{align*}
I_{2,2} &\le cS^2 \cE^{K_s}_{B_{r+\rho}}(\tau,\tau) + \int_{B_{r+\rho}}\tau^2 w^2(x) \left(\int_{B_{r+\rho}} \frac{\vert K_a(x,y)\vert^2}{K_s(x,y)} \d y\right) \d x\\
&\le cS^2 (r+\rho)^d \rho^{-\alpha} + \frac{1}{8} \cE^{K_s}_{B_{r+\rho}}(\tau w,\tau w) + c \rho^{-\alpha} \int_{B_{r+\rho}} \tau^2 w^2(x) \d x\\
&\le \frac{1}{8} \cE^{K_s}_{B_{r+\rho}}(\tau w,\tau w) + cS^2 (r+\rho)^d \rho^{-\alpha}.
\end{align*}
Altogether, we have proved \eqref{eq:LTEdualstep2}.

Step 3: 
Moreover, we claim that for some constants $c,c' > 0$:
\begin{equation}
\label{eq:LTEdualstep3}
-\widehat{\cE}_{(B_{r+\rho} \times B_{r+\rho})^c}(u,\tau^2 w) \le cS^2 (r+\rho)^d  \rho^{-\alpha}  -c'S \rho^d \widehat{\tail}_K(u,r,r+\rho).
\end{equation}
First, we decompose
\begin{align*}
-\cE_{(B_{r+\rho} \times B_{r+\rho})^c}(\tau^2 w, u) &= -2 \int_{B_{r+\rho}} \int_{B_{r+\rho}^c} \tau^2 w(x) u(x) K(x,y) \d y \d x\\
&+ 2 \int_{B_{r+\rho}^c} \int_{B_{r+\rho}} \tau^2 w(y) u(x) K(x,y) \d y \d x\\
&=: J_1 + J_2.
\end{align*}
For $J_1$, we compute using the definition of $w$, nonnegativity of $u$ and \eqref{eq:KaKs}
\begin{align*}
J_1 &= 2 \int_{B_{r+\rho}} \int_{B_{r+\rho}^c} \tau^2(x) (2S-u(x)) u(x) K(x,y) \d y \d x\\
&\le 4S^2 \int_{B_{r+\rho}} \int_{B_{r+\rho}^c} (\tau(x) - \tau(y))^2 K_s(x,y) \d y \d x\\
&\le c S^2 (r+\rho)^2 \rho^{-\alpha}.
\end{align*}
For $J_2$, we observe, using that $\tau^2 \ge 1/16$ in $B_{r + \frac{\rho}{4}}$:
\begin{align*}
J_2 &= 2 \int_{B_{r+\rho}^c} \int_{B_{r+\rho}} \tau^2(y) (u(y) - 2S) u(x) K(x,y) \d y \d x\\
&\le -2S \int_{B_{r+\rho}} \int_{B_{r+\rho}^c} \tau^2(y) u(x) K(x,y) \d x \d y\\
&\le -\frac{S}{8} \int_{B_{r+\frac{\rho}{4}}} \int_{B_{r+\rho}^c} u(x) K(x,y) \d x \d y.
\end{align*}
Finally, using \eqref{UJSdual} and the same argument as in \eqref{eq:LTEhelp3}, we can prove that
\begin{align*}
\rho^d \widehat{\tail}_K(u,r,r+\rho) \le c \int_{B_{r+\frac{\rho}{4}}} \int_{B_{r+\rho}^c} u(x) K(x,y) \d x \d y.
\end{align*}
Altogether, we have established \eqref{eq:LTEdualstep3}, as desired.

Step 4:
Combining \eqref{eq:LTEdualstep1}, \eqref{eq:LTEdualstep2}, and \eqref{eq:LTEdualstep3}, we obtain:
\begin{align*}
\cE^{K_s}_{B_{r+\rho}}(\tau w,\tau w) &\le \cE^{K_s}_{B_{r+\rho}}(u,\tau^2 w) + cS^2 (r+\rho)^d \rho^{-\alpha}\\
&= \widehat{\cE}(u,\tau^2 w) - \widehat{\cE}^{K_a}_{B_{r+\rho}}(u,\tau^2 w) - \widehat{\cE}_{(B_{r+\rho} \times B_{r+\rho})^c}(u,\tau^2 w) + cS^2 (r+\rho)^d \rho^{-\alpha}\\
&\le \widehat{\cE}(u,\tau^2 w) + cS^2 (r+\rho)^d \rho^{-\alpha} + \frac{1}{2} \cE^{K_s}_{B_{r+\rho}}(\tau w,\tau w) - cS \rho^d \widehat{\tail}_K(u,r,r+\rho).
\end{align*}
Consequently,
\begin{equation*}
\widehat{\tail}_K(u,r,r+\rho) \le c \frac{1}{S \rho^d}\widehat{\cE}(u,\tau^2 w) + c \left(\frac{r+\rho}{\rho}\right)^{d}\rho^{-\alpha} S,
\end{equation*}
as desired.
\end{proof}

\autoref{lemma:LTE} can be used to bound $\tail_K(u,r,r+\rho)$ from above by the supremum of $u$. First, we provide such estimate for weak supersolutions to the stationary equations \eqref{ellPDE} and \eqref{ellPDEdual}, which is a direct corollary of \autoref{lemma:LTE} applied with $S = \sup_{B_{r+\rho}} u$.

\begin{corollary}
\label{cor:elltailest}
Assume that \eqref{cutoff} holds true.
\begin{itemize}
\item[(i)] Assume \eqref{K1}, \eqref{UJS} holds true for some $\theta \in [\frac{d}{\alpha},\infty]$. Moreover, assume \eqref{Sob} if $\theta < \infty$. Then there exists $c > 0$ such that for every $0 < \rho \le r \le 1$ and every nonnegative, weak supersolution $u$ to \eqref{ellPDE} in $B_{2r}$, it holds
\begin{align*}
\tail_K(u,r,r+\rho) \le c \left(\frac{r+\rho}{\rho}\right)^{d}\left(\rho^{-\alpha} \sup_{B_{r+\rho}} u + \Vert f \Vert_{L^{\infty}}\right),
\end{align*}
where $B_{2r} \subset \Omega$.
\item[(ii)] Assume \eqref{K1glob}, \eqref{UJSdual} holds true for some $\theta \in [\frac{d}{\alpha},\infty]$. Moreover, assume \eqref{Sob} if $\theta < \infty$. Then there exists $c > 0$ such that for every $0 < \rho \le r \le 1$ and every nonnegative, weak subsolution $u$ to \eqref{ellPDEdual} in $B_{2r}$, it holds
\begin{align*}
\widehat{\tail}_K(u,r,r+\rho) \le c \left(\frac{r+\rho}{\rho}\right)^{d}\left(\rho^{-\alpha} \sup_{B_{r+\rho}} u + \Vert f \Vert_{L^{\infty}}\right),
\end{align*}
where $B_{2r} \subset \Omega$.
\end{itemize}
\end{corollary}

One can also deduce an estimate for the $L^1$-parabolic tail $\int_{I_{r/2}^{\ominus}}\tail_K(u(t),r,r+\rho) \d t$ for supersolutions to \eqref{PDE}, \eqref{PDEdual} from \autoref{lemma:LTE}. 

\begin{corollary}
\label{cor:partailest}
Assume that \eqref{cutoff} holds true.
\begin{itemize}
\item[(i)] Assume \eqref{K1}, \eqref{UJS} holds true for some $\theta \in [\frac{d}{\alpha},\infty]$. Moreover, assume \eqref{Sob} if $\theta < \infty$. Then there exists $c > 0$ such that for every $0 < \rho \le r \le 1$ and every nonnegative, weak supersolution $u$ to \eqref{PDE} in $I_{r}^{\ominus}(t_0) \times B_{2r}$, it holds
\begin{align*}
\int_{I_{r/2}^{\ominus}}\tail_K(u(t),r,r+\rho) \d t \le c \left(\frac{r+\rho}{\rho}\right)^{d} \left(\left(\frac{r+\rho}{\rho}\right)^{\alpha \vee 1}\hspace{-0.3cm} \sup_{I_{(r+\rho)/2}^{\ominus} \times B_{r+\rho}} \hspace{-0.3cm} u + (r+\rho)^{\alpha}\Vert f \Vert_{L^{\infty}}\right),
\end{align*}
where $B_{2r} \subset \Omega$.
\item[(ii)] Assume \eqref{K1glob}, \eqref{UJSdual} holds true for some $\theta \in [\frac{d}{\alpha},\infty]$. Moreover, assume \eqref{Sob} if $\theta < \infty$. Then there exists $c > 0$ such that for every $0 < \rho \le r \le 1$ and every nonnegative, weak supersolution $u$ to \eqref{PDEdual} in $I_{r}^{\ominus}(t_0) \times B_{2r}$, it holds
\begin{align*}
\int_{I_{r/2}^{\ominus}}\widehat{\tail}_K(u(t),r,r+\rho) \d t \le c \left(\frac{r+\rho}{\rho}\right)^{d} \left(\left(\frac{r+\rho}{\rho}\right)^{\alpha \vee 1}\hspace{-0.3cm} \sup_{I_{(r+\rho)/2}^{\ominus} \times B_{r+\rho}} \hspace{-0.3cm} u + (r+\rho)^{\alpha}\Vert f \Vert_{L^{\infty}}\right),
\end{align*}
where $B_{2r} \subset \Omega$.
\end{itemize}
\end{corollary}

\begin{proof}
We only explain the proof of (i). The proof of (ii) works in the same way, but relies on \autoref{lemma:LTEdual} instead of \autoref{lemma:LTE}.
We write $S = \sup_{I^{\ominus}_{(r+\rho)/2} \times B_{r+\rho}} u$ and denote $w = u-2S$, observing that $\partial_t (w^2) = 2 w \partial_t u$. From \autoref{lemma:LTE} and the fact that $u$ is a supersolution to \eqref{PDE}, we deduce
\begin{align*}
\frac{c}{2S\rho^d}\int_{B_{r+\rho}}& \tau^2(x) \partial_t (w^2)(t,x) \d x + \tail_K(u(t),r,r+\rho)\\
&\le c\frac{1}{S\rho^d}\left[(\partial_t u(t) , \tau^2w(t)) +\cE(u(t),\tau^2 w(t))\right] + cS \left( \frac{r+\rho}{\rho} \right)^d \rho^{-\alpha}\\
&\le c\frac{1}{S\rho^d}(f(t),\tau^2w(t)) + cS \left( \frac{r+\rho}{\rho} \right)^d \rho^{-\alpha}\\
&\le c \left( \frac{r+\rho}{\rho} \right)^d \left(\Vert f \Vert_{L^{\infty}} + S  \rho^{-\alpha} \right),
\end{align*}
where $c > 0$ is the constant from \autoref{lemma:LTE} and we tested the equation with $\tau^2w$, $\tau = \tau_{r,\rho}$. Let $\chi \in C^1(\R)$ be a nonnegative function with $\chi(t_0 - ((r+\rho)/2)^{\alpha}) = 0$, $\chi \equiv 1$ in $I_{r/2}^{\ominus}$, $\Vert \chi \Vert_{\infty} \le 1$, $\Vert \chi' \Vert_{\infty} \le 8((r+\rho)^{\alpha} - r^{\alpha})^{-1}$. Multiplying with $\chi^2$ and integrating over $(t_0 - ((r+\rho)/2)^{\alpha} , t)$ for some arbitrary $t \in I_{r/2}^{\ominus}$, we obtain
\begin{align*}
\frac{c}{2S\rho^d}&\int_{B_{r+\rho}} \chi^2(t) \tau^2(x) w^2(t,x) \d x + \int_{t_0 - ((r+\rho)/2)^{\alpha}}^t \chi^2(s) \tail_K(u(s),r,r+\rho) \d s\\
&\le c_1\int_{t_0 - ((r+\rho)/2)^{\alpha}}^t \chi^2(s) S \left( \frac{r+\rho}{\rho} \right)^d \rho^{-\alpha} \d s + c_1 \left( \frac{r+\rho}{\rho} \right)^d (r+\rho)^{\alpha} \Vert f \Vert_{L^{\infty}}\\
&+ c_1\int_{t_0 - ((r+\rho)/2)^{\alpha}}^t \frac{1}{S\rho^d} \chi(s)\vert \chi'(s) \vert \int_{B_{r+\rho}} \tau^2(x) w^2(s,x) \d x \d s,
\end{align*}
where $c_1 > 0$ is a constant. Consequently, using that $w^2 \le 4S^2$
\begin{align*}
\sup_{t \in I_{r/2}^{\ominus}}\frac{c}{2S\rho^d}\int_{B_r} w^2(t,x) \d x &+ \int_{I_{r/2}^{\ominus}} \tail_K(u(s),r,r+\rho) \d s\\
&\le c_2 \left(\frac{r+\rho}{\rho}\right)^{d+(\alpha \vee 1)} \sup_{I_{(r+\rho)/2}^{\ominus} \times B_{r+\rho}} u + c_2\left( \frac{r+\rho}{\rho} \right)^d (r+\rho)^{\alpha} \Vert f \Vert_{L^{\infty}},
\end{align*}
where $c_2 > 0$ and we used that for some $c > 0$, $((r+\rho)^{\alpha} - r^{\alpha})^{-1} \le c\rho^{-(\alpha \vee 1)}(r+\rho)^{(\alpha \vee 1) - \alpha}$. This concludes the proof.
\end{proof}

\section{Harnack inequalities}
\label{sec:FHI}

The goal of this section is to complete the proofs of our main results \autoref{thm:apriorifHI} and \autoref{thm:fHI}. In \autoref{sec:improvedLB}, we give  improved versions of the local boundedness estimates from \autoref{sec:Cacc} and \autoref{sec:CaccMoser}, which do not involve tail terms. These results make use of the tail estimates obtained in \autoref{cor:elltailest} and are the key ingredient to the proof of \autoref{thm:fHI}. In \autoref{sec:proofs} we combine local boundedness estimates with the weak Harnack inequalities from \cite{KaWe22} and obtain our main results.\\
We point out that the proof of \autoref{thm:apriorifHI} does not rely on the tail estimates from \autoref{sec:LTE}. It is an open question -- even in the symmetric case -- how to derive a parabolic Harnack inequality involving only local quantities from suitable tail estimates, as one does in the stationary case. \autoref{sec:parabolicFHI} is dedicated to this issue.

\subsection{Local boundedness without tail terms}
\label{sec:improvedLB}

We obtain local $L^{\infty}-L^{p}$-estimates for solutions to \eqref{ellPDE} and \eqref{ellPDEdual} (see \autoref{thm:LB3}). In comparison with \autoref{thm:LB1} the estimates only contain purely local quantities. The underlying procedure works exactly as for symmetric forms. However, note that we need to redo the iteration in \autoref{thm:LB1} in order to prove \autoref{thm:LB2} since the quantities $\tail_K$ and $\tail_{K,\alpha}$ are in general not comparable.

The following theorem is the key result on our path towards $L^{\infty}-L^{p}$-estimates for nonnegative solutions to \eqref{ellPDE} and \eqref{ellPDEdual} since it does not involve any nonlocal quantities anymore.

\begin{theorem}
\label{thm:LB2}
Assume that \eqref{cutoff} and \eqref{Sob} hold true.
\begin{itemize}
\item[(i)] Assume that \eqref{K1} and \eqref{UJS} hold true for some $\theta \in [\frac{d}{\alpha},\infty]$. Then, for every $\delta \in (0,1]$ there exists $c > 0$ such that for every $0 < R \le 1$ and every nonnegative, weak solution $u$ to \eqref{ellPDE} in $B_{2R} \subset \Omega$, it holds
\begin{align}
\label{eq:LB2}
\sup_{B_{R/2}} u \le c \left( \dashint_{B_R} u^{2}(x) \d x \right)^{1/2} + \delta \sup_{B_{R}} u + c R^{\alpha}\Vert f \Vert_{L^{\infty}} \,.
\end{align}
\item[(ii)] Assume that \eqref{K1glob} and \eqref{UJSdual} hold true for some $\theta \in (\frac{d}{\alpha},\infty]$. Then, for every $\delta \in (0,1]$ there exists $c > 0$ such that for every $0 < R \le 1$ and every nonnegative, weak solution $u$ to \eqref{ellPDEdual} in $B_{2R}$, estimate \eqref{eq:LB2} holds true.
\end{itemize}
\end{theorem}

We present two proofs of this theorem based on the De Giorgi iteration and the Moser iteration. Both proofs rely on a combination of the iteration schemes established in \autoref{sec:Cacc} and \autoref{sec:CaccMoser}, and the tail estimate from \autoref{cor:partailest}.

\begin{proof}[Proof of \autoref{thm:LB2} (based on De Giorgi iteration)]
The proof of (i) is analogous to the proof of \autoref{thm:LB1} (i). We define $(l_i)_i, (\rho_i)_i, (r_i)_i, (w_i)_i$ in the same way. Moreover, we set $A_i = \Vert w_i \Vert_{L^1(B_{r_i})}$. Note that
\begin{equation*}
\left(\frac{r_{i}+\rho_{i}}{\rho_{i}} \right)^d = \left(\frac{(1 + \left( \frac{1}{2} \right)^i) + \left( \frac{1}{2} \right)^i}{\left( \frac{1}{2} \right)^{i+1}} \right)^d \le 2^{(i+2)d}.
\end{equation*}
Consequently, \autoref{cor:elltailest} (i) (applied with $r = r_i + \frac{\rho_i}{2}$ and $\rho = \frac{\rho_i}{2}$) yields
\begin{equation}
\label{eq:ellLB2help1}
\tail_K(u,r_{i}+\frac{\rho_{i}}{2},r_{i}+\rho_{i}) \le c_1 2^{i(d+2)}R^{-\alpha} \left( \sup_{B_R} u + R^{\alpha} \Vert f \Vert_{L^{\infty}} \right)
\end{equation}
for some $c_1 > 0$. Moreover, by following the arguments in the proof of \autoref{thm:LB1} (i), we derive the following analog of \eqref{eq:LB1help3}
\begin{align}
\label{eq:ellLB1help3}
A_i \le c_{2} \frac{1}{(l_{i} - l_{i-1})^{\frac{2}{\kappa'}}}\left( \sigma(r_i,\rho_i) +  \frac{\tail_{K}(u,r_i+\frac{\rho_i}{2},r_{i}+\rho_i) + \Vert f\Vert_{L^{\infty}}}{l_{i}-l_{i-1}} \right)A_{i-1}^{1+\frac{1}{\kappa'}}
\end{align} 
for some $c_2 > 0$, where we can choose $\kappa = \frac{d}{d-\alpha}$ by using that $u$ is a subsolution to the stationary equation \eqref{ellPDE} in \eqref{eq:LB1prep2}.
We combine \eqref{eq:ellLB2help1} and \eqref{eq:ellLB1help3} and obtain
\begin{align*}
A_{i} \le \frac{c_3}{R^{\alpha} M^{\frac{2}{\kappa'}}} 2^{\gamma i} \left(1 + \frac{\sup_{B_{R}} u + R^{\alpha} \Vert f \Vert_{L^{\infty}}}{M}\right)A_{i-1}^{1+\frac{1}{\kappa'}},
\end{align*}
where $c_3 > 0$ and $\gamma > 1$ are constants. We proceed as in the proof of \autoref{thm:LB1} (i) and choose $M:= \delta \left(\sup_{B_R} u + \Vert f \Vert_{L^{\infty}} \right) + C^{\frac{\kappa'^2}{2}}c_3^{\frac{\kappa'}{2}}\delta^{-\frac{\kappa'}{2}}R^{-\frac{\alpha \kappa'}{2}}A_0^{1/2}$, where $C:= 2^{\gamma} > 1$ and conclude
\begin{align*}
A_0 \le \left(\frac{c_3}{\delta R^{\alpha}M^{\frac{2}{\kappa'}}}\right)^{-\kappa'}C^{-\kappa'^2},
\end{align*}
and therefore we obtain from Lemma 7.1 in \cite{Giu03}
\begin{align*}
\sup_{B_{R/2}} u \le M = \delta \left( \sup_{B_R} u + R^{\alpha} \Vert f \Vert_{L^{\infty}}\right) + c_3 \delta^{-\frac{\kappa'}{2}}\left(\dashint_{B_R} u^2(x)\d x\right)^{1/2}
\end{align*}
for some $c_3 > 0$, as desired.\\
In order to prove (ii), we follow the arguments in the proof of \autoref{thm:LB1} (ii) and derive the following analog of \eqref{eq:ellLB1help3}
\begin{align*}
A_i \le c_4 \frac{R^{d\left(\frac{1}{\kappa'} - \mu \right)}}{(l_i - l_{i-1})^{2\mu}} \left( \sigma(r_i,\rho_i)\left(1 + \left( \frac{l_i}{l_i - l_{i-1}}\right)^2 \right) +  \frac{\tail_{K}(u,r_i+\frac{\rho_i}{2},r_{i}+\rho_i) + \Vert f\Vert_{L^{\infty}}}{l_{i}-l_{i-1}} \right)A_{i-1}^{1+\mu}
\end{align*}
for some $c_4 > 0$, where $\mu = \frac{1}{\kappa'} - \frac{1}{\theta}$ and $\kappa = \frac{d}{d-\alpha}$. As before, by \autoref{cor:elltailest} (ii) (applied with $r = r_{i} + \frac{\rho_{i}}{2}$ and $\rho = \frac{\rho_{i}}{2}$) we prove
\begin{align}
\label{eq:LB2help1dual}
\widehat{\tail}_K(u,r_{i}+\frac{\rho_{i}}{2},r_{i}+\rho_{i}) \le c_1 2^{i(d+2)}R^{-\alpha} \left(\sup_{B_R} u + R^{\alpha} \Vert f \Vert_{L^{\infty}} \right).
\end{align}
By combining \eqref{eq:LB2help1dual} with the previous estimate, we deduce
\begin{align*}
A_i \le \frac{c_5}{\delta R^{\mu d} M^{2\mu}} C^{\gamma i} A_{i-1}^{1+\mu}
\end{align*}
for some $c_5 > 0$ and $\gamma > 1$. From here, the desired result follows by the same arguments as in the proof of (i).
\end{proof}

\begin{proof}[Proof of \autoref{thm:LB2} (based on Moser iteration)]
We explain how to prove (ii). The proof of (i) follows exactly the same arguments. Our proof is based on the Moser iteration and works in a similar way compared to the proof of \autoref{thm:LB1dual}.
Let us define $(\rho_i)_i$, $(r_i)_i$ and $(q_i)_i$ in the same way, but set $\kappa = \frac{d}{d-\alpha}$. \\
Note that by following the arguments of the proof of \autoref{thm:LB1dual}, but using that $u$ is a subsolution to the stationary equation in \eqref{eq:posMoserItStep}, we can derive the following analog of \eqref{eq:posMoserIthelp1}:
\begin{align*}
\Vert \U \Vert_{L^{q_{i}}(B_{r_{i}})} &\le c^{\frac{1}{2q_{i-1}}}q_{i-1}^{\frac{\gamma}{2q_{i-1}}} R^{-\frac{\alpha}{2q_{i-1}}} 2^{\frac{d+\eps+2}{2 q_{i-1}} (i+1)} \\
&\left( \Vert \U \Vert_{L^{2q_{i-1}}(B_{r_{i-1}})} + 2^{-(d+\eps+\alpha)(i+1)} R^{\alpha+\frac{d}{2q_{i-1}}} \widehat{\tail}_{K}(u,r_i + \frac{\rho_i}{2},r_i+\rho_i) \right).
\end{align*}
By combining this estimate with \eqref{eq:LB2help1dual}, we obtain
\begin{align*}
\Vert \U \Vert_{L^{q_{i}}(B_{r_{i}})} &\le c^{\frac{1}{2q_{i-1}}}q_{i-1}^{\frac{\gamma}{2q_{i-1}}} R^{-\frac{\alpha}{2q_{i-1}}} 2^{\frac{d+\eps+2}{2 q_{i-1}} (i+1)} \\
&\left( \Vert \U \Vert_{L^{2q_{i-1}}(B_{r_{i-1}})} + R^{\frac{d}{2q_{i-1}}} 2^{-(i+1) (\eps +\alpha - 2)} \left(\sup_{B_R} u + R^{\alpha}\Vert f \Vert_{L^{\infty}}\right)\right).
\end{align*}
From here, the proof follows in a similar way as the proof of \autoref{thm:LB1dual}.
First, we observe that:
\begin{align*}
\sup_{B_{R/2}} \U &\le \left( \prod_{i = 1}^{\infty} c^{\frac{1}{2q_{i-1}}}q_{i-1}^{\frac{\gamma}{2q_{i-1}}} R^{-\frac{\alpha}{2q_{i-1}}} 2^{\frac{d+\eps+2}{2 q_{i-1}} (i+1)}  \right) \Vert \U \Vert_{L^{2q_0}(B_{R})}\\
&+ \left[\sum_{i = 1}^{\infty} \left( \prod_{j = i}^{\infty} c^{\frac{1}{2q_{j-1}}} q_{j-1}^{\frac{\gamma}{2q_{j-1}}} R^{-\frac{\alpha}{2q_{j-1}}} 2^{\frac{d+\eps+2}{2 q_{j-1}} (j+1)}  \right) R^{\frac{d}{2q_{i-1}}} 2^{-(i+1)(\eps +\alpha - 2)} \right] \left(\sup_{B_R} u + R^{\alpha}\Vert f \Vert_{L^{\infty}}\right).
\end{align*}
Moreover, by similar arguments as in the proof of \autoref{thm:LB1dual} 
\begin{align*}
\sum_{i = 1}^{\infty} \left( \prod_{j = i}^{\infty} c^{\frac{1}{2q_{j-1}}} q_{j-1}^{\frac{\gamma}{2q_{j-1}}} R^{-\frac{\alpha}{2q_{j-1}}} 2^{\frac{d+\eps+2}{2 q_{j-1}} (j+1)}  \right) R^{\frac{d}{2q_{i-1}}} 2^{-(i+1)(\eps +\alpha - 2)} &\le  c \sum_{i = 1}^{\infty}  \frac{ 2^{\frac{(d+\eps+2)c_5}{2q_0}(i+1) }}{ 2^{(i+1)(\eps +\alpha - 2)}},
\end{align*}
using that $\sum_{i=0}^{\infty} \kappa^{-i} = \frac{d}{\alpha}$ and $\sum_{i=0}^{\infty} \frac{i}{\kappa^i} < \infty$, where $\kappa = \frac{d}{d-\alpha}$.\\
Now, choose $\eps \ge 1$ so large that $\sum_{i = 1}^{\infty} 2^{-(i+1)\frac{\eps+\alpha-2}{2}} \le \frac{\delta}{2c}$. Then, let us choose $q_0 \ge 1$ so large that $\frac{(d+\eps+2)c_5}{2q_0} \le \frac{\eps + \alpha - 2}{2}$. In that case,
\begin{equation*}
c \sum_{i = 1}^{\infty}  \left( 2^{\frac{(d+\eps+2)c_5}{2q_0}(i+1)} \right) 2^{-(i+1)(\eps +\alpha - 2)} \le c \sum_{i = k+1}^{\infty} 2^{-(i+1)\frac{(\eps +\alpha - 2)}{2}} \le \frac{\delta}{2}.
\end{equation*}
Therefore,
\begin{equation*}
\sup_{B_{R/2}} \U \le c \left(\dashint_{B_R} \U^{2q_0}(x) \d x \right)^{\frac{1}{2q_0}} + \frac{\delta}{2} \left(\sup_{ B_R} u + R^{\alpha}\Vert f \Vert_{L^{\infty}}\right).
\end{equation*}
As a consequence, by using the definition of $\U$, as well as triangle inequality for the $L^{2q_0}$-norm, we deduce
\begin{align*}
\sup_{B_{R/2}} u \le c \left(\dashint_{B_R} u^{2q_0}(x) \d x \right)^{\frac{1}{2q_0}} + \frac{\delta}{2} \sup_{B_R} u + c R^{\alpha} \Vert f \Vert_{L^{\infty}}.
\end{align*}
It remains to prove the desired estimate \eqref{eq:LB2} in case $q_0 > 1$. This follows from Young's inequality:
\begin{align*}
\left(\dashint_{B_R} u^{2q_0}(x) \d x \right)^{\frac{1}{2q_0}} \le \sup_{B_R} u ^{\frac{2q_0-2}{2q_0}} \left(\dashint_{B_R} u^{2}(x) \d x \right)^{\frac{1}{2q_0}} \le \frac{\delta}{2c} \sup_{B_R} u + c \left(\dashint_{B_R} u^{2}(x) \d x \right)^{\frac{1}{2}}.
\end{align*}
\end{proof}

By a standard iteration argument one can deduce local boundedness of nonnegative solutions to \eqref{ellPDE} and \eqref{ellPDEdual} from \autoref{thm:LB2}.

\begin{theorem}
\label{thm:LB3}
Assume that \eqref{cutoff} and \eqref{elower} hold true.
\begin{itemize}
\item[(i)] Assume that \eqref{K1}, \eqref{UJS} hold true for some $\theta \in [\frac{d}{\alpha},\infty]$. Then there exists $c > 0$ such that for every $0 < R \le 1$, every $p \in (0,2]$ and every nonnegative, weak solution $u$ to \eqref{ellPDE} in $ B_{2R}$, it holds
\begin{align}
\label{eq:LB3}
\sup_{B_{R/4}} u \le c \left(\dashint_{B_{R/2}} u^{p}(x)\d x\right)^{1/p} + c R^{\alpha} \Vert f \Vert_{L^{\infty}},
\end{align}
where $B_{2R} \subset \Omega$.
\item[(ii)] Assume that \eqref{K1glob}, \eqref{UJSdual} hold true for some $\theta \in (\frac{d}{\alpha},\infty]$. Then there exists $c > 0$ such that for every $0 < R \le 1$, every $p \in (0,2]$ and every nonnegative, weak solution $u$ to \eqref{ellPDEdual} in $ B_{2R}$, estimate \eqref{eq:LB3} holds true.
\end{itemize}
\end{theorem}

\begin{proof}
We restrict ourselves to proving (i). The proof of (ii) works in the same way.
The proof works as in \cite{DKP14} (p.1828-1829).
Let us point out that this proof crucially relies on local boundedness of $u$, i.e.,
\begin{align}
\label{eq:locallybounded}
\sup_{B_{R/2}} u < \infty,
\end{align}
which follows from \autoref{thm:LB1} and \autoref{thm:LB1dual}, since $\tail_{K,\alpha}(u,R)$ and $\widehat{\tail}_{K,\alpha}(u,R)$ are finite under the assumptions of this theorem due to \autoref{lemma:tailfinite} (i) and \autoref{lemma:tailfinite} (ii), respectively.
Let $\frac{1}{4} \le t < s \le \frac{1}{2}$. We conclude from \autoref{thm:LB2} and a classical covering argument
\begin{align*}
\sup_{B_{tR}} u \le c_1 (s-t)^{-\frac{d}{2}} \left(\dashint_{B_{sR}} u^{2}(x)\d x\right)^{1/2} + c_2 R^{\alpha} \Vert f \Vert_{L^{\infty}} + c_2\delta \sup_{B_{sR}} u,
\end{align*}
where $c_1,c_2>0$ are constants. By Young's inequality (applied with $\frac{2}{p}, \frac{2}{2-p} \ge 1$),
\begin{align*}
\sup_{B_{tR}} u &\le c_1 (s-t)^{-\frac{d}{2}} \sup_{B_{sR}} u^{\frac{2-p}{2}}\left( \dashint_{B_{sR}} u^{p}(x)\d x\right)^{1/2} + c_2\delta \sup_{B_{sR}} u + c_2 R^{\alpha} \Vert f \Vert_{L^{\infty}}\\
&\le \left(c_2\delta + \frac{1}{4}\right) \sup_{B_{sR}} u + c_3 (s-t)^{-\frac{d}{p}}\left(\dashint_{B_{sR}} u^{p}(x)\d x\right)^{1/p} + c_2 R^{\alpha} \Vert f \Vert_{L^{\infty}}
\end{align*}
for some $c_3 > 0$. By choosing $\delta = \frac{1}{4c_2}$, we obtain
\begin{equation*}
\sup_{B_{tR}} u \le \frac{1}{2}\sup_{B_{sR}} u + c_4(s-t)^{-\frac{d}{p}}\left(\dashint_{B_{R/2}} u^{p}(x)\d x\right)^{1/p} + c_4 R^{\alpha} \Vert f \Vert_{L^{\infty}}
\end{equation*}
for $c_4 > 0$, and the result follows by application of Lemma 1.1 in \cite{GiGi82}, using \eqref{eq:locallybounded}.
\end{proof}

\subsection{Proofs of main results}
\label{sec:proofs}

In this section we provide the proofs of our main results \autoref{thm:apriorifHI} and \autoref{thm:fHI}. Let us recall the following theorem from \cite{KaWe22}.

\begin{theorem}(weak Harnack inequality)
\label{thm:wHI}
Assume \eqref{K2}, \eqref{cutoff}, \eqref{Poinc} and \eqref{Sob}.
\begin{itemize}
\item[(i)] Assume that \eqref{K1} holds true for some $\theta \in [\frac{d}{\alpha},\infty]$. Then there is $c > 0$ such that for every $0 < R \le 1$, and every nonnegative, weak supersolution $u$ to \eqref{PDE} in $I_R(t_0) \times B_{2R}$
\begin{equation}
\label{eq:wHI}
\inf_{(t_0 + R^{\alpha} - (R/2)^{\alpha}, t_0 + R^{\alpha}) \times B_{R/2}} u \ge c \left( \dashint_{(t_0 - R^{\alpha}, t_0 - R^{\alpha} + (R/2)^{\alpha}) \times B_{R/2}} u(t,x) \d x \d t - R^{\alpha} \Vert f \Vert_{L^{\infty}} \right),
\end{equation}
where $B_{2R} \subset \Omega$.
\item[(ii)] Assume that \eqref{K1glob} holds true for some $\theta \in (\frac{d}{\alpha},\infty]$. Then there is $c > 0$ such that for every $0 < R \le 1$, and every nonnegative, weak supersolution $u$ to \eqref{PDEdual} in $I_R(t_0) \times B_{2R}$, estimate \eqref{eq:wHI} holds true.
\end{itemize}
\end{theorem}

Now we prove \autoref{thm:apriorifHI} and \autoref{thm:fHI}. Both results require the weak Harnack inequality \autoref{thm:wHI}.

\begin{proof} [Proof of \autoref{thm:apriorifHI}]
We only prove (i) since the proof of (ii) follows by the same line of arguments. (i) follows from a combination of \autoref{thm:wHI} and \autoref{thm:LB1} (or  \autoref{thm:LB1dual}).
First, we deduce from \autoref{thm:LB1} (or \autoref{thm:LB1dual}) and a classical covering argument that for every $\frac{1}{4} \le t < s \le \frac{1}{2}$:
\begin{align*}
\sup_{I_{tR/2}^{\ominus} \times B_{tR}} u \le c_1 (s-t)^{-\frac{d+\alpha}{2}} \left(\dashint_{I_{sR/2}^{\ominus}} \dashint_{B_{sR}} u^{2}(t,x)\d x \d t\right)^{1/2} + \sup_{t \in I_{R/4}^{\ominus}} \tail_{K,\alpha}(u(t),R) + c_2 R^{\alpha} \Vert f \Vert_{L^{\infty}}.
\end{align*}
By a similar iteration argument as in the proof of \autoref{thm:LB3}, we deduce
\begin{align}
\label{eq:LBconsequence}
\sup_{I_{R/8}^{\ominus} \times B_{R/4}} u \le c_2 \left( \dashint_{I_{R/4}^{\ominus} \times B_{R/2}} u(t,x) \d x \d t \right) + c_2\sup_{t \in I_{R/4}^{\ominus}} \tail_{K,\alpha}(u(t),R) + c_2 R^{\alpha} \Vert f \Vert_{L^{\infty}}.
\end{align}
Next, \autoref{thm:wHI} yields
\begin{equation}
\label{whiapplied}
\inf_{(t_0 + (1 - 2^{-\alpha}) R^{\alpha}, t_0 + R^{\alpha}) \times B_{R/2}} u \ge c_1 \left( \dashint_{(t_0 - R^{\alpha}, t_0 -(1 - 2^{-\alpha})R^{\alpha}) \times B_{R/2}}\hspace{-2ex} u(t,x) \d x \d t - R^{\alpha} \Vert f \Vert_{L^{\infty}}\right)
\end{equation}
for some $c_1 > 0$. Note that $(t_0 - R^{\alpha}, t_0 -(1 - 2^{-\alpha})R^{\alpha}) = I_{R/2}^{\ominus}(t_0 -(1 - 2^{-\alpha})R^{\alpha})$.
Consequently, by \eqref{eq:LBconsequence}:
\begin{align*}
&\sup_{I_{R/8}^{\ominus}(t_0 -(1 - 2^{-\alpha})R^{\alpha}) \times B_{R/4}} u \\
&\le c_2 \left( \dashint_{I_{R/4}^{\ominus} (t_0 -(1 - 2^{-\alpha})R^{\alpha}) \times B_{R/2}} \hspace{-0.4cm}u(t,x) \d x \d t \right) + c_2 \hspace{-0.2cm}\sup_{t \in I_{R/4}^{\ominus} (t_0 -(1 - 2^{-\alpha})R^{\alpha})} \tail_{K,\alpha}(u(t),R) + c_2 R^{\alpha} \Vert f \Vert_{L^{\infty}}\\
&\le c_3 \left( \dashint_{(t_0 - R^{\alpha}, t_0 -(1 - 2^{-\alpha})R^{\alpha}) \times B_{R/2}} \hspace{-0.6cm}u(t,x) \d x \d t \right) + c_3 \hspace{-0.2cm}\sup_{t \in I_{R/4}^{\ominus} (t_0 -(1 - 2^{-\alpha})R^{\alpha})} \tail_{K,\alpha}(u(t),R) + c_3 R^{\alpha} \Vert f \Vert_{L^{\infty}}\\
&\le c_4 \inf_{(t_0 + (1 - 2^{-\alpha}) R^{\alpha}, t_0 + R^{\alpha}) \times B_{R/2}} + c_4 \hspace{-0.2cm}\sup_{t \in (t_0 - (1-2^{-\alpha}+4^{-\alpha})R^{\alpha} , t_0 - (1-2^{-\alpha})R^{\alpha})} \tail_{K,\alpha}(u(t),R) + c_4 R^{\alpha} \Vert f \Vert_{L^{\infty}}
\end{align*}
for some $c_2, c_3, c_4 > 0$. The proof is finished upon noticing that $I_{R/8}^{\ominus}(t_0 -(1 - 2^{-\alpha})R^{\alpha}) = (t_0 - (1-2^{-\alpha}+8^{-\alpha})R^{\alpha} , t_0 - (1 - 2^{-\alpha})R^{\alpha})$.
\end{proof}

\begin{proof} [Proof of \autoref{thm:fHI}]
This result follows directly by combination of \autoref{thm:LB3} and \autoref{thm:wHI}, where we apply \autoref{thm:LB3} with $p = 1$. 
\end{proof}

\subsection{Challenges in the parabolic case}
\label{sec:parabolicFHI}

Let us assume that \eqref{cutoff}, \eqref{elower}, \eqref{K1}, \eqref{K2} and \eqref{UJS} hold true for some $\theta \in [\frac{d}{\alpha},\infty]$.
The goal of this section is to discuss the validity of a parabolic version of  \autoref{thm:fHI}, i.e., to investigate the estimate
\begin{align}
\label{eq:parfhi}
\sup_{(t_0 - c_1 R^{\alpha} , t_0 - c_2 R^{\alpha}) \times B_{R/4}} u \le C \left(\inf_{(t_0 + c_2 R^{\alpha} , t_0 + R^{\alpha}) \times B_{R/2}} u +  R^{\alpha} \Vert f \Vert_{L^{\infty}}\right)
\end{align}
for some $C > 0$ and $0 < c_2 < c_1 < 1$ for nonnegative, weak solutions $u$ to \eqref{PDE} in $I_R^{\ominus} \times B_{2R}$, where $B_{2R} \subset \Omega$. In order to keep the presentation short, we will not discuss weak solutions to \eqref{PDEdual}, here.

As in the elliptic case, the general strategy to establish \eqref{eq:parfhi} would be to first prove an $L^{\infty}-L^p$-estimate of the following form (given any $p \in (0,2]$)
\begin{align}
\label{eq:LB3parabolic}
\sup_{I_{R/8}^{\ominus} \times B_{R/4}} u \le c \left(\dashint_{I_{R/4}^{\ominus}} \dashint_{B_{R/2}} u^{p}(t,x)\d x \d t\right)^{1/p} + c R^{\alpha} \Vert f \Vert_{L^{\infty}},
\end{align}
and to deduce \eqref{eq:parfhi} after combination with the weak parabolic Harnack  inequality \autoref{thm:wHI} as in the proof of \autoref{thm:apriorifHI}.\\
A natural approach in order to show \eqref{eq:LB3parabolic} would be to proceed as in the proof of \autoref{thm:LB1} but to apply \autoref{cor:partailest} in order to estimate the nonlocal tail by a local quantity. However, as \autoref{cor:partailest} only provides an estimate for $\int_{I_{r/2}^{\ominus}}\tail_K(u(t),r,r+\rho) \d t$ but not for $\sup_{I_{r/2}^{\ominus}}\tail_K(u(t),r,r+\rho)$, one needs to come up with a new idea to bridge the gap between $\sup_{I_{r/2}^{\ominus}}\tail_K(u(t),r,r+\rho)$ and $\int_{I_{r/2}^{\ominus}}\tail_K(u(t),r,r+\rho) \d t$. Note that the same issue appears in the symmetric case and has not been solved so far. There seems to be no proof of a parabolic Harnack inequality \eqref{eq:parfhi} for jumping kernels $K(x,y) \asymp |x-y|^{-d-\alpha}$ that  uses only analytic arguments. Note that via probabilistic methods, an estimate of the form \eqref{eq:parfhi} has been proved in the symmetric case in \cite{BaLe02}, \cite{ChKu03}.

Let us explain how to deduce \eqref{eq:LB3parabolic} under the condition that $u$ satisfies the following two additional assumptions:
\begin{itemize}
\item[(a)] There exists $c_0 > 0$ such that for every $R/2 \le r \le R$, $0 < \rho \le r \le r+\rho \le R$:
\begin{equation}
\label{eq:tailLB2parabolic}
\sup_{t \in I_{(r+\rho)/4}^{\ominus}} \tail_K(u(t),r+\frac{\rho}{2},r+\rho) \le c_0 \sup_{I_{(r+\frac{\rho}{2})/2}^{\ominus} \times B_{r+\rho}} u.
\end{equation}
\item[(b)] It holds $\sup_{I^{\ominus}_{R/4}(t_0)} \tail_{K,\alpha}(u(t),R) < \infty$.
\end{itemize}

\begin{remark}
\begin{itemize}
\item[(i)] Naturally, the constant $c$ in \eqref{eq:LB3parabolic} will depend on $c_0$.
\item[(ii)] \eqref{eq:tailLB2parabolic} holds true for global solutions to \eqref{PDE} in the symmetric case (see \cite{Str19a}).
\item[(iii)] It has been proposed in \cite{Kim19} to establish \eqref{eq:tailLB2parabolic} for every weak solution $u$ to \eqref{PDE} in $I \times B_{2R}$ with prescribed nonlocal parabolic boundary data $g \in L^{\infty}(I \times \R^d) \cap C(\overline{I} \times \R^d)$ with $c_0$ depending only on $g$. The proof of \cite[Lemma 5.3]{Kim19} is not complete.
\item[(iv)] Note that (b) is an additional restriction and does not naturally follow from our weak solution concept. We refer to \autoref{sec:tails} for a more detailed discussion of finiteness of tail terms.
\end{itemize}
\end{remark}

In order to establish \eqref{eq:LB3parabolic} we need to prove an analog of \eqref{eq:LB2}. As in the proof of \autoref{thm:LB1}, we derive \eqref{eq:LB1help3} and by combining it with \eqref{eq:tailLB2parabolic}, we deduce for every $\delta > 0$
\begin{align*}
A_{i} \le \frac{c_1}{R^{\alpha} M^{\frac{2}{\kappa'}}} 2^{\gamma i} \left(1 + \frac{\sup_{I_{R/2}^{\ominus} \times B_{R}} u + R^{\alpha}\Vert f \Vert_{L^{\infty}}}{M}\right)A_{i-1}^{1+\frac{1}{\kappa'}}
\end{align*}
for some $c_1 > 0$ and $\gamma > 1$. Here, $\kappa = 1 + \frac{\alpha}{d}$. By choosing $M:= \delta\left( \sup_{I_{R/2}^{\ominus} \times B_R} u + R^{\alpha}\Vert f \Vert_{L^{\infty}}\right) + C^{\frac{\kappa'^2}{2}}c_1^{\frac{\kappa'}{2}}\delta^{-\frac{\kappa'}{2}}R^{-\frac{\alpha \kappa'}{2}}A_0^{1/2}$, where $C:= 2^{\gamma} > 1$, we can deduce 
\begin{equation}
\label{eq:parabolicLB2}
\sup_{I_{R/8}^{\ominus} \times B_{R/2}} u \le \delta \left(\sup_{I_{R/2}^{\ominus} \times B_R} u + R^{\alpha}\Vert f \Vert_{L^{\infty}}\right) + c_2 \delta^{-\frac{\kappa'}{2}}\left(\dashint_{I_{R/4}^{\ominus}} \dashint_{B_R} u^{2}(t,x)\d x \d t\right)^{1/2}
\end{equation}
for some $c_2 > 0$. This estimate is a parabolic analog of \eqref{eq:LB2}. Note that \eqref{eq:parabolicLB2} can also be established via the arguments from the proof of \autoref{thm:LB1dual} using the Moser iteration.\\
Next, we intend to prove \eqref{eq:LB3parabolic} by adapting the arguments in the proof of \autoref{thm:LB3} to the parabolic setting.\\
As in the elliptic case, a standard covering argument yields for every $\frac{1}{4} \le t < s \le \frac{1}{2}$
\begin{align*}
\sup_{I_{tR/2}^{\ominus} \times B_{tR}} u \le c_3 (s-t)^{-\frac{d+\alpha}{2}} \left(\dashint_{I_{sR/2}^{\ominus}} \dashint_{B_{sR}} u^{2}(t,x)\d x \d t\right)^{1/2} + c_4 R^{\alpha} \Vert f \Vert_{L^{\infty}} + c_4\delta \sup_{I_{sR/2}^{\ominus} \times B_{sR}} u,
\end{align*}
where $c_3, c_4>0$ are constants. By Young's inequality and choosing $\delta = \frac{1}{c_4}$, we arrive at
\begin{equation*}
\sup_{I_{tR/2}^{\ominus} \times B_{tR}} u \le \frac{1}{2}\sup_{I_{sR/2}^{\ominus} \times B_{sR}} u + c_4(s-t)^{-\frac{d+\alpha}{p}}\left(\dashint_{I_{R/4}^{\ominus}} \dashint_{B_{R/2}} u^{p}(t,x)\d x \d t\right)^{1/p} + c_4 R^{\alpha} \Vert f \Vert_{L^{\infty}},
\end{equation*}
where $p \in (0,2]$ can be chosen arbitrarily.\\
Now, \eqref{eq:LB3parabolic} follows from Lemma 1.1 in \cite{GiGi82}, however, this result is only applicable if 
\begin{align}
\label{eq:parlocallybounded}
\sup_{I_{R/4}^{\ominus} \times B_{R/2}} u < \infty.
\end{align}
In order to obtain \eqref{eq:parlocallybounded}, we apply \autoref{thm:LB1} (or  \autoref{thm:LB1dual}) and use condition (b) on $u$. This concludes the proof of \eqref{eq:LB3parabolic} under the additional assumptions (a) and (b).

\section{Appendix}
\label{sec:app}

The following lemma justifies the way we deal with the weak formulation of \eqref{PDE}, or \eqref{PDEdual}, in the proof of \autoref{thm:LB1} after testing with $\phi(t,x) = \tau^2(x)(u(t,x)-k)_+$ for some $k \ge 0$ where $u$ is a subsolution to the respective equation. In fact, $\phi$ is a priori not differentiable in $t$, which prevents us from integrating by parts. The idea of the proof is to test the equation with an auxiliary function having the required smoothness properties in $t$. This can be achieved with the help of Steklov averages. For symmetric nonlocal equations, such lemmas are well known (see \cite{Str19b}, \cite{FeKa13}). We adapt the idea of the proof of \cite{FeKa13} to the nonsymmetric case. Note that Lemma A.2 in \cite{FeKa13} is not sufficient for the proof of (A.4) in \cite{FeKa13}. Our proof fixes the gap in their argument.

\begin{lemma}
\label{lemma:steklovDG}
Assume \eqref{cutoff}.
\begin{itemize}
\item[(i)] Assume that \eqref{K1} holds true for some $\theta \in [\frac{d}{\alpha},\infty]$. Moreover, assume \eqref{Sob} if $\theta < \infty$. Let $u \in V(B_{r+\rho}|\R^d)$ be a weak subsolution to \eqref{PDE}. Then, for every $[t_1,t_2] \subset I$, $0 < \rho \le r \le 1$ with $B_{r+\rho} \subset \Omega$, every $k \ge 0$, every $\chi \in C^1_c(\R)$
\begin{align*}
\chi^2(t_2)\int_{B_{r+\rho}} & \hspace{-0.4cm}(u(t_2)-k)_+^2 \tau^2 \d x - \chi^2(t_1) \int_{B_{r+\rho}}\hspace{-0.4cm} (u(t_1)-k)_+^2 \tau^2 \d x - \int_{t_1}^{t_2} \partial_t (\chi^2(t)) \int_{B_{r+\rho}} \hspace{-0.4cm} (u(t)-k)_+^2 \tau^2 \d x \d t\\
&+ \int_{t_1}^{t_2} \chi^2(t) \cE(u(t) , \tau^2(u(t)-k)_+) \d t \le \int_{t_1}^{t_2} \chi^2(t) \int_{B_{r+\rho}}\hspace{-0.2cm} f(t,x) \tau^2(x) (u(t,x) -k)_+ \d x \d t,
\end{align*}
where $\tau = \tau_{r,\frac{\rho}{2}}$.
\item[(ii)]
Assume that \eqref{K1glob} holds true for some $\theta \in [\frac{d}{\alpha},\infty]$. Moreover, assume \eqref{Sob} if $\theta < \infty$. Let $u \in V(B_{r+\rho}|\R^d)$ be a weak subsolution to \eqref{PDEdual}. Then, for every $[t_1,t_2] \subset I$, $0 < \rho \le r \le 1$ with $B_{r+\rho} \subset \Omega$, every $k \ge 0$, every $\chi \in C^1_c(\R)$
\begin{align*}
\chi^2(t_2) \int_{B_{r+\rho}} &\hspace{-0.4cm} (u(t_2)-k)_+^2 \tau^2 \d x - \chi^2(t_1) \int_{B_{r+\rho}}\hspace{-0.4cm} (u(t_1)-k)_+^2 \tau^2 \d x - \int_{t_1}^{t_2} \partial_t (\chi^2(t)) \int_{B_{r+\rho}} \hspace{-0.4cm} (u(t)-k)_+^2 \tau^2 \d x \d t\\
&+ \int_{t_1}^{t_2} \chi^2(t) \widehat{\cE}(u(t) , \tau^2(u(t)-k)_+) \d t \le \int_{t_1}^{t_2} \chi^2(t) \int_{B_{r+\rho}}\hspace{-0.2cm} f(t,x) \tau^2(x) (u(t,x) -k)_+ \d x \d t.
\end{align*}
\end{itemize}
\end{lemma}

\begin{proof}
Given $v \in L^1((0,T);X)$ for some Banach space $X$, we define its Steklov average  $v_h(t,x) = \dashint_t^{t+h} v(s,\cdot) \d s$ if $t+h \in I$ and $v_h (t,x) = 0$ otherwise. Observe that $\partial_t u_h(t,x) = \frac{1}{h}(u(t+h,x)-u(t,x)) = \dashint_t^{t+h} \partial_s u(s,x) \d s$. 
According to Lemma A.1 in \cite{FeKa13}, it holds
\begin{align}
\label{eq:steklovaux1}
\Vert v_h(t) - v(t) \Vert_{L^2} &\to 0, ~~ \text{ as } h \searrow 0~~  \text{ if } v \in C((0,T);L^2(B_{r+\rho})),\\
\label{eq:steklovaux2}
\Vert v_h - v \Vert_{L^2([t_1,t_2];X)} &\to 0, ~~ \text{ as } h \searrow 0,\\
\label{eq:steklovaux3}
\Vert v_h \Vert_{L^2([t_1,t_2];X)} &\le \Vert v \Vert_{L^2([t_1,t_2];X)}.
\end{align}
We first explain how to prove (i).
Let $t \in I$. By testing the equation for $u$ with $\phi = \tau^2(u_h(t) - k)_+$, we obtain after integrating over $(t,t+h)$ for some $h > 0$ such that $t +h \in I$ and dividing by $h$:
\begin{equation*}
\int_{B_{r+\rho}} \partial_t u_h(t,x) \phi(t,x) \d x + \cE(u_h(t),\phi(t)) \le (f(t),\phi(t)) .
\end{equation*}
Note that $t \mapsto u_h(t,x)$ is differentiable for a.e. $x \in B_{r+\rho}$, and therefore $\partial_t u_h(t,x) \phi(t,x) = \frac{1}{2} \partial_t [(u_h(t,x)-k)_+^2] \tau^2(x)$.\\
We multiply with $\chi^2(t)$ and integrate over $(t_1,t_2)$. Integration by parts yields
\begin{align*}
\int_{B_{r+\rho}} &\hspace{-0.4cm} \chi^2(t_2)(u_h(t_2)-k)_+^2 \tau^2 \d x - \int_{B_{r+\rho}}\hspace{-0.4cm} \chi^2(t_1) (u_h(t_1)-k)_+^2 \tau^2 \d x - \int_{t_1}^{t_2}\hspace{-0.2cm} \int_{B_{r+\rho}}\hspace{-0.3cm} \partial_t (\chi^2(t)) (u_h(t)-k)_+^2 \tau^2 \d x \d t\\
&+ \int_{t_1}^{t_2} \chi^2(t) \cE(u_h(t) , \tau^2(u_h(t)-k)_+) \d t \le \int_{t_1}^{t_2} \chi^2(t) \int_{B_{r+\rho}} f(t,x) \tau^2(x) (u_h(t,x) -k)_+ \d x \d t.
\end{align*}
Since $\Vert |(u_h(t)-k)_+ - (u(t)-k)_+|\tau^2 \Vert_{L^2(B_{r+\rho})} \le \Vert u_h(t) - u(t) \Vert_{L^2(B_{r+\rho})}$, it follows by \eqref{eq:steklovaux1}:
\begin{equation*}
\int_{B_{r+\rho}} (u_h(t)-k)_+^2 \tau^2 \d x \to \int_{B_{r+\rho}} (u(t) - k)^2_+ \tau^2 \d x ~~ \text{ for } t \in [t_1,t_2].
\end{equation*}
Moreover, \eqref{eq:steklovaux2} implies
\begin{align*}
\int_{t_1}^{t_2} \partial_t (\chi^2(t)) \int_{B_{r+\rho}} (u_h(t)-k)_+^2 \tau^2 \d x \d t &\to \int_{t_1}^{t_2} \int_{B_{r+\rho}} \partial_t (\chi^2(t)) (u(t)-k)_+^2 \tau^2 \d x \d t,\\
\int_{t_1}^{t_2} \chi^2(t) \int_{B_{r+\rho}}\hspace{-0.2cm} f(t,x) \tau^2(x) (u_h(t,x) -k)_+ \d x \d t &\to \int_{t_1}^{t_2} \chi^2(t) \int_{B_{r+\rho}}\hspace{-0.2cm} f(t,x) \tau^2(x) (u(t,x) -k)_+ \d x \d t,
\end{align*}
as $h \searrow 0$.
It remains to prove that 
\begin{equation}
\label{eq:steklovhelp1}
\int_{t_1}^{t_2} \chi^2(t) \cE(u_h(t),\tau^2 (u_h(t) -k)_+) \d t \to \int_{t_1}^{t_2} \chi^2(t) \cE(u(t),\tau^2 (u(t) -k)_+) \d t.
\end{equation}
In Lemma A.2 in \cite{FeKa13}, the authors established a related convergence property for symmetric energy forms. However, their proof has a gap, since Lemma A.2 does not suffice to deduce the desired result (even in the symmetric case), since if $\Phi = f(u)$ it does not hold in general that $\Phi_h = f(u_h)$.

We define $V(t,x,y) = u(t,x) - u(t,y)$, $W(t,x,y) = \tau^2(x) (u(t,x) - k)_+ - \tau^2(y) (u(t,y) - k)_+$ and $\widetilde{W}(t,x,y) = \tau^2(x) (u_h(t,x) - k)_+ - \tau^2(y) (u_h(t,y) - k)_+$.

Our goal is to show that
\begin{align}
\label{eq:steklov1}
\int_{t_1}^{t_2} \left|\cE(u_h(t) - u(t) , \tau^2(u_h(t) - k)_+)\right| \d t &\to 0,\\
\label{eq:steklov2}
\int_{t_1}^{t_2} \left|\cE(u(t) , \tau^2(u_h - k)_+ - \tau^2(u(t) - k)_+)\right| \d t &\to 0.
\end{align}

To establish \eqref{eq:steklov1}, we split 
\begin{align*}
&\int_{t_1}^{t_2} \left|\cE(u_h(t) - u(t) , \tau^2(u_h(t) - k)_+)\right| \d t \le \left\Vert \cE^{K_s}_{B_{r+\rho}}(u_h - u , \tau^2(u_h - k)_+)\right\Vert_{L^1([t_1,t_2])}\\
&+ \left\Vert  \cE^{K_a}_{B_{r+\rho}}(u_h - u , \tau^2(u_h - k)_+) \right\Vert_{L^1([t_1,t_2])} + \left\Vert  \cE_{(B_{r+\rho} \times B_{r+\rho})^c}(u_h - u , \tau^2(u_h - k)_+) \right\Vert_{L^1([t_1,t_2])}\\
&=: I_1 + I_2 + I_3,
\end{align*}
and establish the convergence of each term separately.
For $I_1$, we estimate 
\begin{align*}
I_1 &\le \int_{t_1}^{t_2} \int_{B_{r+\rho}} \int_{B_{r+\rho}} \vert V_h(t,x,y) - V(t,x,y) \vert \vert \widetilde{W}(t,x,y) \vert K_s(x,y) \d y \d x \d t\\
&\le \Vert (V_h - V) K_s ^{1/2}\Vert_{L^2([t_1,t_2] \times B_{r+\rho} \times B_{r+\rho})} \Vert \widetilde{W} K_s^{1/2} \Vert_{L^2([t_1,t_2] \times B_{r+\rho} \times B_{r+\rho})}\\
&\le \Vert u_h - u \Vert_{L^2([t_1,t_2] ; V(B_{r+\rho}|\R^d))} \left\Vert \cE^{K_s}_{B_{r+\rho}}(\tau^2 (u_h -k)_+,\tau^2 (u_h -k)_+) \right\Vert_{L^1([t_1,t_2])}^{1/2}\\
&\le \Vert u_h - u \Vert_{L^2([t_1,t_2] ; V(B_{r+\rho}|\R^d))} \Vert \tau^2 u \Vert_{L^2([t_1,t_2] ; V(B_{r+\rho}|\R^d))}\\
&\to 0,
\end{align*}
where we used \eqref{eq:steklovaux2} and \eqref{eq:steklovaux3}, that $u,\phi \in L^2([t_1,t_2]; V(B_{r+\rho}|\R^d))$ and the fact that due to Markov property of $\cE^{K_s}$ and \eqref{eq:steklovaux3}:
\begin{align}
\label{eq:steklovMarkov}
\begin{split}
\cE^{K_s}_{B_{r+\rho}}(\tau^2 (u_h-k)_+,\tau^2 (u_h-k)_+)& \le \cE^{K_s}_{B_{r+\rho}}(\tau^2 u_h,\tau^2 u_h)\\
&= \cE^{K_s}_{B_{r+\rho}}([\tau^2 u]_h,[\tau^2 u]_h) \le \Vert \tau^2 u \Vert_{V(B_{r+\rho}|\R^d)}^2.
\end{split}
\end{align}
For $I_2$,
\begin{align*}
I_2 &\le \int_{t_1}^{t_2} \int_{B_{r+\rho}} \int_{B_{r+\rho}} \vert V_h(t,x,y) - V(t,x,y) \vert  \tau^2(x)(u_h(t,x)-k)_+  \vert K_a(x,y) \vert \d y \d x \d t\\
&\le \left\Vert (V_h - V)J^{1/2} \right\Vert_{L^2([t_1,t_2] \times B_{r+\rho} \times B_{r+\rho})} \left\Vert \int_{B_{r+\frac{\rho}{2}}} \hspace{-0.3cm} (u_h(\cdot,x) -k)_+^2 \left(\int_{B_{r+\rho}}\hspace{-0.2cm} \frac{\vert K_a(x,y) \vert^2}{J(x,y)} \d y\right) \d x \right\Vert_{L^1([t_1,t_2])}^{1/2}\\
&\le c \Vert u_h - u \Vert_{L^2([t_1,t_2] ; V(B_{r+\rho}|\R^d))} \Vert u \Vert_{L^2([t_1,t_2] ; L^{2\theta'}(B_{r+\frac{\rho}{2}}))}\\
&\le c \Vert u_h - u \Vert_{L^2([t_1,t_2] ; V(B_{r+\rho}|\R^d))} \Vert u \Vert_{L^2([t_1,t_2] ; V(B_{r+\rho}|\R^d))}\\
&\to 0,
\end{align*}
where $c > 0$ might depend on $\rho$ and we used \eqref{K1}, \eqref{eq:steklovaux2}, \eqref{eq:steklovaux3}, that $u \in L^2([t_1,t_2]; V(B_{r+\rho}|\R^d))$ and \eqref{Sob}.
For $I_3$, we obtain
\begin{align*}
I_3 &\le 2 \int_{t_1}^{t_2}\int_{B_{r+\rho}} \int_{B_{r+\rho}^c} \vert V_h(t,x,y) - V(t,x,y) \vert \tau^2(x) (u_h(t,x) - k)_+ K_s(x,y) \d y \d x \d t\\
&\le 2\Vert (V_h - V) K_s ^{1/2}\Vert_{L^2([t_1,t_2] \times B_{r+\rho} \times B_{r+\rho}^c)} \left\Vert \int_{B_{r+\frac{\rho}{2}}} (u_h(\cdot,x) - k)^2_+ \Gamma^{K_s}(\tau,\tau)(x) \d x \right\Vert_{L^1([t_1,t_2])}^{1/2} \\
&\le c \rho^{-\alpha/2} \Vert u_h - u \Vert_{L^2([t_1,t_2] ; V(B_{r+\rho}|\R^d))} \Vert u \Vert_{L^2([t_1,t_2] \times B_{r+\rho})} \\
&\to 0,
\end{align*}
where we used \eqref{eq:KaKs}, \eqref{cutoff}, \eqref{eq:steklovaux2}, \eqref{eq:steklovaux3} and that $u \in L^2([t_1,t_2]; V(B_{r+\rho}|\R^d))$.

It remains to prove \eqref{eq:steklov2}. Again, we split
\begin{align*}
&\int_{t_1}^{t_2} \left|\cE(u(t) , \tau^2(u_h - k)_+ - \tau^2(u(t) - k)_+)\right| \d t  \le \left\Vert \cE^{K_s}_{B_{r+\rho}}(u, \tau^2(u_h - k)_+ - \tau^2(u - k)_+)\right\Vert_{L^1([t_1,t_2])}\\
&\qquad\qquad\qquad\qquad\qquad\qquad+ \left\Vert  \cE^{K_a}_{B_{r+\rho}}(u , \tau^2(u_h - k)_+ - \tau^2(u - k)_+) \right\Vert_{L^1([t_1,t_2])}\\
&\qquad\qquad\qquad\qquad\qquad\qquad+ \left\Vert  \cE_{(B_{r+\rho} \times B_{r+\rho})^c}(u , \tau^2(u_h - k)_+ - \tau^2(u - k)_+) \right\Vert_{L^1([t_1,t_2])}\\
&\qquad\qquad\qquad\qquad\qquad\qquad=: J_1 + J_2 + J_3,
\end{align*}

Convergence of $J_1$ can be proved as follows. First, by H\"older's inequality,
\begin{align*}
J_1 &\le \Vert u \Vert_{L^2([t_1,t_2]; V(B_{r+\rho}|\R^d))} \left\Vert \int_{B_{r+\rho}} \int_{B_{r+\rho}}  |W(\cdot,x,y) - \widetilde{W}(\cdot,x,y)|^2 K_s(x,y)  \d y \d x \right\Vert_{L^1([t_1,t_2])}^{1/2}.
\end{align*}
Since $u \in L^2([t_1,t_2]; V(B_{r+\rho}|\R^d))$, it suffices to prove that the second factor converges to zero, in order to conclude that $J_1 \to 0$. For this, we claim that there exist $\xi(t,x,y), \widetilde{\xi}(t,x,y) \in [0,1]$ such that
\begin{align*}
\tau^2(x)(u(t,x) - k)_+ - \tau^2(y)(u(t,y) - k)_+ &= \xi(t,x,y)[f(t,x) - f(t,y)],\\
\tau^2(x)(u_h(t,x) - k)_+ - \tau^2(y)(u_h(t,y) - k)_+ &= \widetilde{\xi}(t,x,y)[f_h(t,x) - f_h(t,y)],
\end{align*}
where we define $f(t,x) = \tau^2(x)(u(t,x) - k)$. In fact, it is easy to see that\begin{align*}
\xi(t,x,y) &= 
\begin{cases}
1,&  u(t,x),u(t,y) > k,\\
0,&  u(t,x),u(t,y) \le k,\\
\frac{f(t,x)}{f(t,x)-f(t,y)},&  u(t,x) > k \ge u(t,y),\\
\frac{f(t,y)}{f(t,y)-f(t,x)},&  u(t,y) > k \ge u(t,x),\\
\end{cases},\\
\widetilde{\xi}(t,x,y) &= 
\begin{cases}
1,&  u_h(t,x),u_h(t,y) > k,\\
0,&  u_h(t,x),u_h(t,y) \le k,\\
\frac{f_h(t,x)}{f_h(t,x)-f_h(t,y)},&  u_h(t,x) > k \ge u_h(t,y),\\
\frac{f_h(t,y)}{f_h(t,y)-f_h(t,x)},&  u_h(t,y) > k \ge u_h(t,x)\\
\end{cases}
\end{align*}
have the desired properties. We estimate
\begin{align*}
&\left\Vert \int_{B_{r+\rho}} \int_{B_{r+\rho}}  |W(\cdot,x,y) - \widetilde{W}(\cdot,x,y)|^2 K_s(x,y)  \d y \d x \right\Vert_{L^1([t_1,t_2])}^{1/2}\\
&\le 2\left\Vert  \int_{B_{r+\rho}} \int_{B_{r+\rho}} \hspace{-0.2cm} |\widetilde{\xi}(\cdot,x,y)|^2 [(f_h(t,x) - f(t,x)) - (f_h(t,y) - f(t,y))]^2   K_s(x,y) \d y \d x\right\Vert_{L^1([t_1,t_2])}^{1/2}\\
&+ 2\left\Vert  \int_{B_{r+\rho}} \int_{B_{r+\rho}} \hspace{-0.2cm} |\widetilde{\xi}(\cdot,x,y) - \xi(\cdot,x,y)|^2 [f(t,x) - f(t,y)]^2   K_s(x,y) \d y \d x\right\Vert_{L^1([t_1,t_2])}^{1/2}\\
&\le J_{1,1} + J_{1,2}.
\end{align*}
For $J_{1,1}$, note that
\begin{align*}
J_{1,1} \le 2 \Vert f_h - f \Vert_{L^2([t_1,t_2];V(B_{r+\rho}|\R^d))} \to 0,
\end{align*}
where we used that $|\widetilde{\xi}| \le 1$, $f \in L^2([t_1,t_2]; V(B_{r+\rho}|\R^d))$ and \eqref{eq:steklovaux2}.
For $J_{1,2}$. we observe that $\vert \widetilde{\xi}(t,x,y) - \xi(t,x,y) \vert \to 0$, as $h \searrow 0$ for a.e. $t,x,y$. Since $f \in L^2([t_1,t_2]; V(B_{r+\rho}|\R^d))$, it follows from dominated convergence that also $J_{1,2} \to 0$.\\
For $J_2$, we estimate
\begin{align*}
J_2 &\le \int_{t_1}^{t_2} \int_{B_{r+\rho}} \int_{B_{r+\rho}} \vert V(t,x,y) \vert  \tau^2(x) \vert (u_h(t,x)-k)_+ - (u(t,x)-k)_+ \vert \vert K_a(x,y) \vert \d y \d x \d t\\
&\le \Vert u \Vert_{L^2([t_1,t_2]; V(B_{r+\rho}|\R^d))} \left\Vert  \int_{B_{r+\rho}} \hspace{-0.5cm} \left| (u_h(\cdot,x) -k)_+ - (u(\cdot,x) -k)_+\right|^2 \left(\int_{B_{r+\rho}} \hspace{-0.4cm} \frac{\vert K_a(x,y) \vert^2}{J(x,y)} \d y\right) \d x \right\Vert_{L^1([t_1,t_2])}^{1/2}\\
&\le c \Vert u \Vert_{L^2([t_1,t_2]; V(B_{r+\rho}|\R^d))}  \Vert u_h - u \Vert_{L^2([t_1,t_2]; L^{2\theta'}(B_{r+\rho}))}\\
&\to 0,
\end{align*}
where we used \eqref{K1}, $\vert (u_h(t,x)-k)_+ - (u(t,x)-k)_+ \vert \le \vert u_h(t,x) - u(t,x) \vert$, and $u \in L^2([t_1,t_2]; V(B_{r+\rho}|\R^d))$, \eqref{Sob} and \eqref{eq:steklovaux2}. To prove convergence of $J_3$, we proceed as follows
\begin{align*}
J_3 &\le 2\int_{t_1}^{t_2}\int_{B_{r+\rho}} \int_{B_{r+\rho}^c} \vert V(t,x,y) \vert \tau^2(x) |(u_h(t,x) - k)_+ - (u(t,x) - k)_+| K_s(x,y) \d y \d x \d t\\
&\le 2\Vert u \Vert_{L^2([t_1,t_2]; V(B_{r+\rho}|\R^d))} \left\Vert \int_{B_{r+\frac{\rho}{2}}} |(u_h(\cdot,x) - k)_+ - u(\cdot,x) - k)|^2 \Gamma^{K_s}(\tau,\tau)(x) \d x \right\Vert_{L^1([t_1,t_2])}^{1/2}\\
&\le c \rho^{-\alpha/2}\Vert u \Vert_{L^2([t_1,t_2]; V(B_{r+\rho}|\R^d))} \Vert u_h - u \Vert_{L^2([t_1,t_2] \times B_{r+\rho})}\\
&\to 0,
\end{align*}

where we used \eqref{cutoff} and \eqref{eq:steklovaux2} and $u \in L^2([t_1,t_2]; V(B_{r+\rho}|\R^d))$.\\
Altogether, this proves \eqref{eq:steklovhelp1} and we deduce the desired result.
Let us now prove (ii). In analogy to the proof of (i), it is only left to show
\begin{align}
\label{eq:steklovhelp1dual}
\int_{t_1}^{t_2} \widehat{\cE}(u_h(t),\tau^2 (u_h(t) -k)_+) \d t &\to \int_{t_1}^{t_2} \widehat{\cE}(u(t),\tau^2 (u(t) -k)_+) \d t.
\end{align}
We will establish \eqref{eq:steklovhelp1dual} by proving the following two properties:
\begin{align}
\label{eq:steklov1dual}
\int_{t_1}^{t_2} \left|\cE(u_h(t) - u(t) , \tau^2(u_h(t) - k)_+)\right| \d t &\to 0,\\
\label{eq:steklov2dual}
\int_{t_1}^{t_2} \left|\cE(u(t) , \tau^2(u_h - k)_+ - \tau^2(u(t) - k)_+)\right| \d t &\to 0.
\end{align}
Let us first prove \eqref{eq:steklov1dual}. In analogy to the proof of \eqref{eq:steklov1}, we split 
\begin{align*}
&\int_{t_1}^{t_2} \left|\widehat{\cE}(u_h(t) - u(t) , \tau^2(u_h(t) - k)_+)\right| \d t \le \left\Vert \cE^{K_s}_{B_{r+\rho}}(u_h - u , \tau^2(u_h - k)_+)\right\Vert_{L^1([t_1,t_2])}\\
&+ \left\Vert  \widehat{\cE}^{K_a}_{B_{r+\rho}}(u_h - u , \tau^2(u_h - k)_+) \right\Vert_{L^1([t_1,t_2])} + \left\Vert  \widehat{\cE}_{(B_{r+\rho} \times B_{r+\rho})^c}(u_h - u , \tau^2(u_h - k)_+) \right\Vert_{L^1([t_1,t_2])}\\
&=: \widehat{I_1} + \widehat{I_2} + \widehat{I_3}.
\end{align*}
\pagebreak[3]
In (i), we already showed that $\widehat{I_1} \to 0$. Let us estimate $\widehat{I_2}$ as follows
\begin{align*}
\widehat{I_2} &\le \left\Vert (u_h-u) \widetilde{W} |K_a|\right\Vert_{L^1([t_1,t_2])}\\
&\le \left\Vert \cE^{K_s}_{B_{r+\rho}}(\tau^2 (u_h-k)_+,\tau^2 (u_h-k)_+) \right\Vert_{L^1([t_1,t_2])}^{1/2} \Vert u_h - u \Vert_{L^2([t_1,t_2]; L^{2\theta'}(B_{r+\frac{\rho}{2}}))}\\
&\le c\left\Vert \tau^2 u \right\Vert_{L^2([t_1,t_2]; V(B_{r+\rho}|\R^d))} \Vert u_h - u \Vert_{L^2([t_1,t_2]; V(B_{r+\rho}|\R^d))}\\
&\to 0,
\end{align*}
where we used \eqref{K1glob}, \eqref{eq:steklovaux2}, \eqref{eq:steklovMarkov} and \eqref{Sob}.
Moreover, $\widehat{I_3}$ can be treated as follows:
\begin{align*}
\widehat{I_3} &\le \left\Vert\cE^{K_s}_{(B_{r+\rho} \times B_{r+\rho})^c}(u_h - u , \tau^2(u_h - k)_+) \right\Vert_{L^1([t_1,t_2])}\\
&+\left\Vert  \int_{B_{r+\frac{\rho}{2}}} \int_{B_{r+\rho}^c}  (u_h(\cdot,x) - u(\cdot,x)) \tau^2(x) (u_h(\cdot,x) -k)_+ |K_a(x,y)| \d y \d x \right\Vert_{L^1([t_1,t_2])}\\
&+ \left\Vert  \int_{B_{r+\rho}^c} \int_{B_{r+\frac{\rho}{2}}} (u_h(\cdot,x) - u(\cdot,x)) \tau^2(y) (u_h(\cdot,y) -k)_+  |K_a(x,y)| \d y \d x \right\Vert_{L^1([t_1,t_2])}\\
&=: \widehat{I_{3,1}} + \widehat{I_{3,2}} + \widehat{I_{3,3}}.
\end{align*}
The proof of convergence for $\widehat{I_{3,1}}$ goes exactly like for $I_3$.
For $\widehat{I_{3,2}}$, we estimate using the assumptions \eqref{K1glob} and \eqref{cutoff}
\begin{align*}
\widehat{I_{3,2}} + \widehat{I_{3,3}} &\le \left\Vert \int_{\R^d} |u_h(\cdot,x) - u(\cdot,x)|^2 \left(\int_{\R^d} \frac{|K_a(x,y)|^2}{J(x,y)} \d y\right) \d x  \right\Vert_{L^1([t_1,t_2])}^{1/2}\\
&~~~~\left\Vert \int_{B_{r+\frac{\rho}{2}}} (u_h(\cdot,x)-k)_+^2 \Gamma^{J}(\tau,\tau)(x) \d x  \right\Vert_{L^1([t_1,t_2])}^{1/2}\\
&\le c\rho^{-\alpha/2}\Vert u_h-u \Vert_{L^2([t_1,t_2];L^{2\theta'}(\R^d))} \Vert (u_h-k)_+ \Vert_{L^2([t_1,t_2] \times B_{r+\rho})}\\
&\le c\rho^{-\alpha/2}\Vert u_h-u \Vert_{L^2([t_1,t_2];L^{2\theta'}(\R^d))} \Vert u \Vert_{L^2([t_1,t_2] \times B_{r+\rho})}\\
&\to 0,
\end{align*}
where we used \eqref{eq:steklovaux2} and \eqref{eq:steklovaux3} and $u \in L^2([t_1,t_2]; L^{2\theta'}(\R^d))$.
We have established \eqref{eq:steklov1dual}. To prove \eqref{eq:steklov2dual}, let us 
again split
\begin{align*}
&\int_{t_1}^{t_2} \left|\widehat{\cE}(u(t) , \tau^2(u_h - k)_+ - \tau^2(u(t) - k)_+)\right| \d t  \le \left\Vert \cE^{K_s}_{B_{r+\rho}}(u, \tau^2(u_h - k)_+ - \tau^2(u - k)_+)\right\Vert_{L^1([t_1,t_2])}\\
&\qquad\qquad\qquad\qquad\qquad\qquad\qquad + \left\Vert \widehat{\cE}^{K_a}_{B_{r+\rho}}(u , \tau^2(u_h - k)_+ - \tau^2(u - k)_+) \right\Vert_{L^1([t_1,t_2])} \hspace{-0.3cm}\\
& \qquad\qquad\qquad\qquad\qquad\qquad\qquad + \left\Vert \widehat{\cE}_{(B_{r+\rho} \times B_{r+\rho})^c}(u , \tau^2(u_h - k)_+ - \tau^2(u - k)_+) \right\Vert_{L^1([t_1,t_2])}\\
& \qquad\qquad\qquad\qquad\qquad\qquad\qquad=: \widehat{J_1} + \widehat{J_2} + \widehat{J_3}.
\end{align*}
Note that $\widehat{J_1} = J_1 \to 0$. For $\widehat{J_2}$, we estimate
\begin{align*}
\widehat{J_2} &\le \Vert u \Vert_{L^{2}([t_1,t_2];L^{2\theta'}(B_{r+\rho}))} \left\Vert \int_{B_{r+\rho}} \int_{B_{r+\rho}}  |W(\cdot,x,y) - \widetilde{W}(\cdot,x,y)|^2 K_s(x,y)  \d y \d x \right\Vert_{L^1([t_1,t_2])}^{1/2},
\end{align*}
where we used \eqref{K1glob} and that $u \in L^{2}([t_1,t_2];L^{2\theta'}(B_{r+\rho}))$. We conclude that $\widehat{J_2} \to 0$ since the second factor converges to  zero, as we proved already when dealing with $J_1$. \\
To estimate $\widehat{J_3}$, we proceed as follows:
\begin{align*}
\widehat{J_3} &\le \left\Vert \cE^{K_s
}_{(B_{r+\rho} \times B_{r+\rho})^c}(u , \tau^2(u_h - k)_+ - \tau^2(u - k)_+) \right\Vert_{L^1([t_1,t_2])}\\
&+ \left\Vert \int_{B_{r+\frac{\rho}{2}}} \int_{B_{r+\rho}^c} \tau^2(x)|(u_h - k)_+(x) - (u - k)_+(x)| u(x) |K_a(x,y)| \d y \d x\right\Vert_{L^1([t_1,t_2])}\\
&+ \left\Vert \int_{B_{r+\rho}^c} \int_{B_{r+\frac{\rho}{2}}} \tau^2(y)|(u_h - k)_+(y) - (u - k)_+(y)| u(x) |K_a(x,y)| \d y \d x\right\Vert_{L^1([t_1,t_2])}\\
&= \widehat{J_{3,1}} + \widehat{J_{3,2}} + \widehat{J_{3,3}}.
\end{align*}
Note that $\widehat{J_{3,1}} \to 0$ follows very similar to the proof of $J_3 \to 0$.
$\widehat{J_{3,2}}$ and $\widehat{J_{3,3}}$ are estimated as follows, using similar arguments as in the estimate of $\widehat{I_{3,2}}$ and $\widehat{I_{3,3}}$
\begin{align*}
\widehat{J_{3,2}} + \widehat{J_{3,3}} &\le \left\Vert (u_h-k)_+ - (u-k)_+ \right\Vert_{L^2([t_1,t_2] ; L^{2\theta'}(\R^d))} \left\Vert \int_{B_{r+\frac{\rho}{2}}} u^2(x) \Gamma^{K_s}(\tau,\tau)(x) \d x \right\Vert_{L^1([t_1,t_2])}^{1/2} \\
&\le c \rho^{-\alpha/2} \Vert u_h-u\Vert_{L^2([t_1,t_2];L^{2\theta'}(\R^d))}  \Vert u \Vert_{L^2([t_1,t_2] \times B_{r+\rho})}\\
&\to 0,
\end{align*}
where we used \eqref{cutoff}, \eqref{K1glob}, as well as \eqref{eq:steklovaux2} and that $u \in L^{2}([t_1,t_2];L^{2\theta'}(B_{r+\rho}))$.
This proves (ii).
\end{proof}

\begin{remark}
We point out that the above proof can be extended to more general test functions $\phi$ of the form $\phi = \pm\tau^2 g(u)$, where $g : [0,\infty) \to [0,\infty)$. This way, it would be possible to generalize the notion of a weak solution to \eqref{PDE}, or to \eqref{PDEdual}, in $I \times \Omega$, in the sense that the assumption $\partial_t u \in L^1_{loc}(I,L^2(\Omega))$, where $\partial_t u$ is the weak $L^2(\Omega)$-derivative of $u$, can be replaced by $u \in C(I;L^2(\Omega))$.
\end{remark}


\end{document}